\documentclass[a4paper,reqno]{amsart}

\pdfoutput=1

\usepackage{amscd,amsthm,amsmath,amssymb,mathtools}
\usepackage{verbatim}
\usepackage{color}
\usepackage{url}
\usepackage{enumerate,enumitem}
\usepackage{graphicx}
\usepackage{epstopdf,epsfig,subfigure}
\usepackage{curve2e}
\usepackage{array}
\usepackage{algorithm}
 \usepackage{algorithmic}
\usepackage[numbers,sort&compress]{natbib}
\usepackage{slashbox}

\usepackage{caption,graphicx,newfloat}
\DeclareCaptionType{InfoBox}

\usepackage[
    bookmarks=true,         
    unicode=false,          
    pdftoolbar=true,        
    pdfmenubar=true,        
    pdffitwindow=false,     
    pdfstartview={FitH},    
    pdftitle={Saturated feedback stabilizability to trajectories for the Schl\"ogl equation},    
    pdfauthor={Author},     
    pdfsubject={Subject},   
    pdfcreator={Creator},   
    pdfproducer={Producer}, 
    pdfkeywords={keyword1, key2, key3}, 
    pdfnewwindow=true,      
    colorlinks=true,       
    linkcolor=red,          
    citecolor=blue,        
    filecolor=magenta,      
    urlcolor=cyan,           
hypertexnames=false
]{hyperref}

\theoremstyle{plain}

\newtheorem{theorem}{Theorem}[section]

\newtheorem{lemma}[theorem]{Lemma}

\theoremstyle{definition}
\newtheorem{definition}[theorem]{Definition}
\newtheorem{remark}[theorem]{Remark}
\newtheorem{example}[theorem]{Example}

\numberwithin{equation}{section}

\usepackage{geometry}

\newcommand{\linspan}{\mathop{\rm span}\nolimits}

\newcommand{\rest}{\left.\kern-2\nulldelimiterspace\right|_}
\newcommand{\norm}[2]{\left|#1\right|_{#2}}
\newcommand{\dnorm}[2]{\left\|#1\right\|_{#2}}

\newcommand{\vol}{\mathop{\rm vol}\nolimits}

\newcommand{\Id}{{\mathbf1}}
\newcommand{\indf}{1}

\newcommand{\p}{\partial}

\newcommand*{\Bigcdot}{\raisebox{-.25ex}{\scalebox{1.25}{$\cdot$}}}

\newcommand{\clC}{{\mathcal C}}

\newcommand{\clH}{{\mathcal H}}

\newcommand{\clK}{{\mathcal K}}
\newcommand{\clL}{{\mathcal L}}

\newcommand{\clS}{{\mathcal S}}
\newcommand{\clT}{{\mathcal T}}
\newcommand{\clU}{{\mathcal U}}

\newcommand{\clZ}{{\mathcal Z}}

\newcommand{\bbN}{{\mathbb N}}

\newcommand{\bbR}{{\mathbb R}}

\newcommand{\bfK}{{\mathbf K}}

\newcommand{\fkP}{{\mathfrak P}}

\newcommand{\rmD}{{\mathrm D}}

\newcommand{\bfn}{{\mathbf n}}

\newcommand{\rma}{{\mathrm a}}

\newcommand{\rmd}{{\mathrm d}}
\newcommand{\rme}{{\mathrm e}}

\newcommand{\ovlineC}[1]{\overline C_{\left[#1\right]}}

\definecolor{DarkBlue}{rgb}{0,0.08,0.45}
\definecolor{DarkRed}{rgb}{.65,0,0}
\definecolor{applegreen}{rgb}{0.55, 0.71, 0.0}

\newcounter{mymac@matlab}
\newcommand{\matlab}{MATLAB%
   \ifnum\value{mymac@matlab}<1%
   \textregistered%
   \setcounter{mymac@matlab}{1}%
   \fi%
  }

\newcommand{\black}{ \color{black} }

\begin{document}
\title{Saturated feedback stabilizability to trajectories for the Schl\"{o}gl parabolic equation}
\author{Behzad Azmi$^{\tt1}$,\quad Karl Kunisch$^{\tt1,2}$ ,\quad S\'ergio S.~Rodrigues$^{\tt1,2}$ }
\thanks{
\vspace{-1em}\newline\noindent
\newline\noindent
$^{\tt1}$ Johann Radon Institute for Computational and Applied Mathematics,
  \"OAW, Altenbergerstr. 69, 4040 Linz, Austria.
\newline\noindent
$^{\tt2}$ Institute for Mathematics and Scientific Computing, 621, Heinrichstrasse 36, 8010 Graz, Austria.
\newline\noindent
{\sc Emails}:
{\small\tt behzad.azmi@ricam.oeaw.ac.at,\quad karl.kunisch@uni-graz.at,
\newline sergio.rodrigues@ricam.oeaw.ac.at}
 }

\begin{abstract}
It is shown that there exist a finite number of indicator functions, which allow us to track an arbitrary given trajectory of the Schl\"ogl model, by means of an explicit saturated feedback control input whose magnitude is bounded by a constant independent of the given targeted trajectory. Simulations are presented showing the stabilizing performance of the explicit feedback constrained control. Further, such performance is compared to that of a receding horizon constrained control minimizing the classical energy cost functional.
\end{abstract}

\maketitle

\pagestyle{myheadings} \thispagestyle{plain} \markboth{\sc B. Azmi, K. Kunisch,  S. S. Rodrigues}
{\sc Saturated feedback stabilizability to trajectories for the Schl\"ogl equation}


\section{Introduction}
\label{sec:introduction}
We investigate the controlled Schl\"{o}gl system, for time~$t\ge0$,
\begin{subequations}\label{sys-y-u-Cu}
 \begin{align}
& \tfrac{\p}{\p t} y -\nu\Delta y +(y-\zeta_1)(y-\zeta_2)(y-\zeta_3)=h+U_M^\diamond u,\label{sys-y-u}\\
&y(0)=y_0\in W^{1,2}(\Omega),\qquad \tfrac{\p}{\p\bfn}y\rest{\p\Omega}=0,\label{y-icH1}
\intertext{where~$U_M^\diamond u\coloneqq{\textstyle\sum\limits_{i=1}^{M_\sigma}u_i(t)\indf_{\omega^M_{i}}}$, under control constraints as}
&\dnorm{u(t)}{}\le C_u,\label{normu<Cu}
\end{align}
\end{subequations}
evolving in the Hilbert Sobolev space~$W^{1,2}(\Omega)$, where~$\Omega\subset\bbR^d$ is a bounded rectangular domain, with~$d\in\{1,2,3\}$. Further~$M$ and~$M_\sigma$  are positive integers, $U_M\coloneqq\{\indf_{\omega^M_{i}}\mid 1\le i\le M_\sigma\}\subseteq L^2(\Omega)$ is a given family of $M_\sigma$~actuators, which are indicator functions of open subdomains~$\omega^M_{i}\subseteq\Omega$ depending on the index~$M$. Finally, $\nu>0$, $(\zeta_1,\zeta_2,\zeta_3)\in\bbR^3$, ~$h\in L^2_{\rm loc}(\bbR_+,L^2(\Omega))$ is a given external force, and~$u=u(t)=(u_i(t),u_2(t),\dots,u_{M_\sigma}(t))$ is a vector of scalar controls (tuning parameters) at our disposal.
 We are particularly interested in the
case where the control~$u=u(t)$ is subject to constraints as~\eqref{normu<Cu} for an apriori given constant~$C_u\in[0,+\infty]$ and an apriori given norm~$\dnorm{\Bigcdot}{}$ in~$\bbR^{M_\sigma}$. The usual Euclidean norm in~$\bbR^{M_\sigma}$ shall be denoted by~$\norm{\Bigcdot}{\bbR^{M_\sigma}}$. Note that the extremal cases~$C_u=+\infty$ and~$C_u=0$ correspond, respectively, to the unconstrained case and to the  free dynamics case.

Our family of actuators~$U_M$ can be chosen so that the total volume covered by the actuators satisfies~$\vol(\bigcup_{i=1}^{M_\sigma}\omega^M_{i})=r\vol(\Omega)$, with an arbitrary apriori given~$r\in(0,1)$.

System~\eqref{sys-y-u-Cu} is a model for chemical reactions for non-
equilibrium phase transitions;  see~\cite[sect.~4]{Schloegl72} and~\cite{GugatTroeltzsch15}. Also,
when coupled with a suitable ordinary differential equation,
it gives rise to models in neurology and electrophysiology,
namely, to the FitzHugh–Nagumo-like equations ~\cite{NagumoAriYosh62,FitzHugh61,KunRod19-dcds}. It is also an interesting model from the mathematical point
of view. Indeed, the cubic nonlinearity with~$\zeta_1<\zeta_2<\zeta_3$
and vanishing~$(f,u)$, determine the two stable
equilibria~$\zeta_1$ and~$\zeta_3$, and the unstable equilibrium~$\zeta_2$. This
leads to interesting asymptotic behavior of the solutions
and, as we shall see, it lends itself to a nontrivial analysis of
the global saturated feedback control mechanism.

The main stabilizability problem under investigation is as follows. We are given a trajectory/solution~$\widehat y$ of the free dynamics, that is, we assume that~$\widehat y$ solves
\begin{subequations}\label{sys-haty}
 \begin{align}
& \tfrac{\p}{\p t} \widehat y -\nu\Delta +(\widehat y-\zeta_1)(\widehat y-\zeta_2)(\widehat y-\zeta_3)=h,\\
&\widehat y(0,\Bigcdot)=\widehat y_0\in W^{1,2}(\Omega),\qquad \tfrac{\p}{\p\bfn}\widehat y\rest{\p\Omega}=0,
\end{align}
\end{subequations}
 and that~$\widehat y$ has a desired behavior, which we would like to track.

Next, we are also given another initial state~$y_0\in W^{1,2}(\Omega)$. It turns out that the corresponding solution~$y$ of the free dynamics, with~$y(0,\Bigcdot)=y_0\in W^{1,2}(\Omega)$ may present an asymptotic behavior different from the targeted behavior of~$\widehat y$. For example, in the case~$h=0$, and~$\zeta_1<\zeta_2<\zeta_3$ we could think of the free dynamics equilibrium
$\widehat y(t,x)=\zeta_2$, with initial state~$\widehat y(0,x)=\widehat y_0(x)=\zeta_2$, as our desired targeted behavior. We can see that~$\widehat y(t,x)=\zeta_2$ is not a stable equilibrium, and that if~$y_0(x)\coloneqq c\ne\zeta_2$ is a constant initial state, then the state~$y(t,x)$, of the free dynamics solution, corresponding to the initial state~$y(0,x)=c$, does not converge to the targeted $\widehat y(t,x)$.

Hence, to track a desired trajectory~$\widehat y$ we (may) need to apply a control.
Our goal is  to design the control input~$u$ such that the state~$y(t,\Bigcdot)$ of the solution of system~\eqref{sys-y-u-Cu} converges exponentially to the targeted state~$\widehat y(t,\Bigcdot)$ as time increases,
\begin{equation}\label{goal-exp}
\norm{y(t)-\widehat y(t)}{L^2(\Omega)}^2\le \rme^{-\mu (t-s)}\norm{y(s)-\widehat y(s)}{L^2(\Omega)}^2,
\end{equation}
for all $t\ge s\ge0$,
for a suitable constant~$\mu>0$.

We shall construct the stabilizing constrained control~$u$ by saturating a suitable unconstrained stabilizing linear feedback control~$\clK (y-\widehat y)$, with~$\clK\colon W^{1,2}(\Omega)\to\bbR^{M_\sigma}$, through a radial projection as follows
\begin{subequations}\label{rad.proj}
\begin{align}
&u=\overline\clK(y-\widehat y)\coloneqq \fkP^{\dnorm{\Bigcdot}{}}_{C_u}(\clK (y-\widehat y)),
&\intertext{where}
&\fkP^{\dnorm{\Bigcdot}{}}_{C_u}(v)\coloneqq\begin{cases}v,&\mbox{ if }\dnorm{v}{}\le C_u,\\
\frac{C_u}{\dnorm{v}{}}v,&\mbox{ if }\dnorm{v}{}> C_u,
\end{cases}\qquad v\in\bbR^{M_\sigma}.
\end{align}
\end{subequations}
Note that we have, for~$v\ne0$
\begin{equation}\label{rad.proj.min}
\fkP^{\dnorm{\Bigcdot}{}}_{C_u}(0)=0\quad\mbox{and}\quad \fkP^{\dnorm{\Bigcdot}{}}_{C_u}(v)=\min\left\{1,\tfrac{C_u}{\dnorm{v}{}}\right\}v,
\end{equation}
and also that, for all $(v,C_u)\in\bbR^{M_\sigma}\times[0,+\infty]$,
\[\dnorm{\fkP^{\dnorm{\Bigcdot}{}}_{C_u}(v)}{}\le C_u,\qquad \fkP^{\dnorm{\Bigcdot}{}}_{0}(v)=0,\quad\mbox{and}\quad\fkP^{\dnorm{\Bigcdot}{}}_{+\infty}(v)=v.
\]
In particular, the saturated feedback control~$u(t)= \overline\clK(y(t)-\widehat y(t))$ satisfies~$\dnorm{u(t)}{}\le C_u$.

The stabilizability of dynamical systems as~\eqref{sys-y-u} is an important problem for applications, even in the case where the ``magnitude''~$\dnorm{u(t)}{}$ of the control is allowed to take arbitrary large values (i.e., in the case $C_u=+\infty$), as shown by the amount of contributions we can find in the literature.

In applications we may be faced with physical constraints, for example, with an upper bound for the magnitude of the acceleration/forcing provided by an engine, or with an upper bound for the temperature provided by a heat radiator. For this reason it is also important to investigate the case of bounded controls (i.e., the case $C_u<+\infty$).

\begin{remark}\label{R:convexOmega}
We consider rectangular spatial domains for the sake of simplicity of exposition. Analogous results can be obtained for more general convex polygonal domains. We shall revisit this point in Remark~\ref{R:PolyDom}.
\end{remark}

\begin{remark}\label{R:Msigma}
We consider a sequence~$(U_M)_{M\in\bbN_+}$ of families of indicator functions $U_M=\{\indf_{\omega^M_{i}}\mid 1\le i\le M_\sigma\}\subseteq L^2(\Omega)$ with supports~$\overline{\omega^M_{i}}$ depending on the sequence index~$M$. Such dependence on~$M$ is also convenient to be able to consider a sequence of families whose total volume covered by the actuators is fixed apriori, $\vol(\bigcup_{i=1}^{M_\sigma}\omega^M_{i})=r\vol(\Omega)$, with~$r\in(0,1)$ independent of~$M$.
\end{remark}

\subsection{Global exponential stabilizability to zero}
Considering the difference to the target,~$z=y-\widehat y$, our goal~\eqref{goal-exp} reads
 \begin{equation}\notag
\norm{z(t)}{L^2(\Omega)}^2\le \rme^{-\mu (t-s)}\norm{z(s)}{L^2(\Omega)}^2,\mbox{ for all } t\ge s\ge\tau\ge0.
\end{equation}
In this way we ``reduce'' the stabilizability to trajectories to the stabilizability to zero.

Let us consider a general controlled dynamical system, with state~$z$ and control~$u$,
\begin{subequations}\label{sys-gen}
\begin{align}
&\tfrac{\rmd}{\rmd t} z(t)=f(t,z(t),u(t)),\qquad z(0)=z_0\in\clZ,\qquad t\ge0,\\
&u(t)\in\bbR^{M_\sigma},\qquad \dnorm{u(t)}{\bbR^{M_\sigma}}\le C_u,\label{sys-gen-u}
\end{align}
\end{subequations}
evolving in a normed space~$\clZ$.
Consider also the free dynamics
\begin{equation}\label{sys-gen-free}
\tfrac{\rmd}{\rmd t} {z(t)}=f(t,z(t),0),\qquad z(0)=z_0\in\clZ,\qquad t\ge0.
\end{equation}

Let~$\clH$ be another normed space with~$\clZ\subseteq\clH$.

\begin{definition}
System~\eqref{sys-gen-free} is globally exponentially stable in the~$\clH$-norm, if there are constants~$\varrho\ge1$ and~$\mu>0$ such that for every initial condition~$z_0\in\clZ$, we have that
$\norm{z(t)}{\clH}\le \varrho\rme^{-\mu (t-s)}\norm{z(s)}{\clH}$, for all $t\ge s\ge0$.
\end{definition}

\begin{definition}
System~\eqref{sys-gen} is globally exponentially stabilizable in the~$\clH$-norm, if there are constants~$\varrho\ge1$ and~$\mu>0$ such that for every initial condition~$z_0\in\clZ$, there
exists $u\in L^2(\bbR_+,\bbR^{M_\sigma})$ satisfying~\eqref{sys-gen-u} such that
$\norm{z(t)}{\clH}\le \varrho\rme^{-\mu (t-s)}\norm{z(s)}{\clH}$, for all $t\ge s\ge0$.
\end{definition}

The class of systems globally stabilizable with constrained controls, $C_u<+\infty$, is strictly smaller  than that of systems  globally stabilizable with unconstrained controls, $ C_u=+\infty$. This can be illustrated with the following system, where~$r$ is a constant, $U_M=\{1\}$, $z(t)\in\bbR$, and our control input is~$u(t)=u_1(t)\in\bbR$, $t\ge0$,
 \begin{align}\label{sys-y-u-ode1}
 \tfrac{\rmd}{\rmd t} z +r z=u_11, \qquad \norm{u}{\bbR}\le C_u.
\end{align}

 \begin{theorem}\label{T:ode1-intro-Cuinfty}\textit{
If~$ C_u=+\infty$, then for arbitrary $r\in\bbR$, system~\eqref{sys-y-u-ode1} is globally exponentially stabilizable.
}
\end{theorem}

 \begin{theorem}\label{T:ode1-intro}\textit{
If $r<0$, then for an arbitrary given $C_u\in\bbR_+$,
system~\eqref{sys-y-u-ode1} is not globally exponentially stabilizable.
}
\end{theorem}

The proofs  of Theorems~\ref{T:ode1-intro-Cuinfty} and~\ref{T:ode1-intro} are given in the Appendix.
Theorem~\ref{T:ode1-intro} shows that if the free dynamics of system~\eqref{sys-y-u-ode1} is unstable and if we are given a positive bound~$C_u<+\infty$ for the magnitude of the control, then
we cannot globally exponentially stabilize the system. Therefore, we will need to explore some properties of the free dynamics of~\eqref{sys-y-u-Cu} in order to conclude its global stabilizability to trajectories,  with a prescribed arbitrary exponential decrease rate~$\mu>0$.

\subsection{The sequence of families of actuators}\label{sS:actuators}
For our rectangular spatial domain~$\Omega\subset\bbR^d$,
\begin{equation}\notag
\Omega=\Omega^\times=(0,L_1)\times (0,L_2)\times \cdots\times (0,L_d),\quad d\in\{1,2,3\},
\end{equation}
we consider the set~$U_M$ of actuators as in
~\cite[sect.~4.8]{KunRod19-cocv} and~\cite[sect.~5]{KunRodWalter21},
\begin{subequations}\label{U_M}
\begin{align}
&U_M\coloneqq\{\indf_{\omega_{j}}\mid 1\le j\le M_\sigma\}\subseteq L^2(\Omega),\\&\clU_M\coloneqq\linspan U_M,\qquad \dim\clU_M=M_\sigma,
\intertext{where, for a fixed~$M$, $M_\sigma=M^d$ and}
 &\omega_j=\omega_j^M\coloneqq{\bigtimes\limits_{n=1}^d}((c_n)_{j}^M-\tfrac{rL_n}{2M},(c_n)_{j}^M+\tfrac{rL_n}{2M}),
\intertext{with set of centers $c=(c)^M_j$}
 &\{(c)^M_j\mid 1\le j\le M_\sigma\}=\!{\bigtimes\limits_{n=1}^d}\!\{\tfrac{(2k-1)L_n}{2M}\mid 1\le k\le M\}.
\end{align}
\end{subequations}
See Fig.~\ref{fig.suppActSens} for an illustration for the case~$d=2$. See also
~\cite[sect.~5.2]{Rod20-eect},
~\cite[sect.~4]{Rod21-sicon}, ~\cite[sect.~6]{Rod21-aut} where an
analogous placement of the actuators/sensors have been used.
%
\definecolor{lightgray}{gray}{0.75}
\setlength{\unitlength}{.0019\textwidth}
\newsavebox{\Rectfw}%
\savebox{\Rectfw}(0,0){%
\linethickness{2pt}
{\color{black}\polygon(0,0)(120,0)(120,80)(0,80)(0,0)}%
}%
\newsavebox{\RectRef}%
\savebox{\RectRef}(0,0){%
\linethickness{1.5pt}
{\color{lightgray}\polygon*(40,30)(80,30)(80,50)(40,50)(40,30)}%
}%
 \begin{figure}[h!]
\begin{center}
\begin{picture}(500,100)
 \put(0,0){\usebox{\Rectfw}}
\put(0,0){\usebox{\RectRef}}
  \put(190,0){\usebox{\Rectfw}}
   \put(190,0){\scalebox{.5}[.5]{\usebox{\RectRef}}}
    \put(250,0){\scalebox{.5}[.5]{\usebox{\RectRef}}}
   \put(190,40){\scalebox{.5}[.5]{\usebox{\RectRef}}}
  \put(250,40){\scalebox{.5}[.5]{\usebox{\RectRef}}}
 \put(380,0){\usebox{\Rectfw}}
  \put(380,0){\scalebox{.333333}[.333333]{\usebox{\RectRef}}}
 \put(420,0){\scalebox{.333333}[.333333]{\usebox{\RectRef}}}
 \put(460,0){\scalebox{.333333}[.333333]{\usebox{\RectRef}}}
  \put(380,26.66666){\scalebox{.333333}[.333333]{\usebox{\RectRef}}}
 \put(420,26.66666){\scalebox{.333333}[.333333]{\usebox{\RectRef}}}
 \put(460,26.66666){\scalebox{.333333}[.333333]{\usebox{\RectRef}}}
  \put(380,53.33333){\scalebox{.333333}[.333333]{\usebox{\RectRef}}}
 \put(420,53.33333){\scalebox{.333333}[.333333]{\usebox{\RectRef}}}
 \put(460,53.33333){\scalebox{.333333}[.333333]{\usebox{\RectRef}}}

 \put(40,85){$M=1$}
 \put(230,85){$M=2$}
 \put(420,85){$M=3$}
\put(60,35){$\omega_1^1$}
\put(222,17){$\omega_1^2$}
\put(282,17){$\omega_2^2$}
\put(222,57){$\omega_3^2$}
\put(282,57){$\omega_4^2$}
\end{picture}
\end{center}
 \caption{Supports of actuators in the rectangle~$\Omega^\times\subset\bbR^2$.} \label{fig.suppActSens}
 \end{figure}
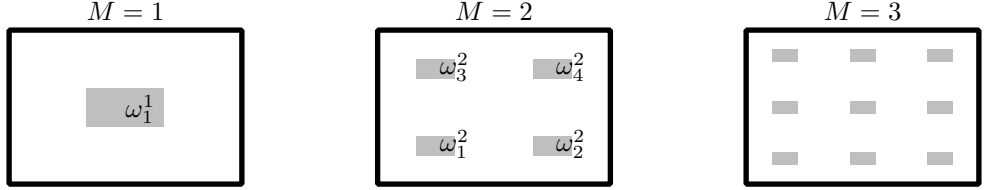

\subsection{The main stabilizability result and the RHC framework}
For the (ordered) family~$U_M$ of linearly independent actuators in~\eqref{U_M},
let~$P_{\clU_M}\in\clL(L^2(\Omega),\clU_M)$ be the orthogonal projection in~$L^2(\Omega)$ onto~$\clU_M$. Recall the control operator isomorphism in~\eqref{sys-y-u-Cu},
\(
U_M^\diamond\colon\bbR^{M_\sigma}\to\clU_M,\quad u\mapsto{\textstyle\sum\limits_{i=1}^{M_\sigma}}u_i\indf_{\omega^M_{i}},
\)
and consider the unconstrained explicit feedback control
\[
z\mapsto-\lambda P_{\clU_M} z,\quad\mbox{for a given}\quad\lambda\ge 0,
\]
 where~$z=y-\widehat y$ is the difference to the target~$\widehat y$.
Finally, we consider the saturated feedback control
\[
\overline\clK_M(z)\coloneqq  \fkP^{\dnorm{\Bigcdot}{}}_{C_u}\left(-\lambda (U_{M}^\diamond)^{-1}P_{\clU_M}z\right).
 \]
 In this way, the difference~$z$ will satisfy the system
\begin{subequations}\label{sys-z-intro}
\begin{align}
 &\tfrac{\p}{\p t} z =\nu\Delta z -f^{\widehat y}(z)+U_{M}^\diamond\fkP^{\dnorm{\Bigcdot}{}}_{C_u}\left(-\lambda (U_{M}^\diamond)^{-1}P_{\clU_M}z\right),\\
&z(0,\Bigcdot)=z_0,\quad \tfrac{\p}{\p\bfn} z\rest{\p\Omega}=0,\quad \dnorm{u(t)}{}\le C_u,
 \end{align}
\end{subequations}
with $z_0\in W^{1,2}(\Omega)$ and $f^{\widehat y}(z)=z^3+(3\widehat y+\xi_2) z^2+(3\widehat y^2+2\xi_2\widehat y+\xi_1) z$, for suitable constants~$\xi_1$ and~$\xi_2$.

Shortly, the main stabilizability result of this paper is as follows, whose  precise statement shall be given in Theorem~\ref{T:mainSchloegl}.

\smallskip
\noindent
{\bf Main Result.} 
\textit{For each~$\mu>0$ there exist large enough constants~$M\in\bbN_+$, $\lambda>0$, and ~$C_u\in\bbR_+$,  such that~\eqref{sys-z-intro} is globally  exponential stable, with exponential decrease rate~$\mu$.}

\smallskip
We also consider the infinite-horizon constrained optimal control problem
\begin{subequations}\label{inf-hor-por}
\begin{align}
&\min_{u \in L^2(\bbR_+,\bbR^{M_\sigma})}J_ {\infty}{(u;y_0,\widehat y)}\mbox{ subject to~\eqref{sys-y-u-Cu}},\label{inf-hor-por-subto}\\
&J_ {\infty}{(u;y_0,\widehat y)}\coloneqq\norm{y-\widehat y}{L^2(\bbR_+,{L^2}
)}^2+\beta\norm{u}{L^2(\bbR_+,\bbR^{M_\sigma})}^2,\label{inf-hor-por-minJ}
\end{align}
\end{subequations}
where $\beta>0$  and  the target trajectory $\widehat y$ is given as the solution to~\eqref{sys-haty} for a pair $(\widehat y_0,h)$.  This problem is an infinite-horizon nonlinear time-varying optimal control problem with control constraints.  One efficient  approach  to deal with~\eqref{inf-hor-por} is  the receding horizon control (RHC).  In this approach,  the stabilizing control is obtained by concatenating  a sequence of finite-horizon open-loop controls.   These controls are computed online as the solutions to problems of the following form, for time~$t\in I_{t_0}^T\coloneqq(t_0,t_0+T)$, $T>0$. With
\begin{subequations}
\label{finite-OP-subto}
\begin{align}
\tfrac{\p}{\p t} y -\nu\Delta y +(y-\zeta_i)(y-\zeta_2)(y-\zeta_3)=h+U_M^\diamond u,\\
y(t_0)=\bar{y}_0,\qquad \tfrac{\p}{\p\bfn}y\rest{\p\Omega}=0,\qquad
\dnorm{u(t)}{}\le C_u, 
\end{align}
\end{subequations}
we consider the problem
\begin{subequations}\label{finite-OP}
\begin{align}
\label{finite-OP-minJ}
&\min_{u\in L^2(I_{t_0}^T,\bbR^{M_\sigma})} J_{T}(u;t_0,\bar{y}_0, \widehat{y})
\mbox{ subject to~\eqref{finite-OP-subto}},\\
&J_{T}(u;t_0,\bar{y}_0, \widehat{y})\!\coloneqq\!\norm{y-\widehat y}{L^2(I_{t_0}^T,{L^2}
)}^2\!+\!\beta\norm{u}{L^2(I_{t_0}^T,\bbR^{M_\sigma})}^2\!.\!\!
\end{align}
\end{subequations}
 The receding horizon framework is detailed by the steps of Algorithm~\ref{RHA}.
\begin{algorithm}[htbp]
\caption{RHC($\delta,T$)}\label{RHA}
\begin{algorithmic}[1]
\REQUIRE{sampling time  $\delta>0$, prediction horizon $T> \delta$,  initial state $y_0 \in W^{1,2}(\Omega)$, targeted trajectory $\widehat{y}$ solving~\eqref{sys-haty}}
\ENSURE{ Receding horizon control  ~$u_{rh} \in L^2(\mathbb{R}_+,  \mathbb{R}^{M_{\sigma}})$.}
\STATE Set~$t_0=0$ and~$\bar{y}_0 =y_0$;
\STATE Find~$(y_T^*(\cdot;t_0,\bar{y}_0),u^*_T(\cdot;t_0,\bar{y}_0))$
for time in~$(t_0,t_0+T)$ by solving the  open-loop problem~\eqref{finite-OP};
\STATE For all $\tau \in [t_0,t_0+\delta)$,  set $u_{rh}(\tau)= u_T^*(\tau;t_0,\bar{y}_0)$;
\STATE Update: $\bar{y}_0\leftarrow y_T^*(t_0+\delta;t_0,\bar{y}_0),$;
\STATE Update: $t_0 \leftarrow t_0 +\delta$;
\STATE Go to step $2$;
\end{algorithmic}
\end{algorithm}
The obtained RHC law is not optimal, as long as $T$ is  finite.  But,  relying on the stabilizability stated in Main Result,  we will show (Thm.~\ref{subopth}) that it is stabilizing and suboptimal.  The control constraints  are enforced  within the finite-horizon open-loop problems.

\subsection{On previous related works in literature}
The literature is rich in results concerning the feedback stabilizability of parabolic like equations under no constraints in the magnitude of the control.
For example,  we can mention~\cite{BarbuTri04,BarRodShi11,KunRod19-cocv,PhanRod18-mcss,BreKunRod17,AzouaniTiti14,LunasinTiti17}, \cite[sect.~2.2]{Barbu11}, and references therein.
Though we do not address,
in the present manuscript, the case of boundary controls, we would like to mention~\cite{PrieurTrelat19,Deutscher16,BarbuLasTri06,Barbu12,BadTakah11,Raymond19,Rod18,Barbu_TAC13,BalKrstic00,CochranVazquezKrstic06,KrsticMagnVazq09}.
When compared to the unconstrained case, there is (it seems) a smaller amount of works in the literature considering an upper bound~$C_u$ for the magnitude
of the control~$u(t)$. For finite-dimensional systems the literature is still rich, as examples we refer the reader to~\cite{HuLinQiu01,CorradiniCristofaroOrlando10,ZhouLam17,LiuChitourSontag96,Teel92,BarbuChengFreeman97,LauvdalFossen97,SaberiLinTeel96,SussmannSontagYang94,WredenhagenBelanger94}.
For infinite-dimensional systems, the  amount of works in the literature is more modest, for parabolic equations we mention~\cite{MironchenkoPrieurWirth21}, and  for wave-like equations we refer the reader to~\cite{LasieckaSeidman03,SeidmanLi01}. See also~\cite{Slemrod89} with an application to the  beam equation in~\cite[sect.~8.1]{Slemrod89}.

We follow an approach which is common in many works dealing with bounded controls. Namely, we consider the saturation of a given unconstrained stabilizing  feedback~$u(t)=\clK(z(t))$, with~$z=y-\widehat y$. This means that, at every instant of time~$t\ge0$, we simply rescale the given unconstrained feedback~$u(t)$ if its magnitude violates the constraint.
Note that the norm of the (unconstrained) feedback control~$u(t)=\clK(z(t))\in\bbR^{M_\sigma}$ can take arbitrary large values (e.g., for linear~$\clK$ and $\gamma>0$ we have~$\dnorm{\clK(\gamma  z)}{}=\gamma\dnorm{\clK(z)}{}$).
Exponentially stabilizing controls given in linear feedback form~$u=\clK z$ are often demanded in applications, because such controls are able to respond to small measurement errors; see the numerical simulations in~\cite{Rod21-aut,KunRod19-cocv}.

In this work,  we also continue the investigation on  the receding horizon framework initiated in \cite{AzmiKun2019} for the stabilization (to zero) of {\em linear} nonautonomous (time-varying)  systems.  In this framework,  no terminal cost or constraints is needed and,  the stability is obtained by an appropriate concatenation scheme on a sequence of overlapping temporal intervals.   Recall that, \emph{in theory}, stabilizing system ~\eqref{sys-y-u-Cu} to a given time-dependent trajectory~$\widehat y=\widehat y(t)$ is equivalent to stabilizing the nonautonomous error dynamics~\eqref{sys-z-intro} to zero. We adapt the analysis given in \cite{AzmiKun2019}  for~\eqref{inf-hor-por},  with the differences that here, firstly, the dynamics is {\em nonlinear}, secondly, control constraints are imposed and, finally, numerically, instead of stabilizing~\eqref{sys-z-intro} to zero we stabilize the original system~\eqref{sys-y-u-Cu} to the trajectory~$\widehat y$.

\subsection{Contents and notation}
The manuscript is organized as follows. Section~\ref{S:Schloegl} is dedicated to the proof of Main Result. In section~\ref{S:optimalcontrol} we discuss the receding horizon algorithm including the existence of optimal controls for the finite-horizon subproblems.  The results of numerical simulations showing the stabilizing performance of both the explicit saturated feedback and the receding horizon control are discussed in section~\ref{S:simulations}. Finally, the Appendix gathers the proofs of Theorems~\ref{T:ode1-intro-Cuinfty} and~\ref{T:ode1-intro}.

Concerning the notation, we write~$\bbR$ and~$\bbN$ for the sets of real numbers and nonnegative
integers, respectively. We set $\bbR_+\coloneqq(0,+\infty)$
and~$\bbN_+\coloneqq\bbN\setminus\{0\}$.

Given Banach spaces~$X$ and~$Y$, we write $X\xhookrightarrow{} Y$ if the inclusion
$X\subseteq Y$ is continuous.
The space of continuous linear mappings from~$X$ into~$Y$ is denoted by~$\clL(X,Y)$. We
write~$\clL(X)\coloneqq\clL(X,X)$.
The continuous dual of~$X$ is denoted~$X'\coloneqq\clL(X,\bbR)$.
The adjoint of an operator $L\in\clL(X,Y)$ will be denoted $L^*\in\clL(Y',X')$.
The space of continuous functions from~$X$ into~$Y$ is denoted by~$\clC(X,Y)$.

The orthogonal complement to a given subset~$B\subset H$ of a Hilbert space~$H$,
with scalar product~$(\Bigcdot,\Bigcdot)_H$,  is
denoted~$B^{\perp H}\coloneqq\{h\in H\mid (h,s)_H=0\mbox{ for all }s\in B\}$.

Given two closed subspaces~$F\subseteq H$ and~$G\subseteq H$ of the
Hilbert space~$H=F+ G$, with~$F\bigcap G=\{0\}$, we denote by~$P_F^G\in\clL(H,F)$
the oblique projection in~$H$ onto~$F$ along~$G$. That is, writing $h\in H$ as $h=h_F+h_G$
with~$(h_F,h_G)\in F\times G$, we have~$P_F^Gh\coloneqq h_F$.
The orthogonal projection in~$H$ onto~$F$ is denoted by~$P_F\in\clL(H,F)$.
Notice that~$P_F= P_F^{F^{\perp H}}$.

By
$\overline C_{\left[a_1,\dots,a_n\right]}$ we denote a nonnegative function that
increases in each of its nonnegative arguments~$a_i$, $1\le i\le n$.

Finally, $C,\,C_i$, $i=0,\,1,\,\dots$, stand for unessential positive constants.

 \section{Exponential stabilizability}\label{S:Schloegl}
We fix the data in the Schl\"ogl system~\eqref{sys-y-u}  as
\begin{equation}\label{data}
\nu>0,\;\; h\in L^2_{\rm loc}(\bbR_+,L^2(\Omega)),\mbox{ and } ( \zeta_1, \zeta_2, \zeta_3)\in\bbR^3.
\end{equation}
We recall the free dynamics
\begin{subequations}\label{Schloegl}
\begin{align}
 &\tfrac{\p}{\p t} y -\nu\Delta y+(y- \zeta_1)(y- \zeta_2)(y- \zeta_3)= h, \\
 &y(0)=y_0\in W^{1,2}(\Omega),\qquad\tfrac{\p}{\p\bfn}y\rest{\p\Omega}=0,
 \end{align}
\end{subequations}
whose solution,  with initial state~$y(0)=y_0$, will be denoted by~$\clS(y_0;t)\coloneqq y(t)$. We show here that a saturated control allows us to track arbitrary solutions of the free-dynamics.
Let us assume that the trajectory~$\widehat y(t)=\clS(\widehat y_0;t)$ has a desired behavior. Our goal is to construct a control which stabilizes the system to this trajectory.
We consider the system
\begin{subequations}\label{Schloegl-feed}
\begin{align}
 &\tfrac{\p}{\p t} y -\nu\Delta y+(y- \zeta_1)(y- \zeta_2)(y- \zeta_3)\notag\\
&\hspace{1em}= h+U_M^\diamond\overline\clK_M(y-\widehat y),\\
&y(0)=y_0\in W^{1,2}(\Omega),\qquad\tfrac{\p}{\p\bfn}\rest{\p\Omega}=0,
\intertext{with the saturated feedback control (cf.~\eqref{sys-z-intro})}
 &\overline\clK_M(y-\widehat y)=  \fkP^{\dnorm{\Bigcdot}{}}_{C_u}\left(-\lambda (U_{M}^\diamond)^{-1}P_{\clU_M}(y-\widehat y)\right).  \label{Behcont}
 \end{align}
\end{subequations}
\black
Let us denote the solution of~\eqref{Schloegl-feed} by~$y(t)\coloneqq\clS_{\rm feed}^{\widehat y}(y_0;t)$.

\begin{theorem}\label{T:mainSchloegl}\textit{
For arbitrary~$\mu>0$, there exists $M_*\in\bbN_+$ such that, for every~$M\ge M_*$ there exists~$\lambda_*>0$ such that, for every~$\lambda>\lambda_*$ there exists~$C_u^*\in\bbR_+$  such that, for all~$C_u>C_u^*$ it holds that: for each~$(\widehat y_0,y_0)\in W^{1,2}(\Omega)\times W^{1,2}(\Omega)$, the solutions $\widehat y(t)\coloneqq\clS(\widehat y_0;t)$ of~\eqref{Schloegl} and $y(t)\coloneqq\clS_{\rm feed}^{\widehat y}(y_0;t)$ of~\eqref{Schloegl-feed} satisfy, for all~$t\ge s\ge0$,
\begin{equation}
\norm{y(t)-\widehat y(t)}{L^2(\Omega)}\le \rme^{-\mu (t-s)}\norm{y(s)-\widehat y(s)}{L^2(\Omega)}.
\label{Tmain.exp}
\end{equation}
Furthermore, $M_*\le\ovlineC{\mu,\norm{\zeta}{\infty}}$, where~$\norm{\zeta}{\infty}\coloneqq\max\limits_{1\le i\le 3}\norm{\zeta_i}{\bbR}$.
}
\end{theorem}

The proof of Theorem~\ref{T:mainSchloegl} is given in
section~\ref{sS:proofT:mainSchloegl}.
Observe  that Theorem~\ref{T:mainSchloegl} states that every trajectory~$\widehat y$ of the free-dynamics system~\eqref{Schloegl} can be tracked exponentially fast.
Note also that~$M\ge M_*$ does not depend on the pair~$(\widehat y_0,y_0)$ of initial states, which means that the number~$M_\sigma$ of actuators can be chosen independently of both the targeted trajectory~$\widehat y$ and of the initial error~$y_0-\widehat y_0$.

For simplicity, we shall often denote~$L^2\coloneqq L^2(\Omega)$ endowed with the usual scalar product, and we shall denote~$V=W^{1,2}(\Omega)$ endowed with the scalar product
\begin{equation}\label{Vscalarprod}
(w,z)_V\coloneqq \nu(\nabla w,\nabla z)_{(L^2)^d}+(w,z)_{L^2}.
\end{equation}
We also write, for
more general Lebesgue and Sobolev spaces,
\[
L^p\coloneqq L^p(\Omega),\quad W^{s,p}\coloneqq W^{s,p}(\Omega),\qquad p\ge1,\quad s\ge0.
\]

\subsection{On the well-posedness of strong solutions}\label{sS:wellPose}
The free dynamics~\eqref{Schloegl} is the particular case of~\eqref{Schloegl-feed} when we take~$\lambda=0$.
When referring to ``the solution of system~\eqref{Schloegl-feed}'' we mean the strong solution~$y\in W_{\rm loc}(\bbR_+,W^{2,2},L^2)$, with
\begin{align}
W_{\rm loc}(\bbR_+,X,Y)&\coloneqq\{y\mid y\in W((0,T),X,Y),
\mbox{ } T>0\},\notag\\
W((0,T),X,Y)&\coloneqq\{y\in L^2((0,T),X)\!\mid \dot y\in L^2((0,T),Y)\}.\notag
\end{align}
whose existence and uniqueness can be derived as a weak limit of the solutions of finite-dimensional Galerkin approximations and suitable apriori ``energy'' estimates. We skip the details here. We just mention the following apriori-like  ``energy'' estimates, with~$Ay\coloneqq-\nu\Delta y+y$. For simplicity let us denote
\begin{align}
\overline h_\lambda(t)&\coloneqq U_M^\diamond\overline\clK_M(y(t)-\widehat y(t))\notag\\
& =U_M^\diamond\fkP^{\dnorm{\Bigcdot}{}}_{C_u}(-\lambda (U_{M}^\diamond)^{-1}P_{\clU_M}(y(t)-\widehat y(t))),\quad\lambda\ge0,\notag
\end{align}
from which we find
\begin{align}
\norm{\overline h_\lambda(t)}{L^2}&\le \lambda\norm{P_{\clU_M}(y(t)-\widehat y(t))}{L^2}\le \lambda\norm{y(t)-\widehat y(t)}{L^2}.\notag
\end{align}

Multiplying the dynamics by~$2A y$, we obtain
\begin{align}
\tfrac{\rmd}{\rmd t}\norm{y}{V}^2&\le-2\norm{A y}{L^2}^2+2(y,Ay)_{L^2}\notag\\
&\quad-2((y-\zeta_1)(y-\zeta_2)(y-\zeta_3)+h+\overline h_\lambda,Ay)_{L^2}\notag\\
&\le-\tfrac43\norm{A y}{L^2}^2+3\norm{h}{L^2}^2+3\lambda^2\norm{y-\widehat y}{L^2}^2\notag\\
&\quad
+2(-y^3+\xi_2 y^2+(\xi_1+1) y +\xi_0,Ay)_{L^2},\notag
\end{align}
where
\begin{equation}\label{xis}
(\xi_2,\xi_1,\xi_0)\coloneqq (\zeta_1+ \zeta_2+ \zeta_3, - \zeta_1 \zeta_2- \zeta_1 \zeta_3- \zeta_2 \zeta_3, \zeta_1 \zeta_2 \zeta_3),
\end{equation}
which gives us
\begin{align}
\tfrac{\rmd}{\rmd t}\norm{y}{V}^2
&\le-\norm{A y}{L^2}^2+3\norm{\xi_0}{L^2}^2+3\norm{h}{L^2}^2
+3\lambda^2\norm{y-\widehat y}{L^2}^2\notag\\
&\quad
+2(-3y^2+2\xi_2 y ,\nu\norm{\nabla y}{\bbR^d}^2)_{L^2}\notag\\
&\quad
+2(-y^2+\xi_2 y+(\xi_1+1) ,y^2)_{L^2}\notag\\
&\le-\norm{A y}{L^2}^2+3\norm{\xi_0}{L^2}^2+3\norm{h}{L^2}^2\notag\\
&\quad
 +3\lambda^2\norm{y-\widehat y}{L^2}^2
+2C\norm{y}{V}^2,\label{dty-exist}
\end{align}
where~$C\coloneqq\max\{C_0,C_1\}$, with~$C_1\coloneqq\max\limits_{s\in\bbR}(-3s^2+2\xi_2 s)$ and~$C_0\coloneqq\max\limits_{s\in\bbR}(-s^2+\xi_2 s+\xi_1+1)$.

\subsubsection*{The free dynamics}
From~\eqref{dty-exist} we obtain, for~$\lambda=0$,
\begin{align}
\tfrac{\rmd}{\rmd t}\norm{y}{V}^2
&\le-\norm{A y}{L^2}^2+3\norm{\xi_0}{L^2}^2+3\norm{h}{L^2}^2
 +2C\norm{y}{V}^2.\notag
\end{align}
Then, the Gronwall inequality and time integration give us, for an arbitrary~$T>0$,
$
y\in L^\infty((0,T),V){\textstyle\bigcap}L^2((0,T),\rmD(A)).
$
It turns out that the (graph) norm of the domain~$\rmD(A)$ of~$A$ is equivalent to the usual norm in~$W^{2,2}(\Omega)$. Thus, to show that~$y\in W((0,T),W^{2,2},L^2)$, it remains to show that
$\dot y\in L^2((0,T),L^2)$. For this purpose, we  observe that the nonlinearity~$f(y)\coloneqq(y-\zeta_1)(y-\zeta_2)(y-\zeta_3)$ satisfies
\begin{align}
&\norm{f(y)}{L^2}^2=\!\norm{y^3-\xi_2 y^2-(\xi_1+1) y -\xi_0}{L^2}^2\!\le C_1\!\left(\norm{y}{L^6}^6+1\right)\notag
\end{align}
for a suitable constant~$C_1>0$. From~$V\xhookrightarrow{}L^6$, we can derive that
$f(y)\in L^\infty((0,T),L^2)$. Then, from the dynamics equation in~\eqref{Schloegl}, it follows that
$\dot y\in L^2((0,T),L^2)$.
\subsubsection*{The controlled dynamics} By assumption the targeted state~$\widehat y$ satisfies the free dynamics. In particular
$
\widehat y\in L^2((0,T),\rmD(A))\subseteq L^2((0,T),L^2).
$
Now, if~$\lambda>0$ and if~$y$ satisfies the corresponding controlled dynamics with targeted trajectory~$\widehat y$, by~\eqref{dty-exist} we find that
\begin{align}
\tfrac{\rmd}{\rmd t}\norm{y}{V}^2
&\le-\norm{A y}{L^2}^2+3\norm{\xi_0}{L^2}^2+3\norm{h}{L^2}^2
 +6\lambda^2\norm{\widehat y}{L^2}^2\notag\\
&\quad+6\lambda^2\norm{y}{L^2}^2
+2C\norm{y}{V}^2\notag
\end{align}
and, since~$V\xhookrightarrow{} L^2$, it follows that
\begin{align}
\tfrac{\rmd}{\rmd t}\norm{y}{V}^2
&\le-\norm{A y}{L^2}^2+3\norm{\xi_0}{L^2}^2+3\norm{h}{L^2}^2
 +6\lambda^2\norm{\widehat y}{L^2}^2\notag\\
&\quad
+(2C+6\lambda^2C_1)\norm{y}{V}^2,\notag
\end{align}
 for a suitable constant~$C_1>0$. We can argue as in the free dynamics case to conclude that~$y\in W((0,T),W^{2,2},L^2)$.

\subsection{The dynamics of the error}
For the dynamics of the error~$z\coloneqq y-\widehat y$, that is, of the difference between the solution
$y(t)$ of system~\eqref{Schloegl-feed} and the targeted solution~$\widehat y(t)$ of the free dynamics~\eqref{Schloegl}, we find
\begin{subequations}\label{Schloegl-diff}
\begin{align}
 &\tfrac{\p}{\p t} z -\nu\Delta z+f(y)-f(\widehat y)= U_M^\diamond\overline\clK_M(z)\\
 &z(0)=z_0\coloneqq y_0-\widehat y_0,\qquad\tfrac{\p}{\p\bfn} z\rest{\p\Omega}=0,
\intertext{for given data as in~\eqref{data}, and with}
&
f(w)\coloneqq (w- \zeta_1)(w- \zeta_2)(w- \zeta_3).
\end{align}
\end{subequations}
We start by observing that~$f(w)=w^3 +\xi_2 w^2+\xi_1 w+\xi_0$
with the~$\xi_j$s as in~\eqref{xis},
which leads us to
\begin{align}
f(z+\widehat y)&=z^3 +3\widehat y z^2 +3\widehat y^2 z+\widehat y^3
+\xi_2( z^2+2\widehat y z +\widehat y^2)\notag\\
&\quad + \xi_1(z+\widehat y)+\xi_0\notag\\
&=z^3 +(3\widehat y+\xi_2) z^2+(3\widehat y^2+2\xi_2\widehat y+\xi_1) z+f(\widehat y),\notag
\end{align}
and, thus,
\begin{align}
 \tfrac{\p}{\p t} z &=\nu\Delta z -z^3 -\widehat f(z)+U_M^\diamond\overline\clK_M(z),\notag\\
\widehat f(z)&\coloneqq(3\widehat y+\xi_2) z^2+(3\widehat y^2+2\xi_2\widehat y+\xi_1) z.\notag
 \end{align}

By multiplying the difference dynamics by~$2z$, we find
\begin{align}
 \tfrac{\rmd}{\rmd t} \norm{z}{L^2}^2&=-2\norm{z}{V} ^2+2(z -z^3,z)_{L^2} -2(\widehat f(z),z)_{L^2}\notag\\
&\quad +2(U_M^\diamond\overline\clK_M(z),z)_{L^2}\notag\\
&\hspace{0em}=-2\norm{z}{V} ^2+2\norm{z}{L^2}^2 -2( z^3+\widehat f(z),z)_{L^2}\notag\\
&\quad +2(U_M^\diamond\overline\clK_M(z),z)_{L^2}.\label{dtz1}
\end{align}

Observe that
\begin{subequations}\label{dtz1-aux}
\begin{align}
&-2( z^3+\black\widehat f(z),z)_{L^2}\notag\\
&\qquad =-2\left( z^2+3\widehat yz+\xi_2 z+3\widehat y^2+2\xi_2\widehat y+\xi_1,z^2\right)_{L^2}\notag
\end{align}
and, by writing
\begin{align}
 z^2+3\widehat yz+3\widehat y^2&=(\tfrac{15}{16})^2 z^2+3\widehat yz+(\tfrac{16}{10})^2\widehat y^2\notag\\
&\quad 
+(1-(\tfrac{15}{16})^2) z^2+(3-(\tfrac{16}{10})^2) \widehat y^2\notag
\\
&=(\tfrac{15}{16} z+\tfrac{16}{10}\widehat y)^2
+\tfrac{31}{256} z^2+\tfrac{44}{100} \widehat y^2\notag
\end{align}
 we arrive at
\begin{align}
&-2( z^3+\widehat f(z),z)_{L^2}+2\left((\tfrac{15}{16} z+\tfrac{16}{10}\widehat y)^2,z^2\right)_{L^2}\notag\\
&\le-2\left(\tfrac{31}{256} z^2+\tfrac{44}{100} \widehat y^2+\xi_2 z+2\xi_2\widehat y+\xi_1,z^2\right)_{L^2}
\\
&=-\tfrac{31}{128}\norm{z}{L^4}^4-\left(\tfrac{22}{25} \widehat y^2+4\xi_2\widehat y+2\xi_1,z^2\right)_{L^2}-2\left(\xi_2 z,z^2\right)_{L^2}.\notag
\end{align}
The Cauchy-Schwarz and Young inequalities give us
\begin{align}
&-2\left(\xi_2 z,z^2\right)_{L^2}\le 2\norm{\xi_2}{\bbR}\norm{z}{L^2}\norm{z}{L^4}^2\notag\\
&\qquad\le \gamma^{-1}\norm{\xi_2}{\bbR}^2\norm{z}{L^2}^2+\gamma\norm{z}{L^4}^4,\quad\mbox{for all}\quad\gamma>0.\notag
\intertext{Further, observe that}
&\left(-\tfrac{22}{25}\widehat y^2-4\xi_2\widehat y-2\xi_1,z^2\right)_{L^2}\le \widehat C \norm{z}{L^2}^2,
\intertext{where~$\widehat C\le\ovlineC{\norm{\zeta}{\infty}}$, namely,}
&\widehat C\coloneqq\tfrac{50}{11}\xi_2^2-2\xi_1=\max_{s\in\bbR}(-\tfrac{22}{25}s^2-4\xi_2s-2\xi_1),
\end{align}
\end{subequations}
where~$\norm{\zeta}{\infty}\coloneqq\max\limits_{1\le j\le 3}\norm{ \zeta_j}{\bbR}$;
 recall the~$\xi_i$s defined in~\eqref{xis}.

Now, from~\eqref{dtz1-aux} with~$\gamma=\frac{15}{128}$, we obtain
\begin{align}\notag
-2(z^3+\widehat f(z),z)_{L^2}&\le-\tfrac18\norm{z}{L^4}^4+(\tfrac{128}{15}\norm{\xi_2}{\bbR}^2+\widehat C) \norm{z}{L^2}^2.
\end{align}

Then, together with~\eqref{dtz1}, we arrive at the  inequality
\begin{align}
 \tfrac{\rmd}{\rmd t} \norm{z}{L^2}^2&\le-2\norm{z}{V} ^2-\tfrac18\norm{z}{L^4}^4+(\tfrac{128}{15}\norm{\xi_2}{\bbR}^2+\widehat C+2) \norm{z}{L^2}^2\notag\\ &\quad+2(U_M^\diamond\overline\clK_M(z),z)_{L^2},\label{dtz2}
\end{align}
which shall be used in following sections, together with the following
auxiliary results.
 \begin{lemma}\label{L:gen-poly}\textit{
Let~$(\beta_0,\beta_1,\beta_2,\varkappa,p)\in\bbR_+^5$. Then we have
$ -\beta_2\varkappa^{p}+\beta_0\varkappa\le-\beta_1\varkappa^\frac{p+1}{2}$
if $
\varkappa^\frac{p-1}{2}\ge\tfrac{\beta_1+\sqrt{\beta_1^2+4\beta_2\beta_0}}{2\beta_2}.$
}
\end{lemma}
\begin{proof}
With $r\coloneqq \varkappa^\frac{p-1}{2}$, we can rewrite~$-\beta_2\varkappa^{p}+\beta_0\varkappa\le-\beta_1\varkappa^\frac{p+1}{2}$
as~$(-\beta_2r^2+\beta_1r+\beta_0)\varkappa\le0$.
\end{proof}

\begin{lemma}\label{L:Ku-monot}\textit{
The constrained feedback operator satisfies
\begin{align}
&\widetilde\bfK(t)\coloneqq\left(U_M^\diamond\overline\clK_M(z(t)),z(t)\right)_{L^2}\notag\\
&=
\begin{cases}
-\lambda\min\left\{1,\tfrac{C_u}{\dnorm{v(t)}{}}\right\}\norm{P_{\clU_M}z(t)}{L^2}^2
&\mbox{ if } P_{\clU_M}z(t)\ne0,\\
0,&\mbox{ if } P_{\clU_M}z(t)=0,
\end{cases}\notag
\end{align}
where~$v(t)\coloneqq-\lambda(U_M^\diamond)^{-1}P_{\clU_M}z(t)$.
}
\end{lemma}
\begin{proof}
Recalling~\eqref{rad.proj.min} and the feedback in~\eqref{Schloegl-feed}, with~$v(t)=-\lambda(U_M^\diamond)^{-1}P_{\clU_M}z(t)$ we find
\begin{align}\notag
\widetilde\bfK(t)&=\left(U_M^\diamond\fkP^{\dnorm{\Bigcdot}{}}_{C_u}(-\lambda(U_M^\diamond)^{-1}P_{\clU_M}z(t)),z(t)\right)_{L^2}\notag\\
&=\min\left\{1,\tfrac{C_u}{\dnorm{v(t)}{}}\right\}(U_M^\diamond v(t),z(t))_{L^2},\quad\mbox{if}\quad v(t)\ne0,\notag
\end{align}
from which we can conclude the proof.
\end{proof}

\subsection{Norm decrease  for large error}\label{sS:normdeclargeerror}%
We consider again system~\eqref{Schloegl-feed}.
We show that if the error norm is large, at a given instant of time,
then such norm is decreasing at that instant of time. Here, the nonlinear term plays a crucial role.
Also, recall that~$\widehat y_0=\widehat y(0)$ is the initial state of the targeted trajectory~$\widehat y$ solving~\eqref{Schloegl}.
\begin{lemma}\label{L:expdec.largez0}\textit{
For every~$\mu>0$, there is a constant~$D\ge1$ such that for all~$(z_0,\widehat y_0)\in W^{1,2}(\Omega)\times W^{1,2}(\Omega)$ the solution of system~\eqref{Schloegl-diff}
satisfies
\begin{align}
 &\tfrac{\rmd}{\rmd t}\norm{z(t)}{L^{2}(\Omega)}\le-\mu\norm{z(t)}{L^{2}(\Omega)}\quad\mbox{if}\quad \norm{z(t)}{L^{2}(\Omega)}\ge D,\notag
\intertext{and}
&\norm{z(t)}{L^{2}(\Omega)}\le D\quad\mbox{for all}\quad t\ge (\mu^2D)^{-\frac1{2}}.\notag
\end{align}
Moreover, if for some~$s\ge0$ we have that~$\norm{z(s)}{L^{2}(\Omega)}\le D$,  then~$\norm{z(t)}{L^{2}(\Omega)}\le D$ for all~$t\ge s$.
Further,
$D\le\ovlineC{\mu,\norm{\zeta}{\infty}}$ is independent of~$(z_0,\widehat y_0, C_u)$.
}
\end{lemma}
\begin{proof}
By~\eqref{dtz2} and Lemma~\ref{L:Ku-monot} we find
\begin{align}
 \tfrac{\rmd}{\rmd t} \norm{z}{L^2}^2&\le-\tfrac18\norm{z}{L^4}^4+(\tfrac{128}{15}\norm{\xi_2}{\bbR}^2+\widehat C+2) \norm{z}{L^2}^2,\notag
 \end{align}
which, together with
\[
\norm{z}{L^2}^2\le\dnorm{\Omega}{}^\frac12\norm{z}{L^4}^2,\qquad \dnorm{\Omega}{}\coloneqq{\int_\Omega} 1\,\rmd x,
\]
gives us
\begin{align}
 \tfrac{\rmd}{\rmd t} \norm{z}{L^2}^2&\le -\tfrac18\dnorm{\Omega}{}^{-1}\norm{z}{L^{2}}^{4}+C_1\norm{z}{L^{2}}^{2},\notag
 \end{align}
where $0\le C_1=\tfrac{128}{15}\norm{\xi_2}{\bbR}^2+\widehat C+2\le\ovlineC{\norm{\zeta}{\infty}}$; with~$\xi_2$ and~$\widehat C$ defined in~\eqref{xis} and~\eqref{dtz1-aux}, respectively.

By taking~$(\beta_0,\beta_1,\beta_2,\varkappa,p)=(C_1,2\mu,\frac18\dnorm{\Omega}{}^{-1},\norm{z}{L^{2}}^{2},2)$ in Lemma~\ref{L:gen-poly}, we find
\begin{align}
&\tfrac{\rmd}{\rmd t} \norm{z}{L^{2}}^{2}\le -2\mu(\norm{z}{L^{2}}^{2})^{\frac{3}{2}}\notag\\
&\mbox{while}\quad \norm{z}{L^{2}} \ge \widehat D\coloneqq \tfrac{2\mu+\sqrt{4\mu^2+\frac12\dnorm{\Omega}{}^{-1}C_1}}{\frac14\dnorm{\Omega}{}^{-1}}.
\label{L2k-dty5}
 \end{align}
In particular
\begin{subequations}\label{L2k-dty6}
\begin{align}
&\tfrac{\rmd}{\rmd t} \norm{z}{L^{2}}^{2}\le -2\mu(\norm{z}{L^{2}}^{2})^{\frac{3}{2}}\quad\mbox{and}\quad\tfrac{\rmd}{\rmd t} \norm{z}{L^{2}}^{2}\le -2\mu\norm{z}{L^{2}}^{2}\\
&\hspace{0em}\mbox{while}\quad \norm{z}{L^{2}} \ge D\coloneqq\max\{1,\widehat D\}.
 \end{align}
\end{subequations}
Next, observe that for~$r>1$, the solution~$\varphi(t)\in\bbR$ of
\[
\dot\varphi=-2\mu\varphi^r,\qquad \varphi(0)=\varphi_0\ge0,\qquad t\ge0,
\]
is given by
\[
\varphi(t)=\frac{\varphi_0}{(1+2\mu(r-1)\varphi_0^{r-1}t)^\frac{1}{r-1}}.
\]
Hence, from~\eqref{L2k-dty6}, we conclude that if~$\norm{z_0}{L^{2}}>D$, then
\begin{align}
\norm{z(t)}{L^{2}}^{2}
\le \frac{\norm{z_0}{L^{2}}^{2}}{\left(1+\mu\norm{z_0}{L^{2}}t\right)^{2}}
\quad\mbox{while}\quad \norm{z(t)}{L^{2}} \ge  D.\notag
 \end{align}
In particular
\begin{align}
\norm{z(t)}{L^{2}}^{2}
\le (\mu t)^{-2} \mbox{ while } \norm{z(t)}{L^{2}} \ge  D,\quad \mbox{if } \norm{z_0}{L^{2}} >  D.\notag
 \end{align}

Observe also that if~$\norm{z_0}{L^{2}}>  D$, then
\begin{align}
&(\mu t)^{-2}= D
\quad\Longleftrightarrow\quad
t=\tau\coloneqq(\mu^2 D)^{-\frac12}.\notag
\end{align}

Therefore,
if $\norm{z_0}{L^{2}}> D$, then~$\norm{z(t_*)}{L^{2}}=  D$ for some $ t_*\le \tau$.
By~\eqref{L2k-dty6} it also follows that, if
$\norm{z(t_0)}{L^{2}}\le  D$ for some $t_0\ge 0$, then~$\norm{z(t)}{L^{2}}\le D$ for all $t\ge t_0$.
In particular, for every initial error~$z_0\in W^{1,2}(\Omega)$ it holds that
$\norm{z(t)}{L^{2}}\le D$, for all $t\ge\tau$.
Furthermore, from the second inequality in~\eqref{L2k-dty6} we can conclude that
 if~$\norm{z_0}{L^{2}} >  D$, then
\[
\norm{z(t)}{L^{2}}^2\le\rme^{-2\mu(t-s)}\norm{z(s)}{L^{2}}^2 \quad\mbox{while}\quad \norm{z(t)}{L^{2}} \ge  D.
\]

Finally, the constants $\widehat D\le\ovlineC{\mu,\norm{\zeta}{\infty}}$ and~$D\le\ovlineC{\mu,\norm{\zeta}{\infty}}$, defined in~\eqref{L2k-dty5} and~\eqref{L2k-dty6}, are independent of~$(z_0,\widehat y_0,C_u)$.
\end{proof}

\subsection{A property of the sequence of families of actuators}\label{sS:act.Lemma}
We present an auxiliary result concerning the sequence~$(\clU_M)_{M\in\bbN_+}$ in section~\ref{sS:actuators}.
Recall that~$P_{\clU_M}$ stands for the orthogonal projection in~$L^2(\Omega)$
onto~$\clU_M$. Further, for simplicity we denote by~$S^\perp\coloneqq S^{\perp,L^2(\Omega)}$ the orthogonal complement in~$L^2(\Omega)$ of a subset~$S\subseteq L^2(\Omega)$.
\begin{lemma}\label{L:HbddVlamM}\textit{
For every~$\varpi>0$, there exists~$M_*=\ovlineC{\varpi}\in\bbN_+$ such that
for all~$M\ge M_*$ we can find~$\lambda_*=\lambda_*(M)=\ovlineC{\varpi}>0$ such that
\[
\norm{w}{V}^2+2\lambda_*\norm{P_{\clU_M}w}{H}^2\ge \varpi\norm{w}{L^2}^2,\mbox{ for all }  w\in V=W^{1,2}(\Omega).
\]
}
\end{lemma}
\begin{proof}
In fact Lemma~\ref{L:HbddVlamM} follows from the result in~\cite[Cor.~3.1]{KunRodWalter21}
for general diffusion-like operators, which include the shifted Laplacian
$A=-\nu\Delta+\Id$, as mentioned in~\cite[sect.~5]{KunRodWalter21}. From the proof of~\cite[Cor.~3.1]{KunRodWalter21} we find
\[
\norm{w}{V}^2+2\lambda\norm{P_{\clU_M}w}{H}^2\ge\norm{w}{V}^2+2\lambda\norm{P_{\widetilde\clU_M}^{\clU_M^\perp}w}{\clL(H)}^{-2}\norm{P_{\widetilde\clU_M}^{\clU_M^\perp}P_{\clU_M}w}{H}^2
\]
where~$\widetilde\clU_M$ is an auxiliary finite-dimensional space, satisfying ~$L^2(\Omega)=\widetilde\clU_M+\clU_M^\perp$ and $\widetilde\clU_M\bigcap\clU_M^\perp=\{0\}$, and~$P_{\widetilde\clU_M}^{\clU_M^\perp}$ is the oblique projection in~$L^2(\Omega)$ onto~$\widetilde\clU_M$ along~$\clU_M^\perp$. From~\cite[sect.~6]{Rod21-aut} we know that by choosing the auxiliary space as the span of ``regularized'' actuators as in~\cite[Eq.~(6.8)]{Rod21-aut}, then the norm of the oblique projection is independent of~$M$. Hence
\[
\norm{w}{V}^2+2\lambda_*\norm{P_{\clU_M}w}{H}^2\ge\norm{w}{V}^2+2\lambda_* \Xi\norm{P_{\widetilde\clU_M}^{\clU_M^\perp}w}{H}^2
\]
with~$\Xi$ independent of~$M$. By the proof of~\cite[Lem.~3.5]{KunRodWalter21} the desired result follows if
\[
\beta_{M_+}\ge 4\varpi\quad\mbox{and}\quad\lambda_*\ge (2\varpi\norm{\Id}{\clL(V,H)}^2+1)\tfrac{\beta_{M}^2}{2\Xi},
\]
where, see~\cite[Eqs.~(2.2) and~(3.1)]{KunRodWalter21},
\[
\beta_{M_+}\coloneqq\inf_{\varTheta\in (V\bigcap\clU_M^\perp)\setminus\{0\}}\tfrac{\norm{\varTheta}{V}}{\norm{\varTheta}{H}}
\quad\mbox{and}\quad
\beta_M\coloneqq\sup_{\theta\in \widetilde\clU_M\setminus\{0\}}\tfrac{\norm{\theta}{V}}{\norm{\theta}{H}}.
\]
From~\cite[sect.~5]{Rod21-sicon}, it follows that~$\beta_{M_+}=C_1M^2+1$.
Therefore, we can choose~$M_*=\min\{M\in\bbN_+\mid \beta_{M_+}\ge 4\varpi\}$ and
$\lambda_*=(2\varpi\norm{\Id}{\clL(V,H)}^2+1)\tfrac{\beta_{M}^2}{2\Xi}\}$.
\end{proof}

\begin{remark}\label{R:lambda-M}
 From~\cite[end of Sect.~5]{Rod21-sicon} (see also~\cite[Thm.~6.1]{Rod21-aut}) it follows that~$\beta_{M}\ge C_2M^2+1$. Therefore, for large~$M$  we (may) need to take large $\lambda_*$.
\end{remark}
\begin{remark}\label{R:PolyDom}
The divergence~$\beta_M\to+\infty$ plays a crucial role in the derivation of Lemma~\ref{L:HbddVlamM}. The proof of such divergence, shown in~\cite[Sect.~5]{Rod21-sicon} \cite[Thm.~6.1]{Rod21-aut} for rectangular domains (boxes), can be adapted for general polygonal domains which are the union of a finite number of triangles (simplexes). The proof in~\cite{Rod21-sicon,Rod21-aut} is based on the fact that a rectangle can be partitioned into rescaled copies of itself. Note that a triangle can also be partitioned into smaller triangles. Indeed, for planar triangles, $d=2$, we obtain $4$ similar congruent triangles by connecting the middle points of the edges, and iterating the procedure we obtain finer partitions into congruent triangles. For the case~$d=3$, it may be not possible to partition a triangle~$\clT$ (tetrahedron) into smaller triangles all congruent to~$\clT$, however the partition is possible into triangles where the number of congruent classes does not exceed~$3$; see~\cite[Sect.~3 and Fig.~5]{EdelGrayson00} \cite[Thm.~4.1 and Fig.~5]{Bey00}. This fact allows us to repeat/adapt the arguments in~\cite{Rod21-sicon,Rod21-aut}.  

On the other hand the satisfiability of the divergence~$\beta_M\to+\infty$ for smooth domains is an open nontrivial question (cf. \cite[Conj.~4.6]{Rod21-jnls} and discussion thereafter).
\end{remark}

\subsection{Proof of Theorem~\ref{T:mainSchloegl}}\label{sS:proofT:mainSchloegl}
For an arbitrary given~$\mu>0$, by~\eqref{dtz2} we have that
\begin{align}
 \tfrac{\rmd}{\rmd t} \norm{z}{L^2}^2&\le-\norm{z}{V} ^2+(\tfrac{128}{15}\norm{\xi_2}{\bbR}^2+\widehat C+2) \norm{z}{L^2}^2\notag\\ &\quad+2(U_M^\diamond\overline\clK_M(z),z)_{L^2}\label{dtz0}\\
&\hspace{-.5em}\le-\norm{z}{V} ^2 +2(U_M^\diamond\overline\clK_M(z),z)_{L^2}+\varpi\norm{z}{L^2}^2-2\mu\norm{z}{L^2}^2
,\notag
\end{align}
with
\begin{equation}\notag
\varpi\coloneqq 2\mu+\tfrac{128}{15}\norm{\xi_2}{\bbR}^2+\widehat C+2,
\end{equation}
and~$\xi_2$ and~$\widehat C$ as in~\eqref{xis} and~\eqref{dtz1-aux}.
With~$M_*=\ovlineC{\varpi}\le\ovlineC{\mu,\norm{\zeta}{\infty}}$ and~$\lambda_*>0$ be given by Lemma~\ref{L:HbddVlamM}, we arrive at
\begin{equation}\notag
\norm{z}{V}^2+2\lambda\norm{P_{\clU_M}z}{L^2}^2\ge \varpi\norm{z}{L^2}^2,\mbox{ if }  M\ge\! M_* \mbox{ and } \lambda\ge\!\lambda_*(M),
\end{equation}
which leads us to
\begin{align}
 \tfrac{\rmd}{\rmd t} \norm{z}{L^2}^2
&\le 2(U_M^\diamond\overline\clK_M(z),z)_{L^2}+2\lambda\norm{P_{\clU_M}z}{L^2}^2
\notag\\
&\quad-2\mu\norm{z}{L^2}^2.\label{dtz+2}
\end{align}

Observe that, by Lemma~\ref{L:Ku-monot}, we have that
\begin{align}
&\left(U_M^\diamond\overline\clK_M(z(t)),z(t)\right)_{L^2}=
-\lambda\norm{P_{\clU_M}z(t)}{L^2}^2\notag\\
\quad\Longleftarrow\quad& \dnorm{-\lambda(U_M^\diamond)^{-1}P_{\clU_M}z}{}\le C_u\notag\\
\quad\Longleftarrow\quad&  C_u\ge \lambda\dnorm{(U_M^\diamond)^{-1}}{}\norm{z}{L^2}.\label{inactCu1}
\end{align}
where~$\dnorm{(U_M^\diamond)^{-1}}{}\coloneqq\max\limits_{w\in L^2\setminus\{0\}}\frac{\dnorm{(U_M^\diamond)^{-1}w}{}}{\norm{w}{L^2}}$.

From Lemma~\ref{L:expdec.largez0}, there is a constant~$D\ge 1$ such that for
\begin{subequations}\label{zhaty.tauD-lt}
\begin{align}
&t_\rma\coloneqq\min\{t\ge0\mid\norm{z(t)}{L^2}\le D\}\le(\mu^2D)^{-\frac12},\\
\intertext{we have that}
&\norm{z(t)}{L^2}\le \rme^{-\mu(t-s)}\norm{z(s)}{L^2},\mbox{ for all } 0\le s\le t\le t_\rma,\\
&\norm{z(t)}{L^2}\le D,\mbox{ for all } t\ge t_\rma.
\end{align}
\end{subequations}

Now, motivated by~\eqref{inactCu1} we set
\begin{equation}\label{Cu*}
C_u^*\coloneqq  \lambda\dnorm{(U_M^\diamond)^{-1}}{}D.
\end{equation}

We can conclude that
\begin{align}
\dnorm{-\lambda(U_M^\diamond)^{-1}P_{\clU_M}z(t)}{}\le C_u,\quad\mbox{if}\quad C_u\ge C_u^*
\quad\mbox{and}\quad t\ge t_\rma.\notag
\end{align}
By~\eqref{inactCu1}, we have~$(U_M^\diamond\overline\clK(z(t)),z(t))_{L^2}=-\lambda\norm{P_{\clU_M}z(t)}{L^2}^2$ if~$C_u\ge C_u^*$ and~$t\ge t_\rma$.
Therefore, by~\eqref{dtz+2},
\begin{align}
 \tfrac{\rmd}{\rmd t} \norm{z}{L^2}^2
&\le-2\mu\norm{z}{L^2}^2,\quad\mbox{for all}\quad t\ge t_\rma, \quad\mbox{if}\quad C_u\ge C_u^*,\notag
\end{align}
which implies that
\begin{align}
\norm{z(t)}{L^2}
&\le\rme^{-\mu(t-s)}\norm{z(s)}{L^2},\;\; t\ge s \ge t_\rma, \;\; C_u\ge C_u^*.\label{z1+}
\end{align}

Combining~\eqref{zhaty.tauD-lt} and~\eqref{z1+} we find
\begin{align}
&\norm{z(t)}{L^2}\le\rme^{-\mu(t-t_\rma)} \rme^{-\mu(t_\rma-s)}\norm{z(s)}{L^2}=\rme^{-\mu(t-s)} \norm{z(s)}{L^2},\notag\\ &\hspace{3em}\quad\mbox{for all}\quad  t\ge t_\rma\ge s\ge0,\quad\mbox{if}\quad C_u\ge C_u^*.\label{z1x}
\end{align}
By~\eqref{zhaty.tauD-lt},~\eqref{z1+}, and~\eqref{z1x}, we have that~\eqref{Tmain.exp} holds true.
\hfill\qed

\begin{remark}\label{R:Cu-M}
Note that~$C_u^*$ as in~\eqref{Cu*} depends on~$\lambda=\lambda(M)$. Due to Remark~\ref{R:lambda-M}, we can expect~$\lambda(M)$ to increase with~$M$. From the proof of Theorem~\ref{T:mainSchloegl} it also follows that the triple~$(M_*,\lambda_*,C_u^*)$ can be chosen independent of $h$, however, the same triple depends on~$\nu$ through the used Lemma~\ref{L:HbddVlamM}, because by~\eqref{Vscalarprod} the $V$-norm depends on~$\nu$.
\end{remark}

 \section{Receding Horizon Control}
  \label{S:optimalcontrol}
 We investigate the performance and stability of Algorithm~\ref{RHA} applied to  infinite-horizon problem~\eqref{inf-hor-por}.  To present our results, we introduce the finite- and infinite-horizon value functions.
\begin{definition} Let $\widehat y \in  W_{\rm loc}(\bbR_+,W^{2,2},L^2)$ (cf. sect.~\ref{sS:wellPose}) be a solution to~\eqref{sys-haty} for a given pair $(\widehat y_0,h)$.  The infinite-horizon value function $V_{\infty}:  W^{1,2} \to\overline{ \mathbb{R}_+}$ and the finite-horizon value function $V_{T}\colon\overline{ \mathbb{R}_+}\times W^{1,2} \to \overline{\mathbb{R}_+}$, with~$T>0$, are as follows; see~\eqref{inf-hor-por} and~\eqref{finite-OP}. With~$I_{t_0}^T\coloneqq(t_0,t_0+T)$,
\begin{align}\notag
V_{\infty}(y_0)&\coloneqq
\inf_{ u \in L^2(\mathbb{R}_+,\mathbb{R}^{M_{\sigma}})}\{J_ {\infty}({u;y_0,\widehat y)} \text{ subject to~\eqref{sys-y-u-Cu}}\}.
\\\notag
V_{T}(t_0, \bar {y}_0)&\coloneqq \inf_{u \in L^2(I_{t_0}^T,\mathbb{R}^{M_{\sigma}})} \{J_{T}(u;t_0,\bar{y}_0,\widehat y) \text{ subject to~\eqref{finite-OP-subto}} \}.
\end{align}
\end{definition}
\begin{theorem}
\label{subopth}\textit{
 Suppose that for a chosen $(\mu, \lambda, C_u)$,  system~\eqref{Schloegl-feed} is globally exponentially stable with rate $\mu$  around the reference trajectory $\widehat{y}$ of~\eqref{Schloegl}, that is,  \eqref{Tmain.exp} holds for every $(y_0, \widehat{y}_0) \in W^{1,2} \times W^{1,2}$.
Then,  for every given  sampling time $\delta>0$, there are numbers $T^* > \delta$, and $\alpha \in (0,1)$,  such that for every  prediction horizon $T \geq T^*$, and every pair of  $(y_0, \widehat{y}_0) \in W^{1,2} \times W^{1,2}$,  the receding horizon control $u_{rh}$ given by  Algorithm~\ref{RHA} satisfies the suboptimality inequality
\begin{equation}
\notag
\alpha V_{\infty}(y_0) \leq \alpha J_{\infty}(u_{rh};t_0,y_0,\widehat{y}) \leq V_{\infty}(y_0),
\end{equation}
 and provides the asymptotic behavior $\norm{y_{rh}(t)-\widehat y(t)}{L^2(\Omega)}\to 0$ as $t \to +\infty$.
}
\end{theorem}
\begin{proof}
 We consider the dynamics of the error $z:= y-\widehat{y}$, see~\eqref{sys-z-intro}, with a general control~$u$ in the time interval~$I_{t_0}^T=(t_0,t_0+T)$, with~$t_0\ge0$ and $T \in (0,+\infty]$,
\begin{subequations}\label{sys-z-rhc}
\begin{align}
&\tfrac{\p}{\p t} z -\nu\Delta z +f^{\widehat y}(z)\!=\! U_M^\diamond u,\; \tfrac{\p}{\p\bfn} z\rest{\p\Omega}\!=\!0,\; z(t_0)\!=\! \bar{z}_0,\label{sys-eq-rhc}\\ &\dnorm{u(t)}{}\le C_u,   
 \label{sys-eq-contc}
\end{align}
 \end{subequations}
where~$f^{\widehat y}(z)=z^3+(3\widehat y+\xi_2) z^2+(3\widehat y^2+2\xi_2\widehat y+\xi_1) z$.   In this way,  for the solution $z = z(t_0,\bar z, u)$ of~\eqref{sys-z-rhc},  $\beta >0$,  we define the following performance index
\begin{equation}
\label{Obj}
J^D_{T}{(u; t_0,  \bar{z}_0)}\coloneqq\norm{z}{L^2({ I_{t_0}^T},{L^2}
)}^2+\beta\norm{u}{L^2({I_{t_0}^T},\bbR^{M_\sigma})}^2.
\end{equation}
Then,  we can also define the following value functions
\begin{align}
\overline{V}_{\infty}(z_0)&:=
\inf_{ u \in L^2(\mathbb{R}_+,\mathbb{R}^{M_{\sigma}})}\{J^D_ {\infty}({u;0,z_0)} \text{ subject to~\eqref{sys-z-rhc}}\},\notag \\
\overline{V}_T{(t_0,z_0)}&:=
\inf_{ u \in L^2(I_{t_0}^T,\mathbb{R}^{M_{\sigma}})}\{J^D_ {T}{(u;t_0,z_0)} \text{ subject to~\eqref{sys-z-rhc}}\},\notag
\end{align}
 (for given initial pairs $(0, z_0)$ and~$(t_0, z_0)$, respectively).
 Next, for a given triple~$(t_0,T,\bar z_0)$, we define the  open-loop problem
\begin{equation}
\label{optD}
\min_{u \in  L^2(I_{t_0}^T,\bbR^{M_\sigma} )}   J^D_T(u;t_0,\bar{z}_0) \mbox{ subject to ~\eqref{sys-z-rhc}},
\end{equation}
and write Algorithm~\ref{RHA2}. By induction, it can be easily  be shown  that both of Algorithms~\ref{RHA} and~\ref{RHA2}
\begin{algorithm}[htbp]
\caption{RHC2($\delta,T$)}\label{RHA2}
\begin{algorithmic}[1]
\REQUIRE{sampling time  $\delta>0$, prediction horizon $T> \delta$,   initial state $y_0 \in W^{1,2}$, targeted trajectory $\widehat{y}$ solving~\eqref{sys-haty}}
\ENSURE{ Receding horizon control  ~$u_{rh} \in L^2(\mathbb{R}_0,  \mathbb{R}^{M_{\sigma}})$.}
\STATE Set~$t_0=0$ and~$\bar{z}_0: =y_0-\widehat{y}_0$;
\STATE Find~$(z_T^*(\cdot;t_0,\bar{z}_0),u^*_T(\cdot;t_0,\bar{y}_0))$
for time in~$(t_0,t_0+T)$ by solving the  open-loop problem~\eqref{optD};
\STATE For all $\tau \in [t_0,t_0+\delta)$,  set $u_{rh}(\tau)=u_T^*(\cdot;t_0,\bar{y}_0)$;
\STATE Update: $\bar{z}_0\leftarrow z_T^*(t_0+\delta;t_0,\bar{z}_0),$;
\STATE Update: $t_0 \leftarrow t_0 +\delta$;
\STATE Go to step $2$;
\end{algorithmic}
\end{algorithm} deliver the same receding horizon control $u_{rh}$.  Further for $z_0 = y_0-\widehat{y}_0$,  it can be seen that  $V_{\infty}(y_0) = \overline{V}_{\infty}(z_0)$,  $V_{T}(0,y_0) = \overline{V}_{T}(0,z_0)$,  and $J^D_{\infty}(u_{rh};0,z_0) =J_{\infty}(u_{rh};0,y_0)$.  Thus we can restrict ourselves  to show that the following suboptimality inequalities and asymptotic behavior hold for Algorithm~\ref{RHA2},
\begin{subequations}
\begin{align}
\label{subOD}
&\alpha \overline{V}_{\infty}(z_0) \leq \alpha J^D_{\infty}(u_{rh};0,z_0) \leq \overline{V}_{\infty}(z_0),\\
\label{asymRHC}
&\norm{z_{rh}(t)}{L^2}\to 0 \text{ as }t \to +\infty.
\end{align}
\end{subequations}

\textbf{Verification of~\eqref{subOD}:}  Since the stabilizability result for the time-varying  system~\eqref{sys-z-rhc} holds globally,   the proof follows with similar arguments given in \cite[Thm.~2.6]{AzmiKun2019}.  Therefore, we omit the proof here and restrict ourselves to the verification of Properties  {\bf P1}--{\bf P2} in  \cite[Thm.~2.6]{AzmiKun2019} as follows.

\noindent
{\bf P1}: {\em There is a continuous,  nondecreasing,  and bounded function $\gamma_2: \mathbb{R}_+ \to \mathbb{R}_+$ such that, for all  $(t_0,z_0)\in \mathbb{R}_+\times W^{1,2}$,
\begin{equation}
\label{e7}
\overline{V}_T(t_0,z_0)  \leq  \gamma_2(T)|z_0|^2_{L^2}.
\end{equation}}
First, we show that for every given $(t_0,z_0)$  there exist a control~$\hat u \in L^2((t_0,+\infty), \mathbb{R}^{M_{\sigma}})$, satisfying~\eqref{sys-eq-contc},  for which the solution of~\eqref{sys-eq-rhc}
 satisfies
 \begin{align}\label{goal.parab-MPC-exp}
\norm{z(t)}{L^2}\le \rme^{-\mu (t-t_0)}\norm{y(t_0)}{L^2} \mbox{ for } t\ge t_0.
\end{align}
This follows  by applying the control  law~\eqref{Behcont}, with~$(\mu, \lambda, C_u)$ given by Theorem~\ref{T:mainSchloegl}, to  the following time-shifted system  with $ \bar z (s) :=z(s+t_0 )$,  $\bar{\widehat{y}}(s) := \widehat{y}(t_0+s)$,   for $s \geq 0$,
\begin{align}
&\tfrac{\p}{\p s} \bar{z} -\nu\Delta \bar{z} +f^{\bar{\widehat{y}}}(\bar{z})= U_M^\diamond\bar{u},\quad \tfrac{\p}{\p\bfn}\bar{z}\rest{\p\Omega}=0,\quad
 \bar{z}(0,\Bigcdot)= z_0,\notag
 \end{align}
with  $\bar{\widehat{y}}$  the solution of the free dynamics~\eqref{sys-haty} with forcing function  $\bar{h}: = h(\cdot + t_0)$  and initial function $\widehat{y}(t_0)$ in place of $h$ and $\widehat{y}_0$,  respectively.   Since $h \in  L^2_{\rm loc}(\bbR_+,L^2(\Omega))$ and $\hat{y}_0 \in W^{1,2}$ it follows that  $\bar{h} \in L^2_{\rm loc}(\bbR_+,L^2(\Omega))$ and $\widehat{y}(t_0) \in W^{1,2}$ and,  thus,  the saturated control $ U_M^\diamond\bar{u}=  \sum\limits_{i=1}^{M_\sigma} \bar{u}_i\indf_{\omega^M_{i}}:=U_{M}^\diamond\fkP^{\dnorm{\Bigcdot}{}}_{C_u}\left(-\lambda (U_{M}^\diamond)^{-1}P_{\clU_M}\bar{z}\right)$ is exponentially stabilizing.  Therefore, ~\eqref{goal.parab-MPC-exp} holds for  $\hat{u}(t):=\bar{u}(t-t_0) $ for $t\geq t_0$,  where  $\lambda$, $C_u$,  and $M$ are independent of $t_0$ and $z_0$.    Evaluating~\eqref{Obj}  at  $u = \hat{u}$ we obtain
\begin{align}
\overline{V}_T(t_0,z_0)& \le J^D_T( \hat{u}; t_0,z_0) \le  \tfrac{(1+\beta\bar{C} ) }{2\mu}(1-e^{-2\mu T})\norm{z_0}{L^2}^2\notag\\
&\eqqcolon \gamma_2(T)\norm{z_0}{L^2}^2,\label{estVal}
\end{align}
where $ \bar{C}$ depends on $\clK$,  $C_u$,   and  $U_{M}^\diamond$; see ~\eqref{rad.proj}.

\noindent
\underline{\bf P2}: {\em For every $(t_0, \bar{z}_0) \in \overline{\mathbb{R}}_+ \times W^{1,2}$,  every finite horizon optimal control problem of the form ~\eqref{optD} admits a solution:} We use a standard argument from calculus of variations.  Since the set of admissible controls
\begin{equation}\notag
\mathcal{U}_{ad}:=\{ u \in L^2( I_{t_0}^T, \mathbb{R}^{M_{\sigma}}) :  \| u(t) \| \leq C_u \},
\end{equation}
is nonempty and  $J^D_T$ is nonnegative,   we can select an admissible minimizing sequence $\{(u^n,z^n)\}_{n} \in L^2( I_{t_0}^T, \mathbb{R}^{M_{\sigma}})   \times W((0,T),W^{2,2},L^2) $ satisfying $J^D_T(u^n;t_0, \bar{z}_0)\to \overline{V}_T(t_0, \bar{z}_0)$.  Due to the fact that ~$ \mathcal{U}_{ad}$ is a closed bounded and convex subset of~$L^2( I_{t_0}^T, \mathbb{R}^{M_{\sigma}})$ and using the energy estimate
\begin{equation}
\label{energ}
\norm{z}{W(I_{t_0}^T,W^{2,2},L^2)} \leq C_z\left(\norm{\bar{z}_0}{W^{1,2}}+\norm{u}{L^2(I_{t_0}^T, \mathbb{R}^{M_{\sigma}})} \right),
\end{equation}
with $C_z>0$   for ~\eqref{sys-eq-rhc},  which is justified by similar arguments given in section~\ref{sS:wellPose},  we can infer that there exists a weakly convergent subsequence, still denoted by~$\{(u^n,z^n)\}_{n}$, so that
\begin{equation}
 \label{weaklim}
\begin{split}
 z^n    \xrightharpoonup[W( I_{t_0}^T, W^{2,2},L^2)]{}z^*  \quad  \text{ and } \quad
 u^n \xrightharpoonup[L^2( I_{t_0}^T, \mathbb{R}^{M_{\sigma}})]{} u^*.
 \end{split}
 \end{equation}
We verify that $z^*$  is the strong solution of~\eqref{sys-eq-rhc} corresponding  to  $u^*$.  To show this, we need to pass the limit in the variational formulation for~\eqref{sys-eq-rhc}.   From~\eqref{weaklim} we find that
 \begin{equation}\notag
 \left(\tfrac{\p}{\p t} z^n,\Delta z^n, U_M^\diamond u^n \right) \xrightharpoonup[\left(L^2( I_{t_0}^T, L^2)\right)^3]{} \left(\tfrac{\p}{\p t} z^*,\Delta z^*,  U_M^\diamond u^* \right).
 \end{equation}
It remains to show that $f^{\widehat y}(z^n)\xrightharpoonup[L^2( I_{t_0}^T, L^2)]{} f^{\widehat y}(z^*)$.  Writing~$ \delta_z^n\coloneqq z^n- z^*$ and~$\delta_f^n\coloneqq f^{\widehat y}(z^n)-f^{\widehat y}(z^*)$, we find
\begin{subequations}\label{nonlin-exist}
\begin{align}
\norm{ \delta_f^n}{L^2}
&\le\norm{(z^n)^3-(z^*)^3}{L^2}+\norm{\widehat f(z^n)-\widehat f(z^*)}{L^2} \notag\\
&\hspace{0em}\le \norm{(z^n)^3-(z^*)^3}{L^2}+\norm{(3\widehat y+\xi_2) ((z^n)^2- (z^*)^2)}{L^2}\notag\\
&\quad+\norm{(3\widehat y^2+2\xi_2\widehat y+\xi_1) \delta_z^n}{L^2}
\end{align}
and,  estimating separately the terms on the right-hand side,
\begin{align}
&\norm{(z^n)^3-(z^*)^3}{L^2}
\le
\norm{ (z^n)^2+z^n z^{*}+  (z^*)^2}{L^3}\norm{\delta_z^n}{L^6}\notag\\
&\hspace{0em}\le\varPsi_1\norm{\delta_z^n}{L^6},\mbox{ with }
\varPsi_1:=\tfrac32\left(
\norm{z^n}{L^6}^2+\norm{z^*}{L^6}^2\right);\\
&\norm{(3\widehat y+\xi_2) ((z^n)^2- (z^*)^2)}{L^2}\notag\\
&\hspace{0em}
\le\norm{(3\widehat y+\xi_2) (z^n+ z^*)}{L^3}\norm{\delta_z^n}{L^6}
\le\varPsi_2\norm{ z^n- z^*}{L^6}\notag\\
&\hspace{3em}\mbox{with }
\varPsi_2:=\norm{3\widehat y+\xi_2}{L^6} \left(\norm{z^n}{L^6}+\norm{z^*}{L^6}\right);\\
&\norm{(3\widehat y^2+2\xi_2\widehat y+\xi_1) \delta_z^n}{L^2}
\le\norm{3\widehat y^2+2\xi_2\widehat y+\xi_1}{L^3}\norm{\delta_z^n}{L^6}\notag\\
&\hspace{0em}\le\varPsi_3\norm{\delta_z^n}{L^6},
\mbox{ with }
\varPsi_3:=\left(3\norm{\widehat y}{L^6}^2\!+\!\norm{2\xi_2\widehat y+\xi_1}{L^3}\right).
\end{align}
\end{subequations}
Now,  since the terms $\norm{z^*}{W( I_{t_0}^T,W^{2,2},L^2)}$,  $\norm{\widehat{y}}{W( I_{t_0}^T,W^{2,2},L^2)}$,  and  $\{ \norm{z^n} {W( I_{t_0}^T,W^{2,2},L^2)}\}_{n}$  are bounded and since  $W( I_{t_0}^T,W^{2,2},L^2) \hookrightarrow L^{\infty}( I_{t_0}^T,W^{1,2})$,   we have
\[   \varPsi_i(t) \leq C_{ \varPsi} \quad  \text { for all }  t \in  I_{t_0}^T  \text{ and }  i \in \{ 1,2,3 \}   \]
Therefore,  from~\eqref{nonlin-exist}, we can conclude
\begin{equation}
\label{nonEs}
\norm{\delta_ f^n}{L^2( I_{t_0}^T, L^2)}^2 \leq 9C^2_{\varPsi}\norm{\Id}{\clL(W^{1,2},L^6)}^2\norm{\delta_z^n}{L^2( I_{t_0}^T;W^{1,2})}^2.
\end{equation}
Using~\eqref{weaklim}, \eqref{nonEs},  and the fact that the embedding $W( I_{t_0}^T,W^{2,2},L^2) \hookrightarrow L^{2}( I_{t_0}^T,W^{1,2})$ is compact,  we conclude that $f^{\widehat y}(z^n)\xrightarrow[L^2( I_{t_0}^T, L^2)]{} f^{\widehat y}(z^*)$ and,  as a consequence,  $z^*$ is the strong solution associated to $u^*$.

Finally,  since  $z^n    \xrightarrow[L^2( I_{t_0}^T,L^2)]{}z^*$  and $J^D_T$ is a convex and continuous mapping in $u$, it is weakly lower semi-continuous and,  thus, we obtain that
\[ J_T^D(u^*; t_0,\bar{z}_0)  \leq \liminf_{n \to \infty} J_T^D(u^n; t_0,\bar{z}_0)
= \overline{V}_T(t_0, \bar{z}_0).
\]
Hence,  $(z^*,u^*)$   is optimal.
The rest of proof follows as in the proof of  \cite[Thm.~2.6]{AzmiKun2019}  by using  a dissipation inequality for the finite-horizon value function $\overline{V}_T$.   It is derived by applying~\eqref{e7} for all initial pairs $(t_i, z_{rh}(t_i))$ with $i>1$,  where $t_i= \delta+ t_{i-1}$ and $z_{rh}(t_i)=z^*_T(t_i; t_{i-1},z_{rh}(t_{i-1}))$.  Therefore it is essential that $z_{th}(t_i) \in W^{1,2}$ for all $i\geq1$.   This fact can be justified using induction and estimate~\eqref{energ} repeatedly.

\textbf{Verification of~\eqref{asymRHC}}.  Using~\eqref{subOD} and~\eqref{estVal}, for $t_0= 0$ and $T \to \infty$,  we obtain
\begin{equation}\notag
\norm{z_{rh}}{L^2(\mathbb{R}_0, L^2)}^2+ \beta\norm{u_{rh}}{L^2(\mathbb{R}_0, \mathbb{R}^{M_{\sigma}})}^2\!  \leq\! \tfrac{\overline{V}_{\infty}(z_0)}\alpha\!  \leq\!  \tfrac{(1+\beta\bar{C} ) }{2\alpha\mu}\norm{z_0}{L^2}^2.
\end{equation}
Therefore,  there exists a constants $C_J$ such that
\begin{equation}
\label{bouForRh}
\max\left\{\norm{z_{rh}}{L^2(\mathbb{R}_+, L^2)}^2,\norm{u_{rh}}{L^2(\mathbb{R}_+, \mathbb{R}^{M_{\sigma}})}^2\right\} \leq C_J \norm{z_0}{L^2}^2.
\end{equation}
Further,  similarly to the estimate~\eqref{dtz0},  we can find
\begin{align}\notag
 \tfrac{\rmd}{\rmd t} \norm{z_{rh}}{L^2}^2&\le-\norm{z_{rh}}{V} ^2+D_0\norm{z_{rh}}{L^2}^2+2(z_{rh},   U_M^\diamond u_{rh})_{L^2},
 \end{align}
 where $D_0 := \tfrac{128}{15}\norm{\xi_2}{\bbR}^2+\widehat C+2$  was defined for~\eqref{dtz0}.
  After time integration over the interval~$(s,t)$, with~$t>s$, we obtain
  \begin{align}
   &\Xi(t,s)\coloneqq \norm{z_{rh}(t)}{L^2}^2-\norm{z_{rh}(s)}{L^2}^2\label{estAsy}\\
&\quad\le D_0 \int^t_{s} \norm{z_{rh}}{L^2}^2\,\rmd t+  2C_{B}\int^t_{s} \norm{z_{rh}}{L^2}\norm{ u_{rh}}{\mathbb{R}_{M_{\sigma}}}\,\rmd t,\notag 
\end{align}
where $C_{B}>0$ depends only on $U_M$.  Using ~\eqref{estAsy} with $s=0$, ~\eqref{bouForRh},  and Young's inequality we obtain
\begin{equation}
\label{LinfEs}
\norm{z_{rh}}{L^{\infty}(\mathbb{R}_0, L^2)} \leq C_{inf} \norm{z_0}{L^2},
\end{equation}
for some $C_{inf}>0$.
From~\eqref{estAsy}, with~$Y\coloneqq L^2((s,t), L^2)$,
\begin{align}
\Xi(t,s) &\le D_0  \norm{z_{rh}}{Y}^2+  2C_{B} \norm{z_{rh}}{Y} \norm{u_{rh}}{L^2((s,t), \mathbb{R}_{M_{\sigma}})}\notag\\
& \le C_h (t-s)^{\frac{1}{2}}\|z_0\|^2_{L^2}\label{EsHol}
\end{align}
 for every $t \geq s$, where  $C_h:=C_{inf}(D_0+2C_B)C_J$,  and ~\eqref{bouForRh}  and ~\eqref{LinfEs} were used.
The rest of proof follows  the same lines as in  the proof of  \cite[Thm.~6.4]{AzmiKun2019}  based on~\eqref{EsHol}  and~\eqref{bouForRh}.
\end{proof}

 \section{Numerical simulations}\label{S:simulations}
We present the results of numerical simulations showing the stabilizing performance of the proposed saturated feedback. We compute the targeted trajectory~$\widehat y$, solving~\eqref{Schloegl},
\begin{align}
 &\tfrac{\p}{\p t} \widehat y -\nu\Delta \widehat y+f(\widehat y)= h, \quad \tfrac{\p}{\p\bfn}y\rest{\p\Omega}=0,\quad \widehat y(0)=\widehat y_0,\notag
 \end{align}
where~$f(w)\coloneqq(w- \zeta_1)(w- \zeta_2)(w- \zeta_3)$,
and the controlled trajectory~$y$ solving~\eqref{Schloegl-feed},
\begin{align}
 &\tfrac{\p}{\p t} y -\nu\Delta y+f(y)= h+U_M^\diamond\overline\clK_M(y-\widehat y),\notag\\
&y(0)=y_0,\qquad\tfrac{\p}{\p\bfn}y\rest{\p\Omega}=0,\notag
\intertext{with the saturated feedback control}
 &\overline\clK_M(y-\widehat y)=  \fkP^{\dnorm{\Bigcdot}{}}_{C_u}\left(-\lambda (U_{M}^\diamond)^{-1}P_{\clU_M}(y-\widehat y)\right).\notag
 \end{align}
We also report on numerical experiments associated with Algorithm~\ref{RHA} and draw a comparison between the saturated controls and the RHC laws. 
We have chosen the parameters
\begin{align}\notag
&\nu=0.1\quad\mbox{and}\quad (\zeta_1,\zeta_2,\zeta_3)=(-1,0,2).
\end{align}
The spatial domain is the unit square $\Omega=(0,1)\times(0,1)\subset\bbR^2$.  As spatial discretization we have taken a standard piecewise linear finite element approximation, for the triangulations given  in Fig.~\ref{Fig:Mesh}. Most of the simulations correspond to the case~$M_\sigma=9$ with ${\rm npts}\Omega=3328$ mesh points (degrees of freedom). For the temporal discretization we have taken a standard Crank--Nicolson/Adams--Bashford scheme with stepsize $ k = 10^{-3}$.
For solving the open-loop problems within Algorithm~\ref{RHA},  we employed a projected gradient method to the   associated reduced problems, namely,  we used the iteration
 \begin{equation}\notag
u^{j+1} = P_{\mathcal{U}_{ad}}(u^j -\alpha_j\mathcal{F}'(u^j) ),
\end{equation}
where $\mathcal{F}'$ stands for the gradient of the reduced problem and the  stepsize $\alpha_j$ is computed by a nonmonotone linesearch algorithm which uses the Barzilai--Borwein stepsizes \cite{BB88, AzmiKun2020} corresponding to $ \mathcal{F}$ as the initial trial stepsize, see \cite{AzmiKun2021},  and the references therein for more details.  For every problem,  we used the saturated control with $\lambda = 175$ as the initial iterate.  Further,  the optimization algorithm was  terminated when the norm of the difference corresponding to two successive iterations was less that $10^{-4}$.

We will consider the cases of~$1$, $4$, $9$, and~$16$ actuators, whose locations and supports are illustrated in Fig.~\ref{Fig:Mesh}.

\begin{figure}[ht]
\centering
\subfigure
{\includegraphics[width=0.45\textwidth]{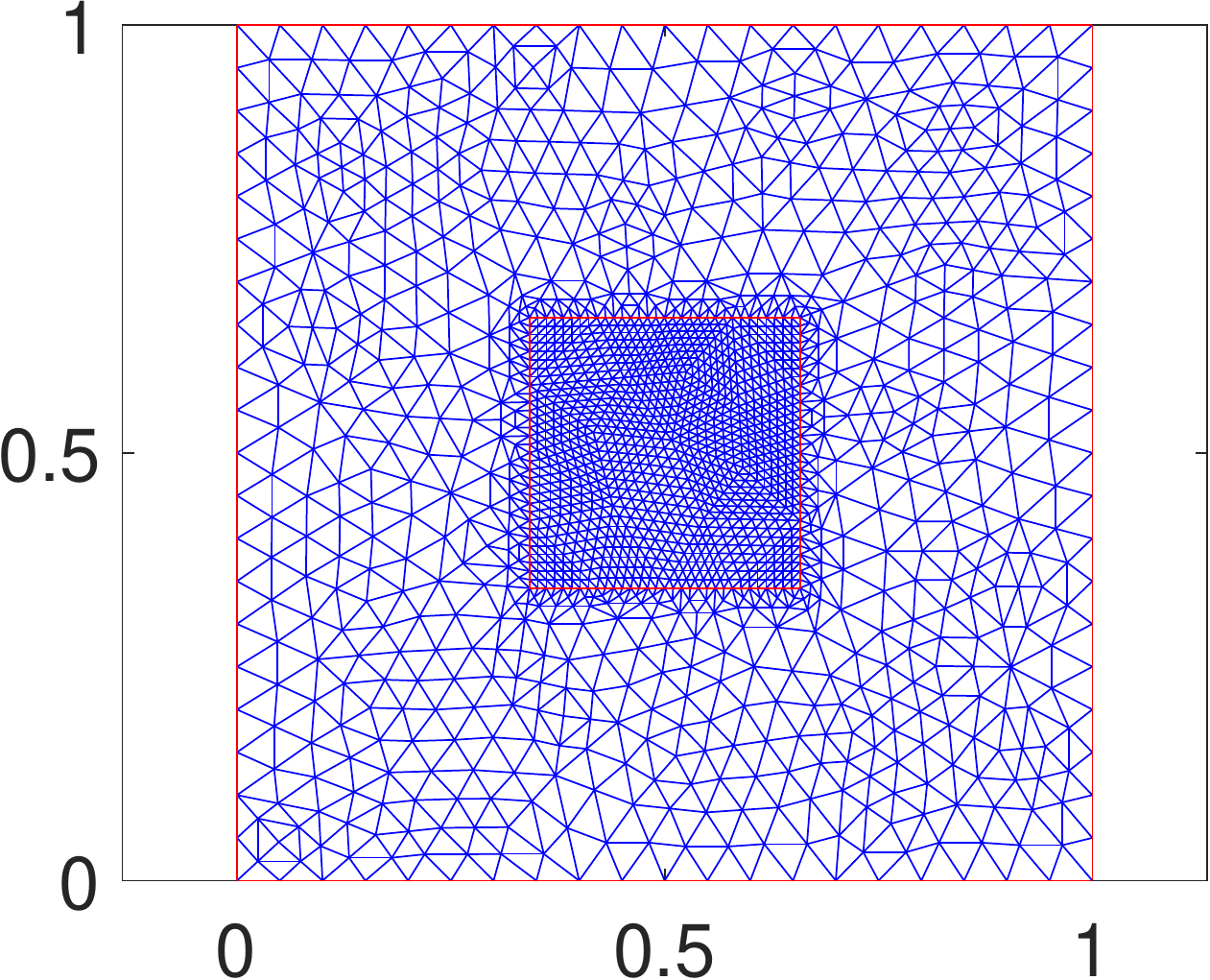}}\quad
\subfigure
{\includegraphics[width=0.45\textwidth]{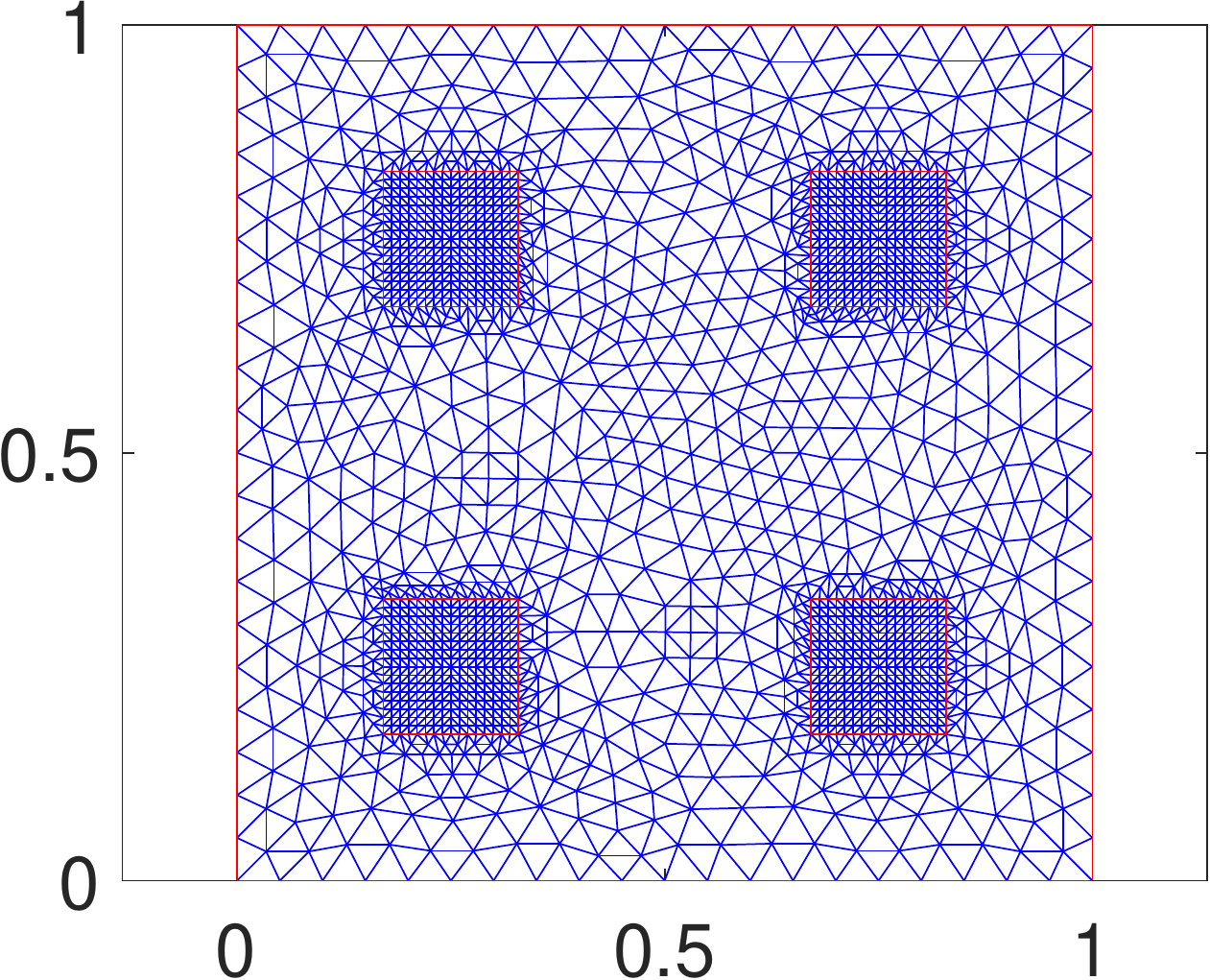}}
\\
\subfigure
{\includegraphics[width=0.45\textwidth]{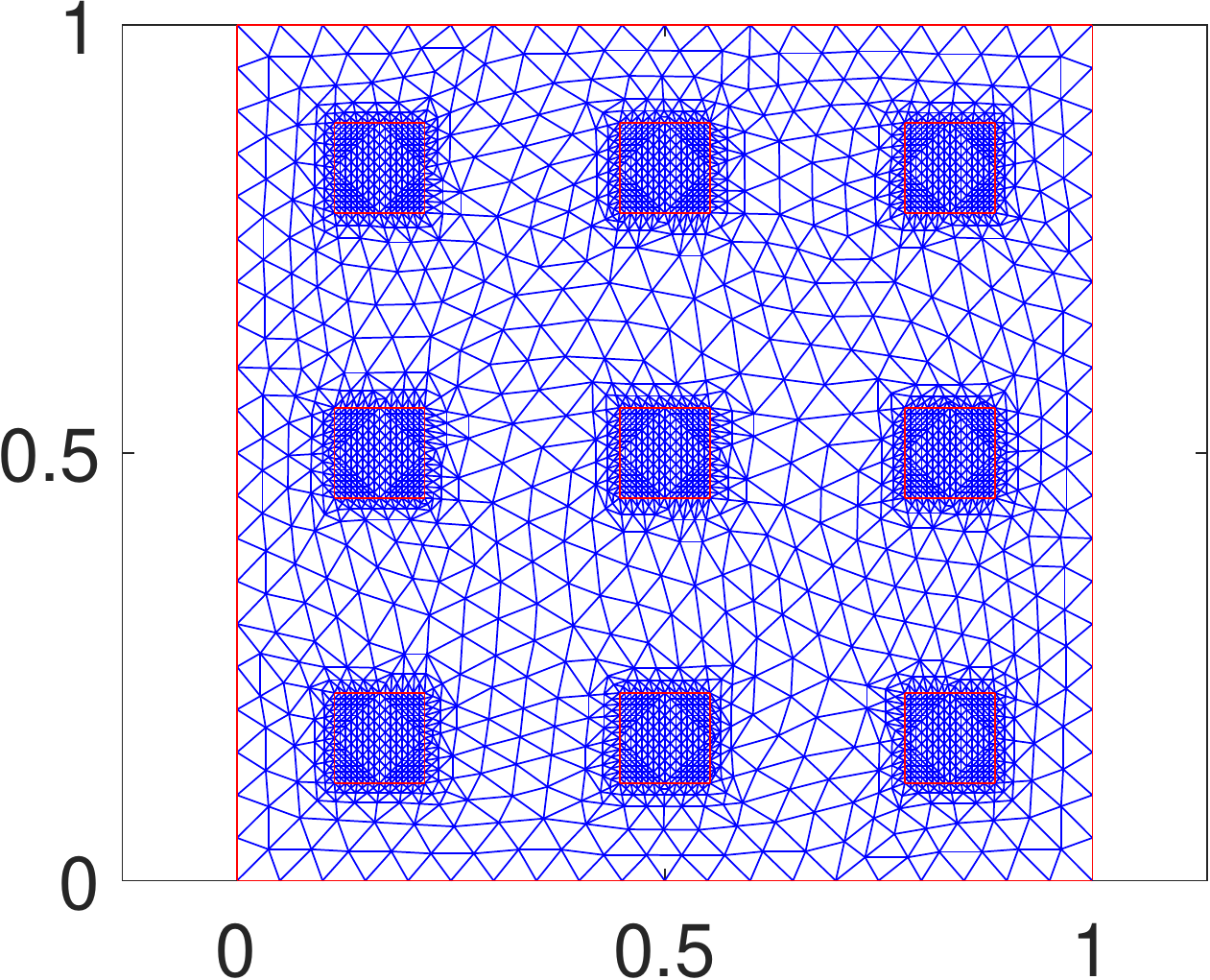}}\quad
\subfigure
{\includegraphics[width=0.45\textwidth]{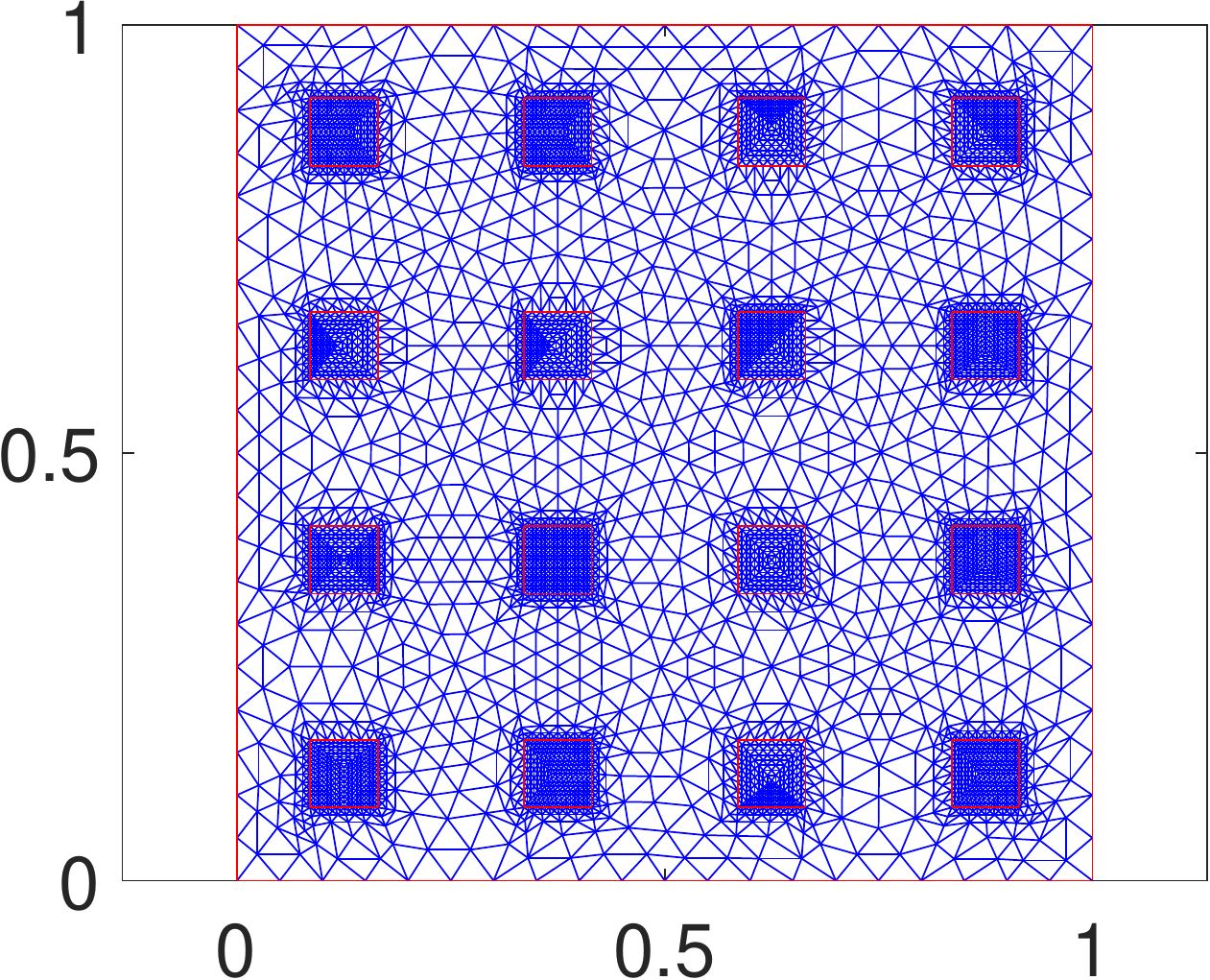}}
\caption{Location of the actuators. Cases~$M_\sigma\in\{1,4, 9,16\}$.}
\label{Fig:Mesh}
\end{figure}

\begin{example}\label{Exa:eq23h0}
We take as initial states the constant functions
\[
\widehat y_0(x)=\zeta_2\quad\mbox{and}\quad
y_0(x)=\zeta_3,
\]
and we take the vanishing external force
\begin{align}\notag
& h(t,x)=0.
\end{align}
Note that~$\widehat y_0(x)$ is an unstable equilibrium and~$y_0(x)$ is a stable equilibrium, for the free dynamics.   Thus,  our goal is to leave a stable state and to approach and stabilize an unstable one.
In Fig.~\ref{Fig:zCu_ic4CuLaInfnef1} we can see that the saturated feedback control is able to stabilize the error dynamics in case ~$C_u \geq \rme^{3.5}$.

\begin{figure}[ht]
\centering
\subfigure
{\includegraphics[width=0.45\textwidth]{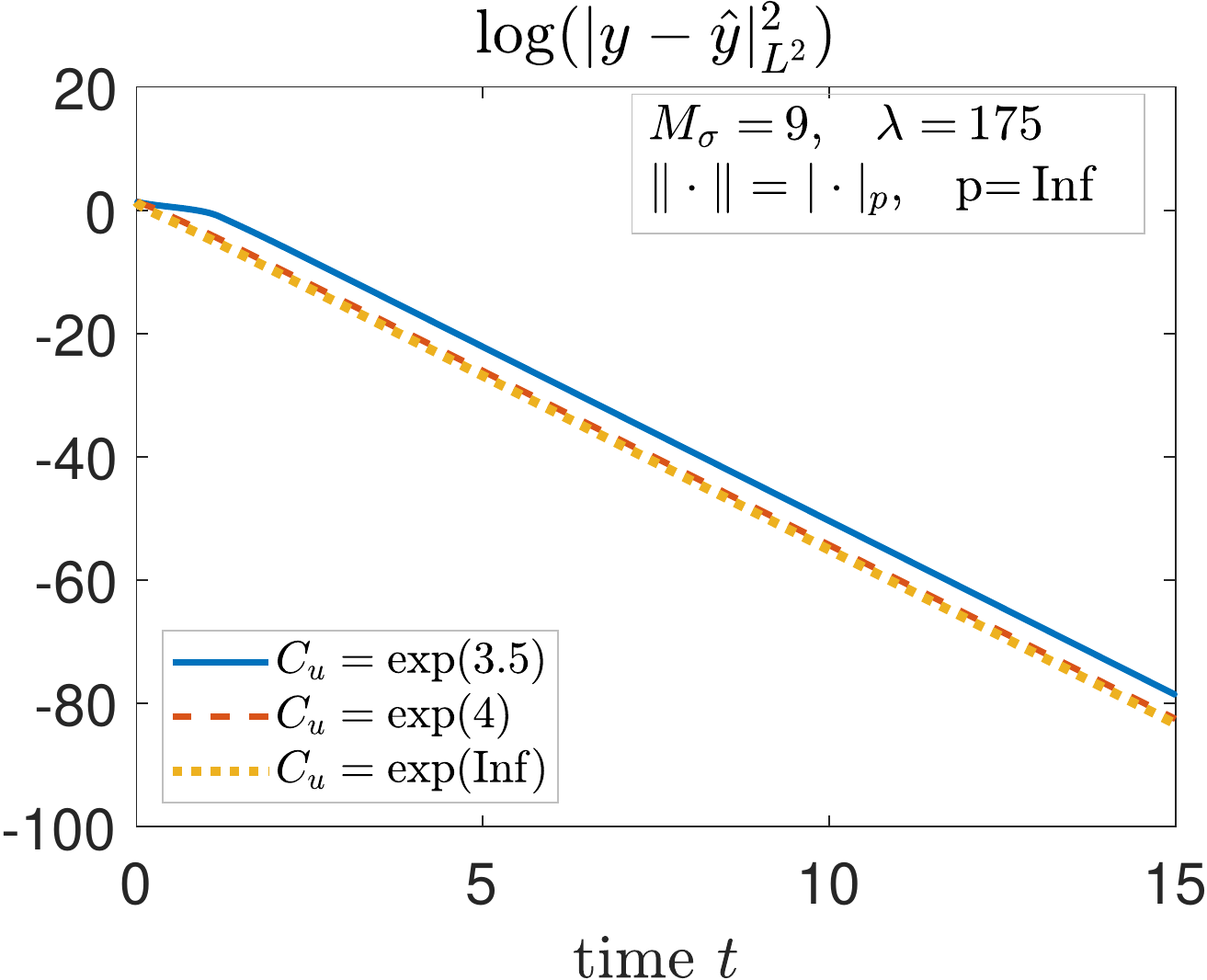}}
\quad
\subfigure
{\includegraphics[width=0.45\textwidth]{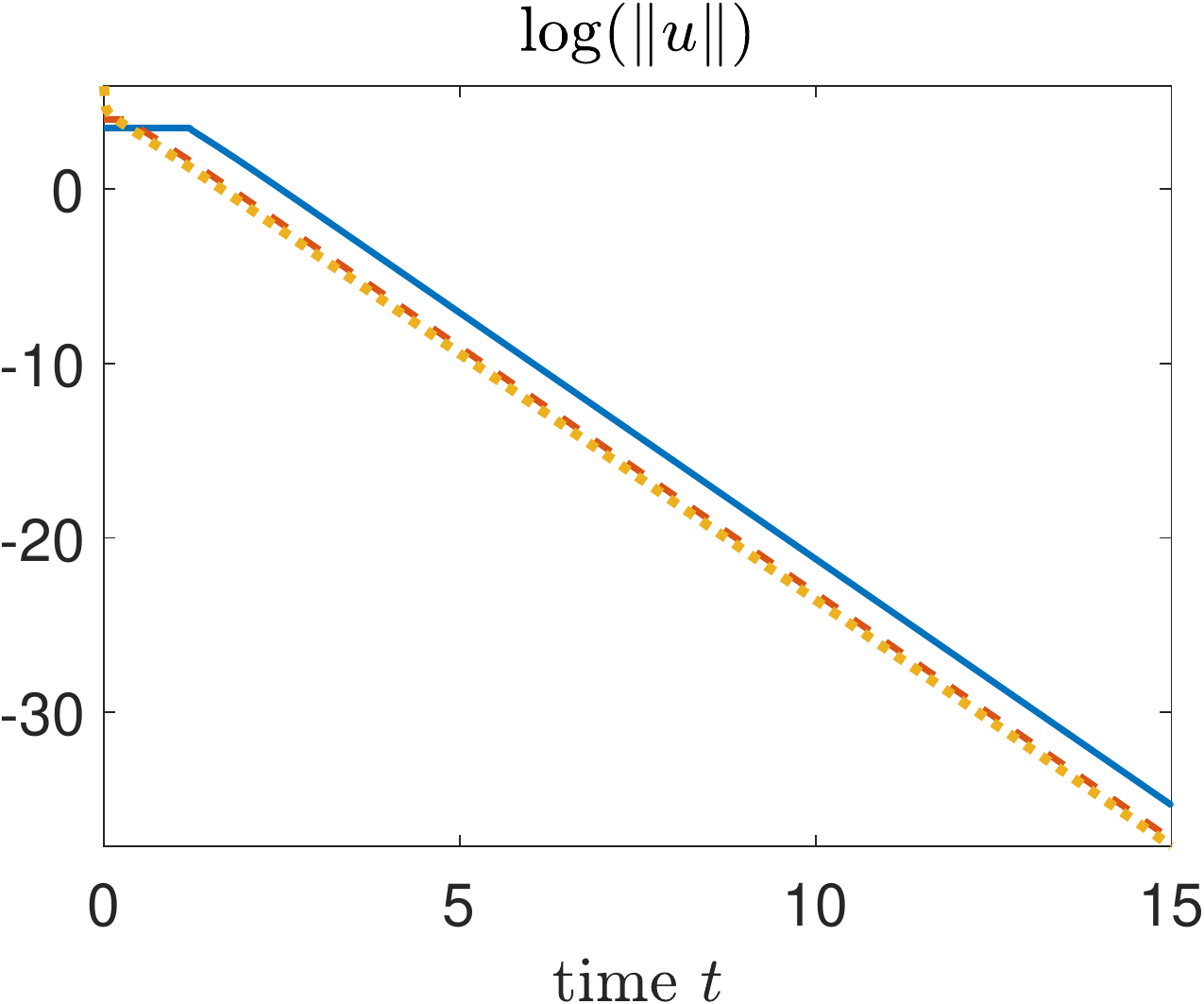}}
\caption{Norms of error and control. Large control constraint. (Ex.~\theexample)}
\label{Fig:zCu_ic4CuLaInfnef1}
\end{figure}

On the other hand,  for magnitudes~$C_u \leq \rme^{1}$, in Fig.~\ref{Fig:zCu_ic4CuSmInfnef1} we can see that the saturated control is not able to stabilize the error dynamics.
\begin{figure}[ht]
\centering
\subfigure
{\includegraphics[width=0.45\textwidth]{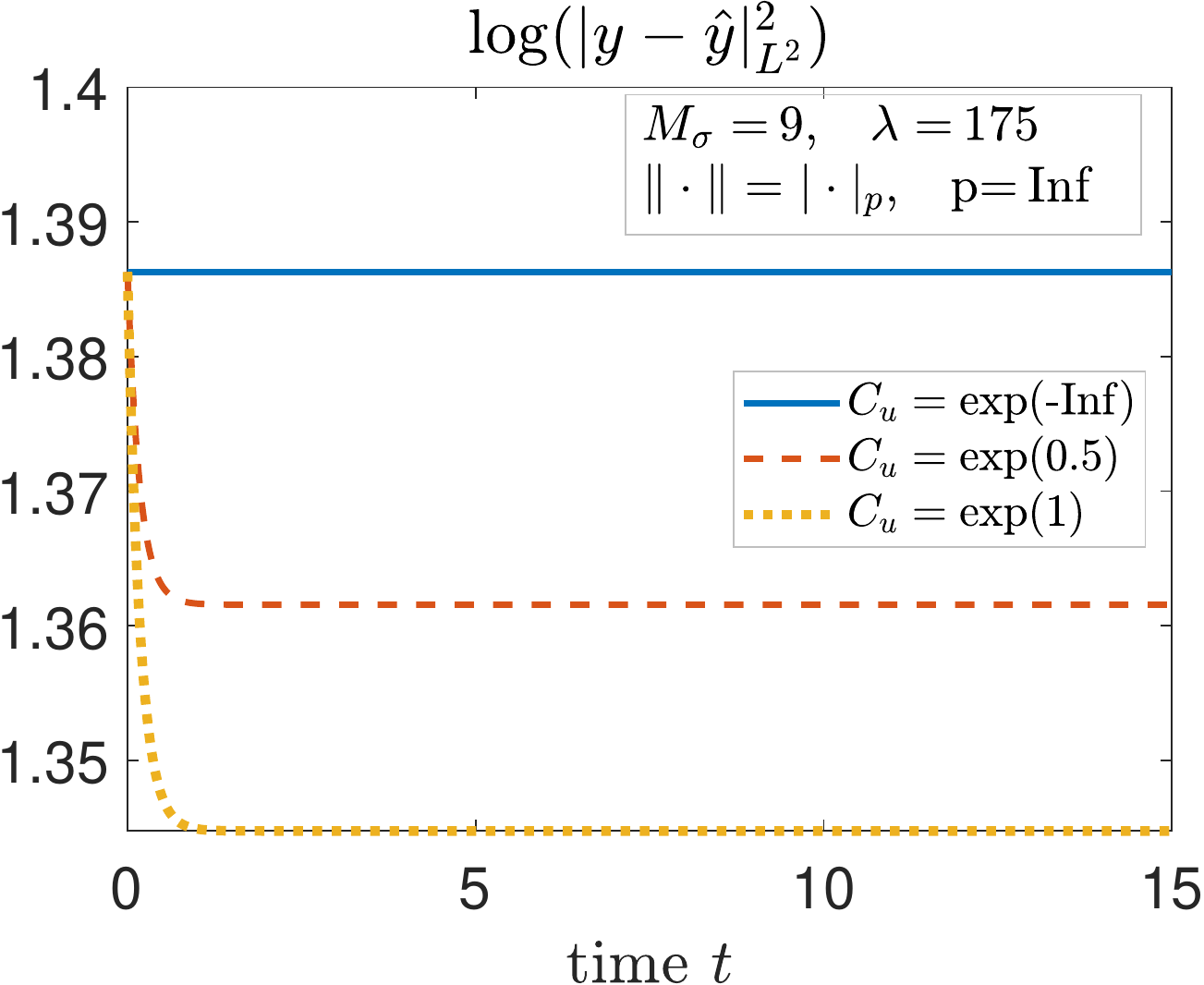}}
\quad
\subfigure
{\includegraphics[width=0.45\textwidth]{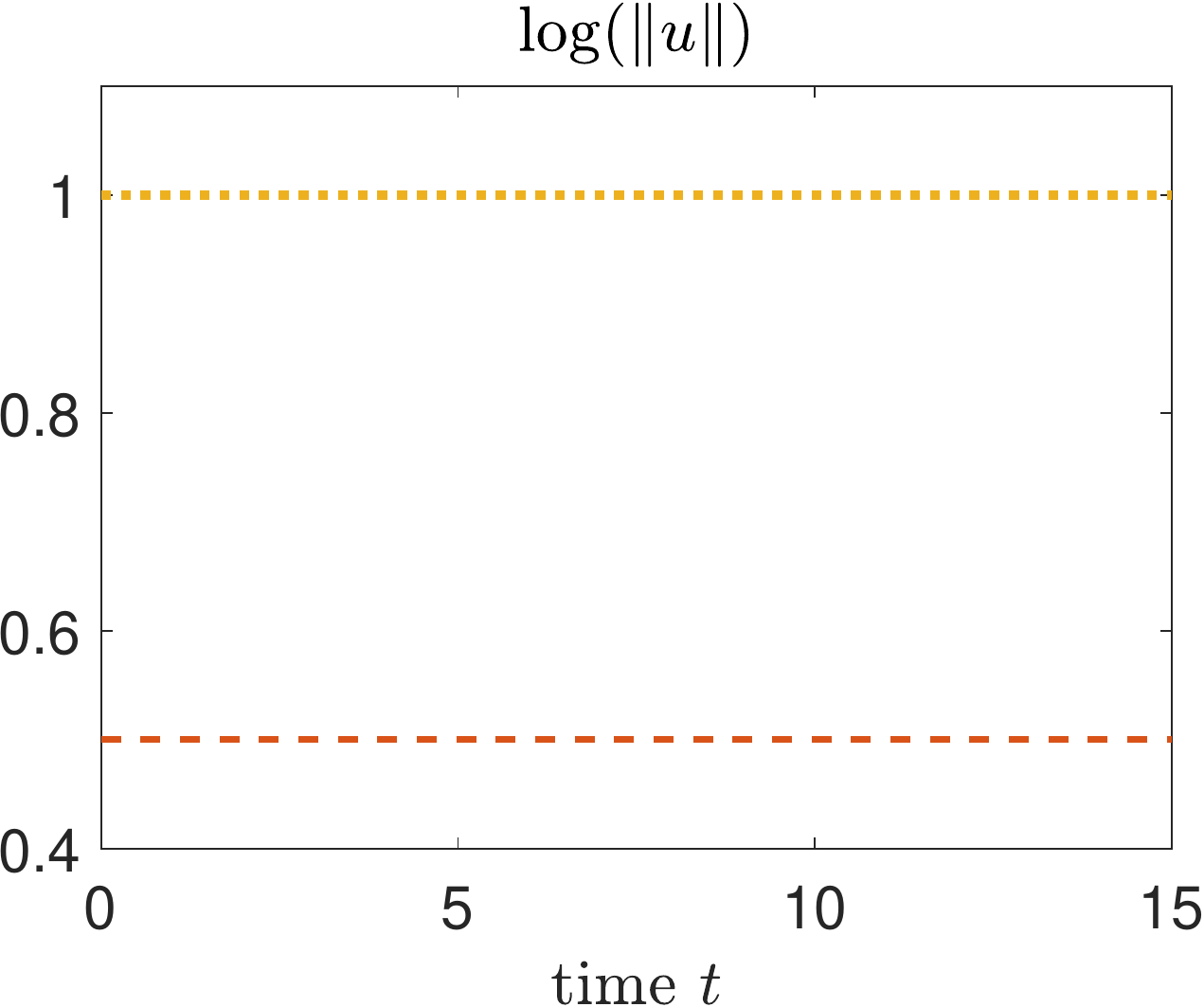}}
\caption{Norms of error and control. Small control constraint. (Ex.~\theexample)}
\label{Fig:zCu_ic4CuSmInfnef1}
\end{figure}

\end{example}

\begin{example}\label{Exa:eq31hper}
 We compare the performance of the control generated by Algorithm~\ref{RHA} (RHC) with that of the explicitly given saturated  control.   For Algorithm ~\ref{RHA},   we set $T = 1.25$ and  $\delta = 0.5$ and run the algorithm until the final computational time $T_{\infty}$  for two control cost parameters $\beta \in\{ 10^{-3}, 10^{-5}\}$.  As initial states we take the constant functions
\[
\widehat y_0(x)=\zeta_3\quad\mbox{and}\quad
y_0(x)=\zeta_1,
\]
and as external force we take the time-periodic function
\begin{align}\label{h=per}
& h(t,x)=\tfrac12\indf_{\{s\ge0\mid\norm{\sin(6s)}{\bbR}>\frac12\}}(t)\indf_{\{w\in\Omega\mid\norm{w}{\bbR}^2<\frac12\}}(x).
\end{align}
The initial states are stable equilibria of the free dynamics in the case of a vanishing external force, $h=0$.
 Here, we are taking a nonzero time-periodic external forcing, which induces the asymptotic periodic-like behavior  shown in Fig.~\ref{Fig:haty_ic5} for the norm of the targeted trajectory~$\widehat y$.
From Figs.~\ref{Fig05}--\ref{FigInf}, we can see that the saturated feedback control is able to stabilize the error dynamics for~$C_u \geq  \rme^{1.5}$,  whereas RHC is stabilizing for $ C_u \geq \rme^{1}$.  This can be seen from Fig.~\ref{Fig1},  where the  saturated control  fails to track~$\widehat{y}$, while  RHC succeeds.  For the case $C_{u} =e^{0.5}$,  it can be seen in Fig.~\ref{Fig05} that,  none of the control laws is able to track~$\widehat y$.
As illustrated in Figs.~\ref{Fig05}--\ref{FigInf},   the saturated control is active for a longer interval compared to RHC.  Further,  observing the plots related to $ \norm{y-\widehat{y}}{L^2((0,T_{\infty}),L^2)}^2$,   we can see that the influence of $\beta$ is only recognizable for $C_u \geq e^2$.

The value of the performance index function
\[  J_{T_{\infty}}(u;0,y_0, \widehat{y}) = \norm{y-\widehat{y}}{L^2((0,T_{\infty}),L^2)}^2+\beta \norm{u}{L^2((0,T_{\infty}),\mathbb{R}^{M_{\sigma}})}^2,  \]
for different control laws is reported in Table~\ref{tableJ}. As we would expect, in all the cases,  RHC delivers better results than the saturated controls concerning the value of $J_{\infty}$.    We can also see that as
$C_{u}$ is getting larger the performance of the saturated control and RHC are getting closer to each other (except the case $C_{u} = e^{\infty}$).

\begin{table}[htpb]
\begin{center}
\scalebox{1.0}{%
\begin{tabular}{ |c||c|c |c| c|c| }
 \hline
\backslashbox{Control}{$(C_u, T_{\infty})$}& $(e^{0.5},25)$  &$(e^{1},20)$&$(e^{1.5},10)$&$(e^{2},7)$&$(e^{\infty},5)$\\
 \hline\hline
 RHC $\Bigl.\beta = 10^{-3}\Bigr.$   & $202.47$   & $90.376$&  $30.238$ &$17.554$&    $25.827$  \\
 SatCon $\lambda =175$& $203.04$& $152.93$& $34.226$ & $17.922$&   $30.787$\\
\hline\hline
 RHC $\Bigl.\beta = 10^{-5}\Bigr.$   & $201.86$    &$89.497$&   $29.415$ &$15.459$&   1.0466  \\
 SatCon $\lambda =175$& $202.47$ &  $151.73$ &$ 33.439$ &$16.729$&   $1.0973$ \\
 \hline
\end{tabular}}
\end{center}
\caption{ The value of $J_{T_{\infty}}$.  (Ex.~\theexample) }
\label{tableJ}
\end{table}

\begin{figure}[ht]
\centering
{\includegraphics[width=0.45\textwidth]{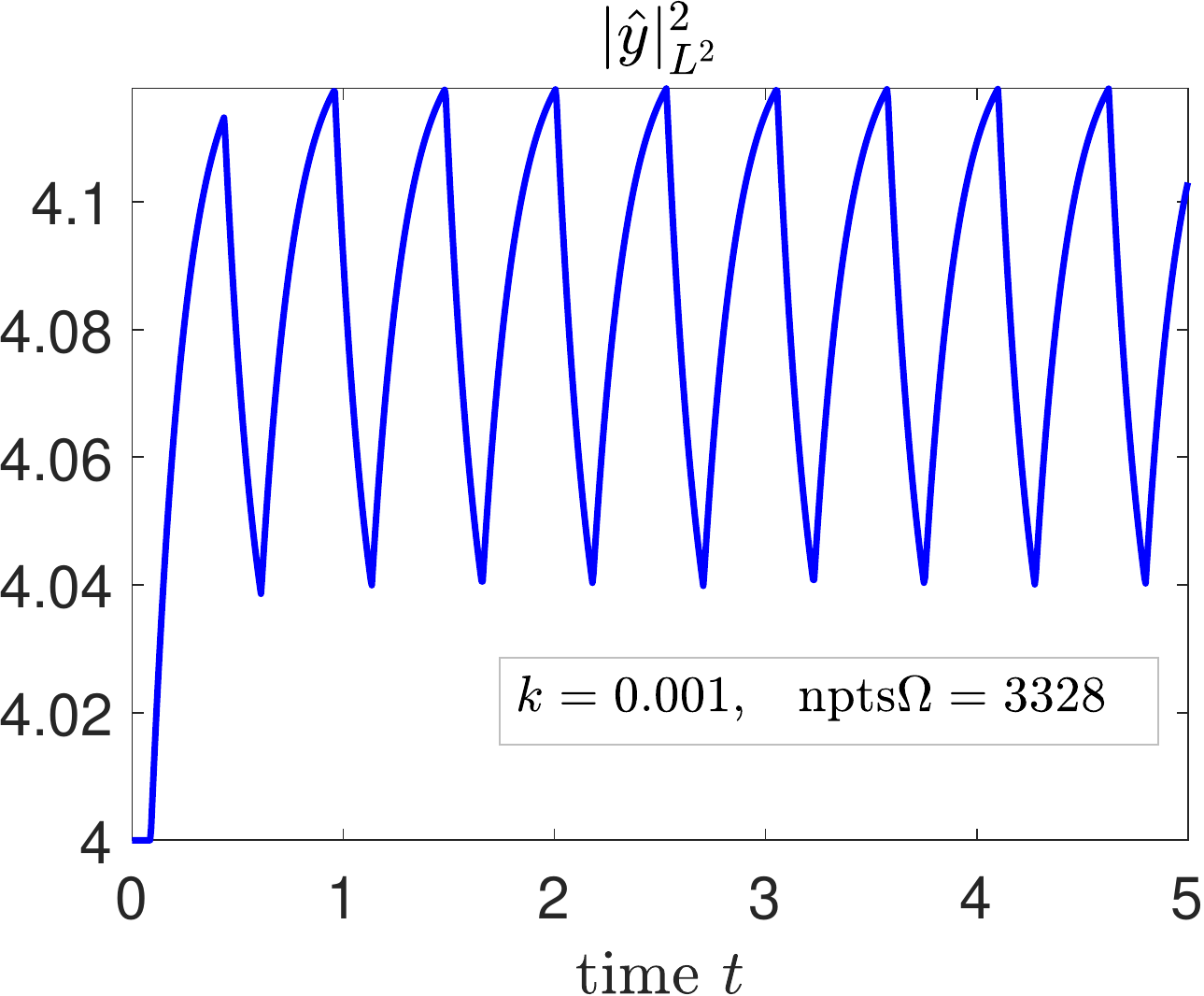}}
\caption{Norm of targeted trajectory. (Ex.~\theexample)}
\label{Fig:haty_ic5}
\end{figure}

Note that for large time the logarithm of the norm of the error~$y-\widehat y$ stays close to a small value, approximately around~$-35$. This can be explained due to the accuracy/precision used in the numerical computations, in fact that value is relatively close to the standard Matlab precision
${\tt eps}\approx10^{-16}$ we have used;  indeed~$\rme^{-35}\approx 6\times10^{-16}$.

Recall that we are also computing the targeted trajectory~$\widehat y$, whose values are then used to compute the controlled trajectory~$y$. We cannot expect that the computational errors associated with the solution of the two systems will cancel each other.  Hence,  the
aforementioned behavior for the norm of~$y-\widehat y$ is consistent.
\begin{figure}[ht]
\centering
\subfigure
{\includegraphics[width=0.45\textwidth]{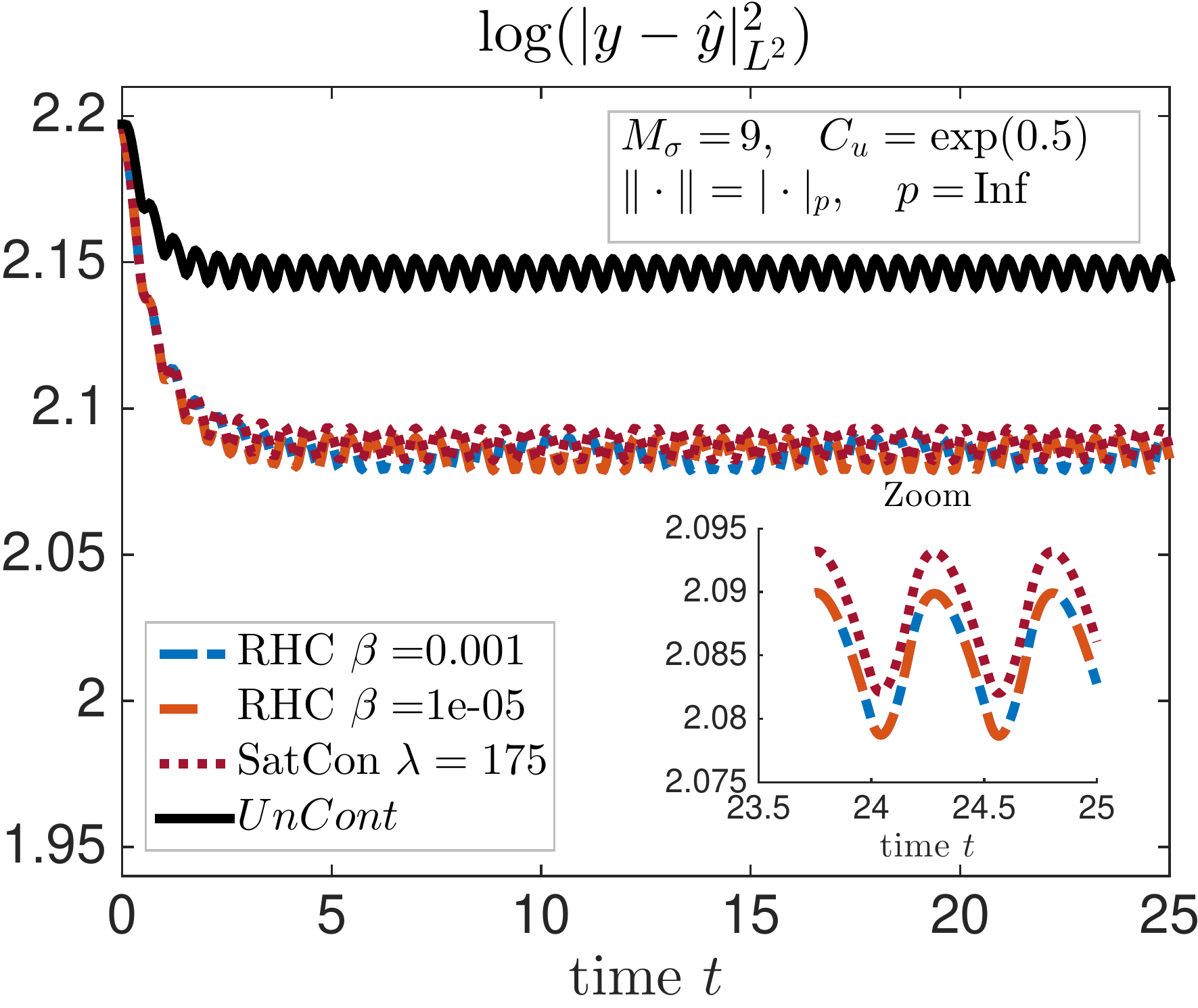}}
\quad
\subfigure
{\includegraphics[width=0.45\textwidth]{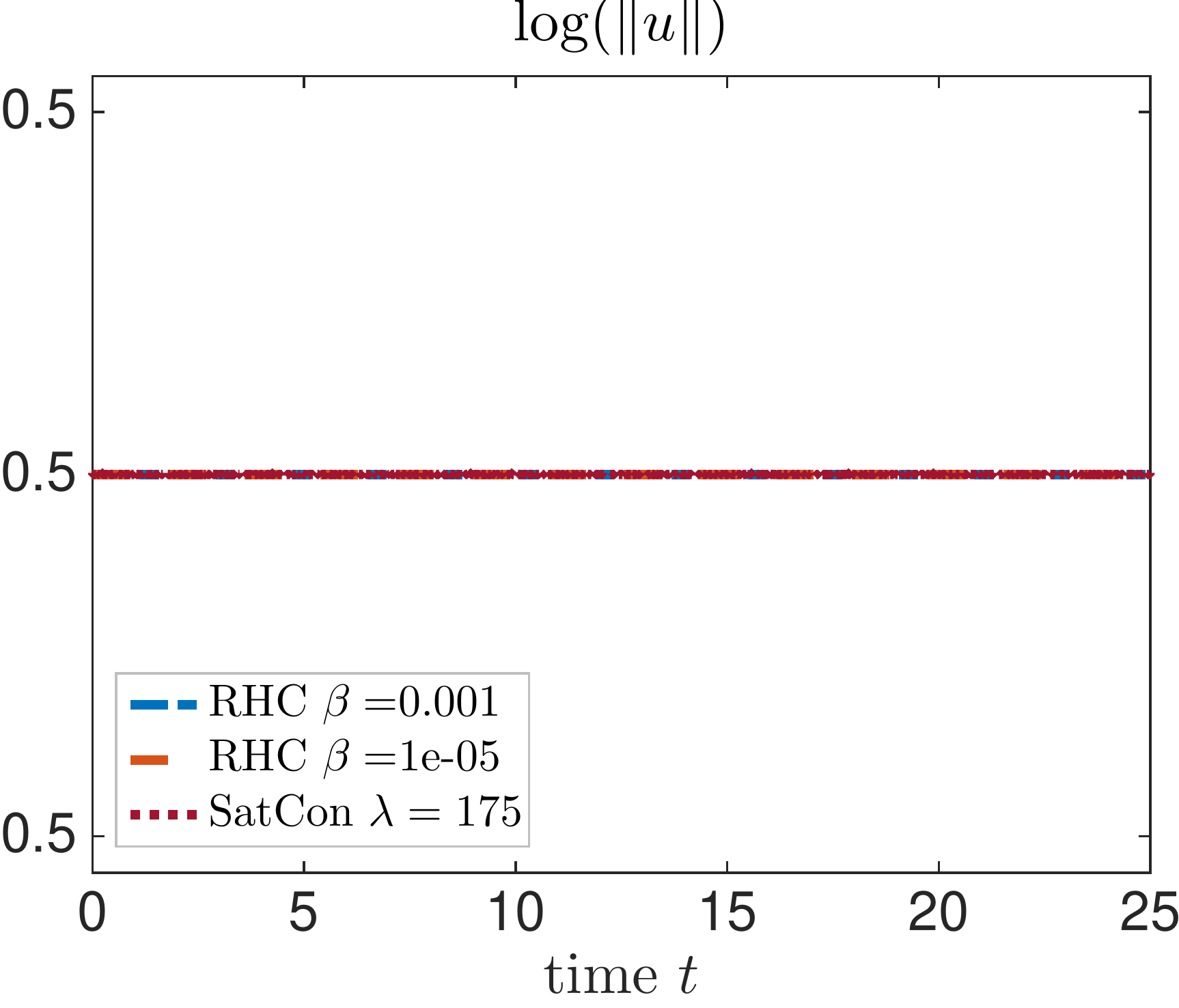}}
\caption{Norms of error and control for $(C_u,  T_{\infty}) =( e^{0.5},  25)$. (Ex.~\theexample)}
\label{Fig05}
\end{figure}

\begin{figure}[ht]
\centering
\subfigure
{\includegraphics[width=0.45\textwidth]{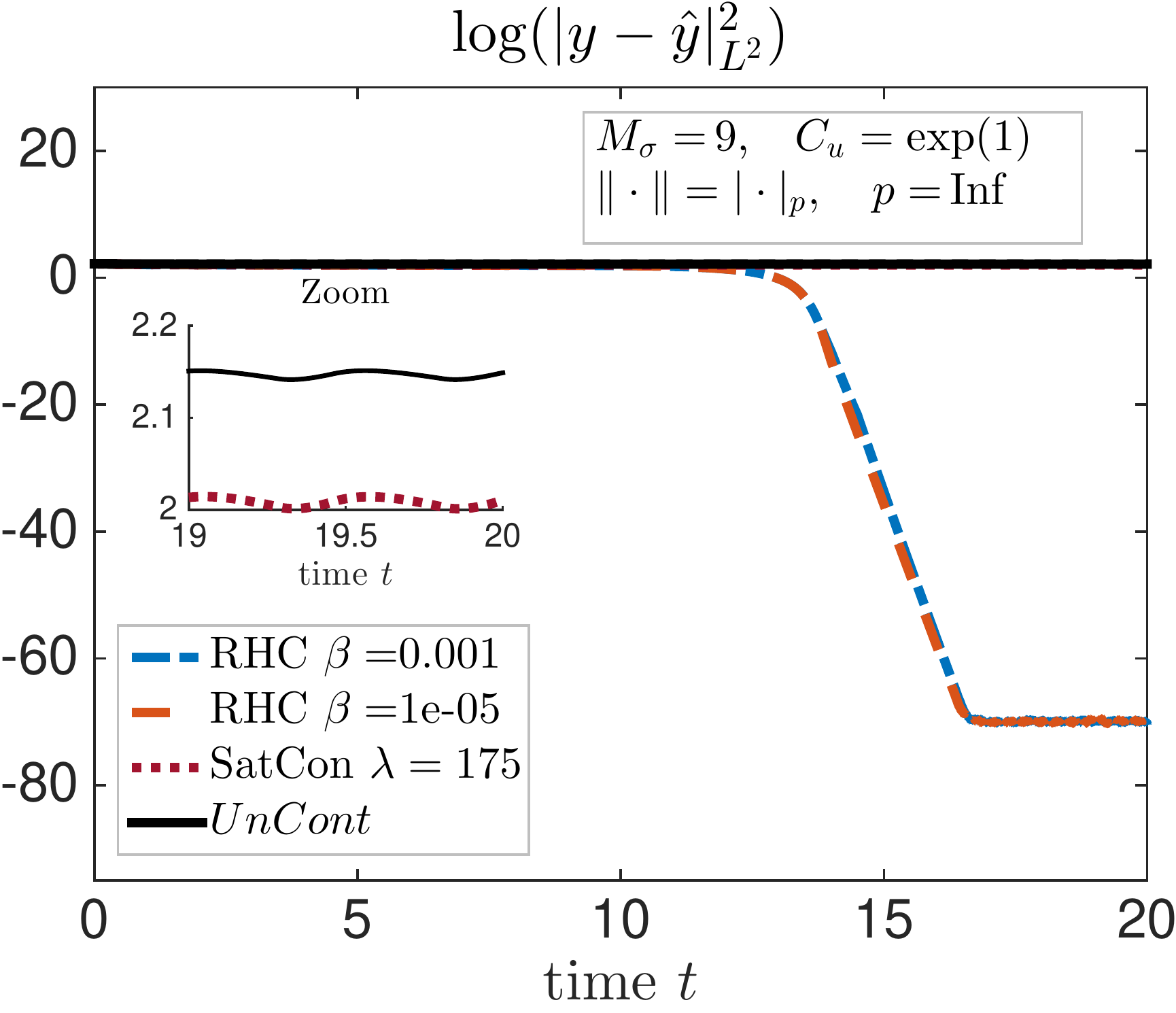}}
\quad
\subfigure
{\includegraphics[width=0.45\textwidth]{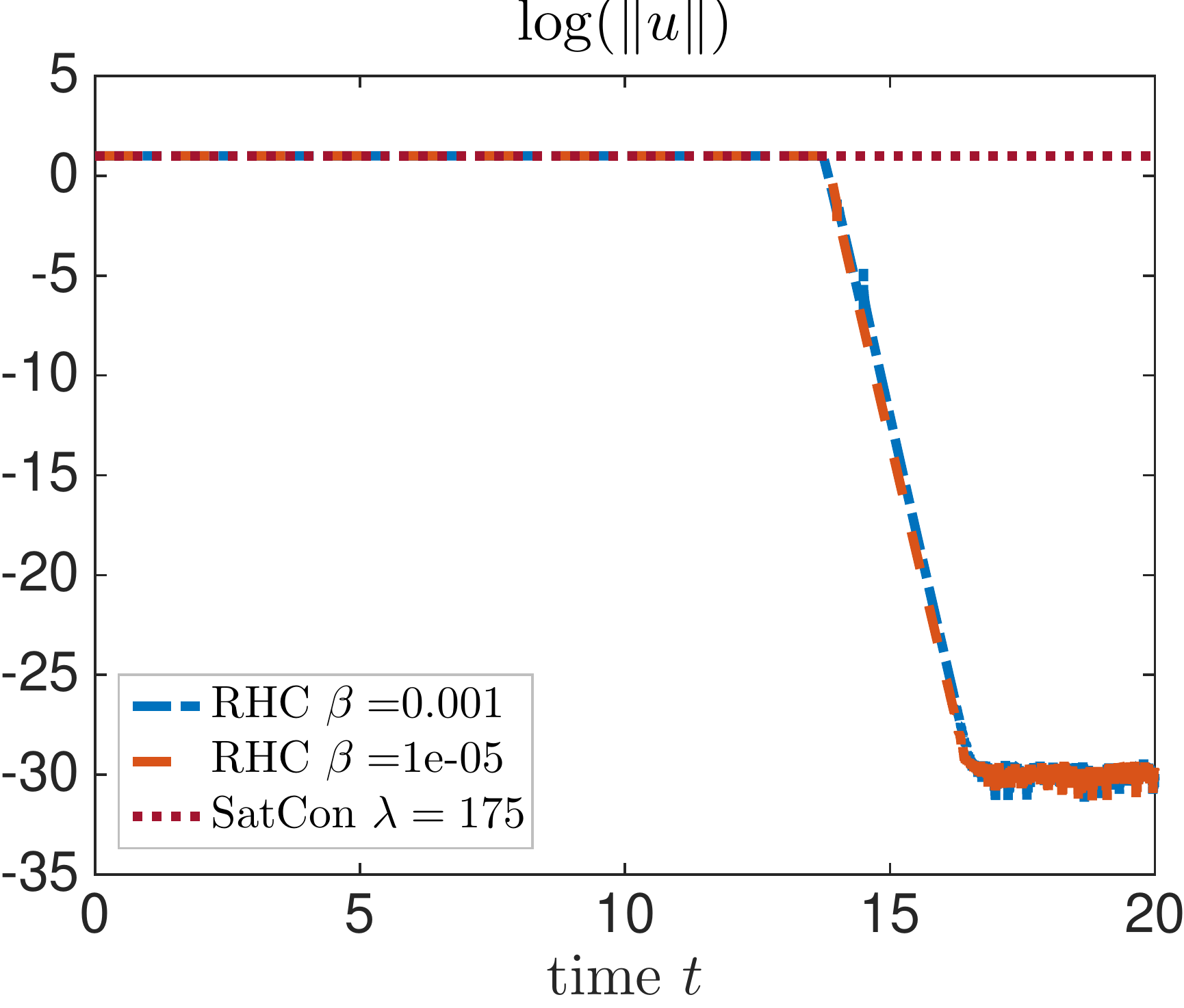}}
\caption{Norms of error and control for $(C_u,  T_{\infty}) =( e^{1},  20)$. (Ex.~\theexample)}
\label{Fig1}
\end{figure}

\begin{figure}[ht]
\centering
\subfigure
{\includegraphics[width=0.45\textwidth]{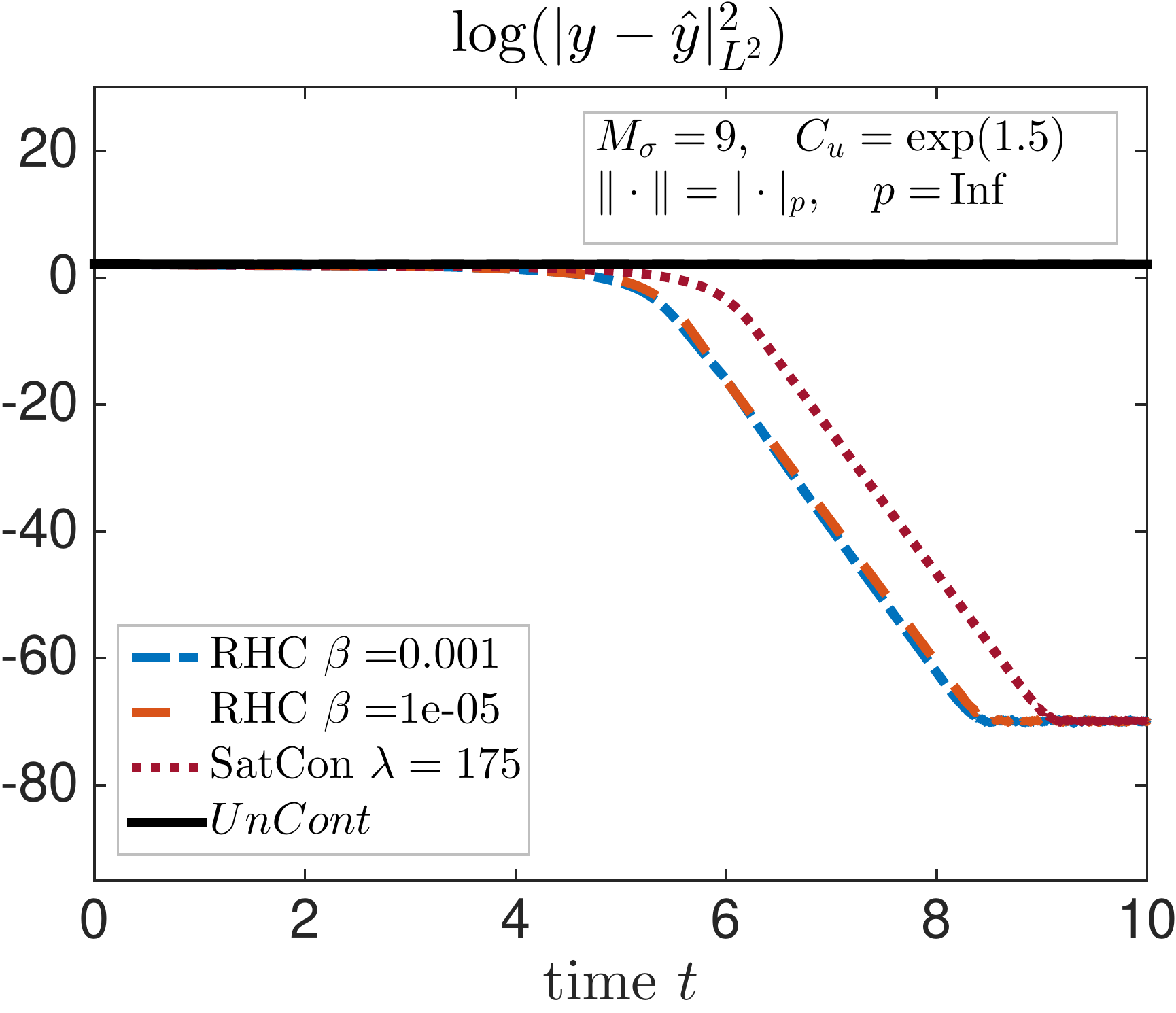}}
\quad
\subfigure
{\includegraphics[width=0.45\textwidth]{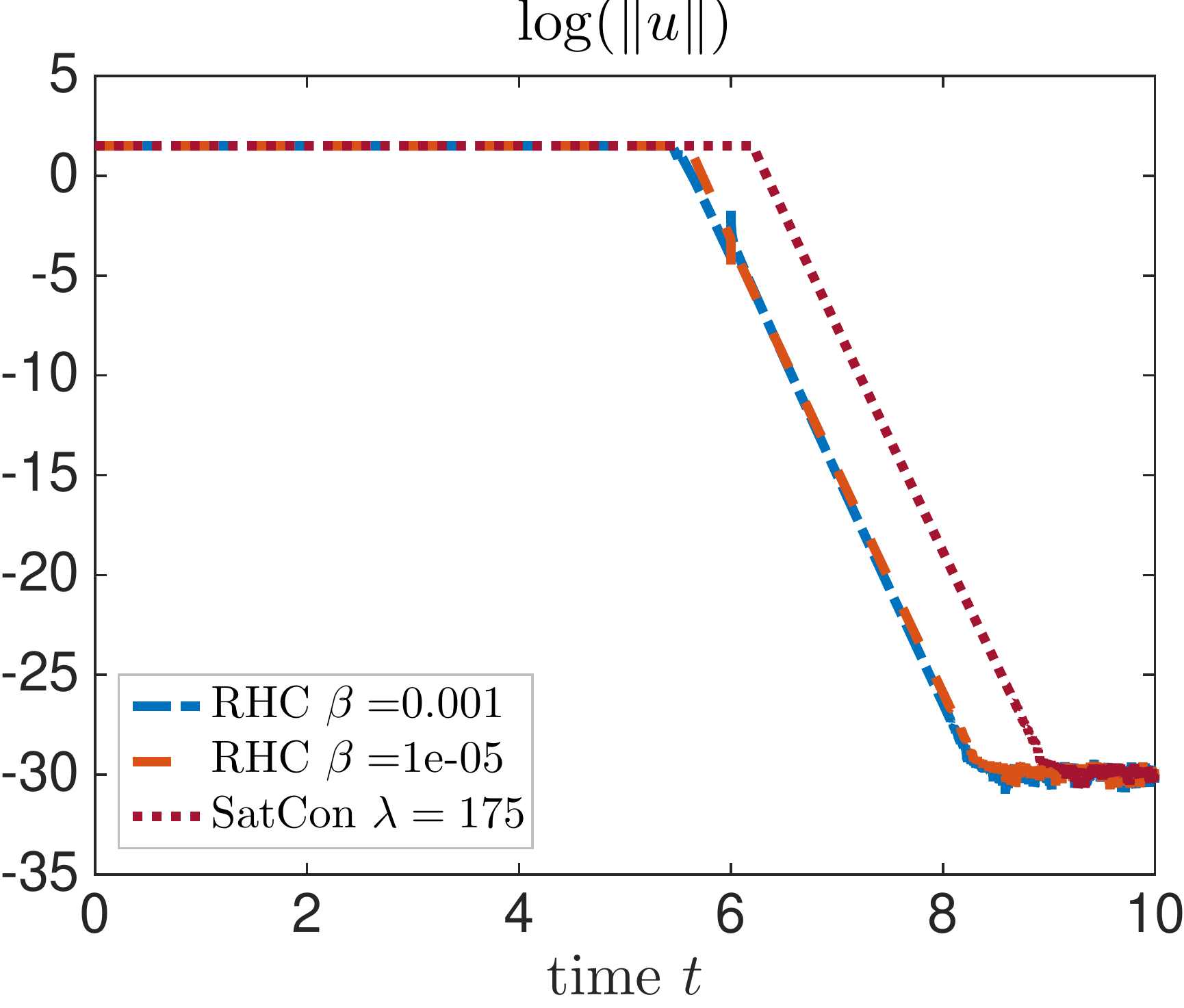}}
\caption{Norms of error and control. $(C_u,  T_{\infty}) =( e^{1.5},  10)$. (Ex.~\theexample)}
\label{Fig15}
\end{figure}
\end{example}

\begin{figure}[ht]
\centering
\subfigure
{\includegraphics[width=0.45\textwidth]{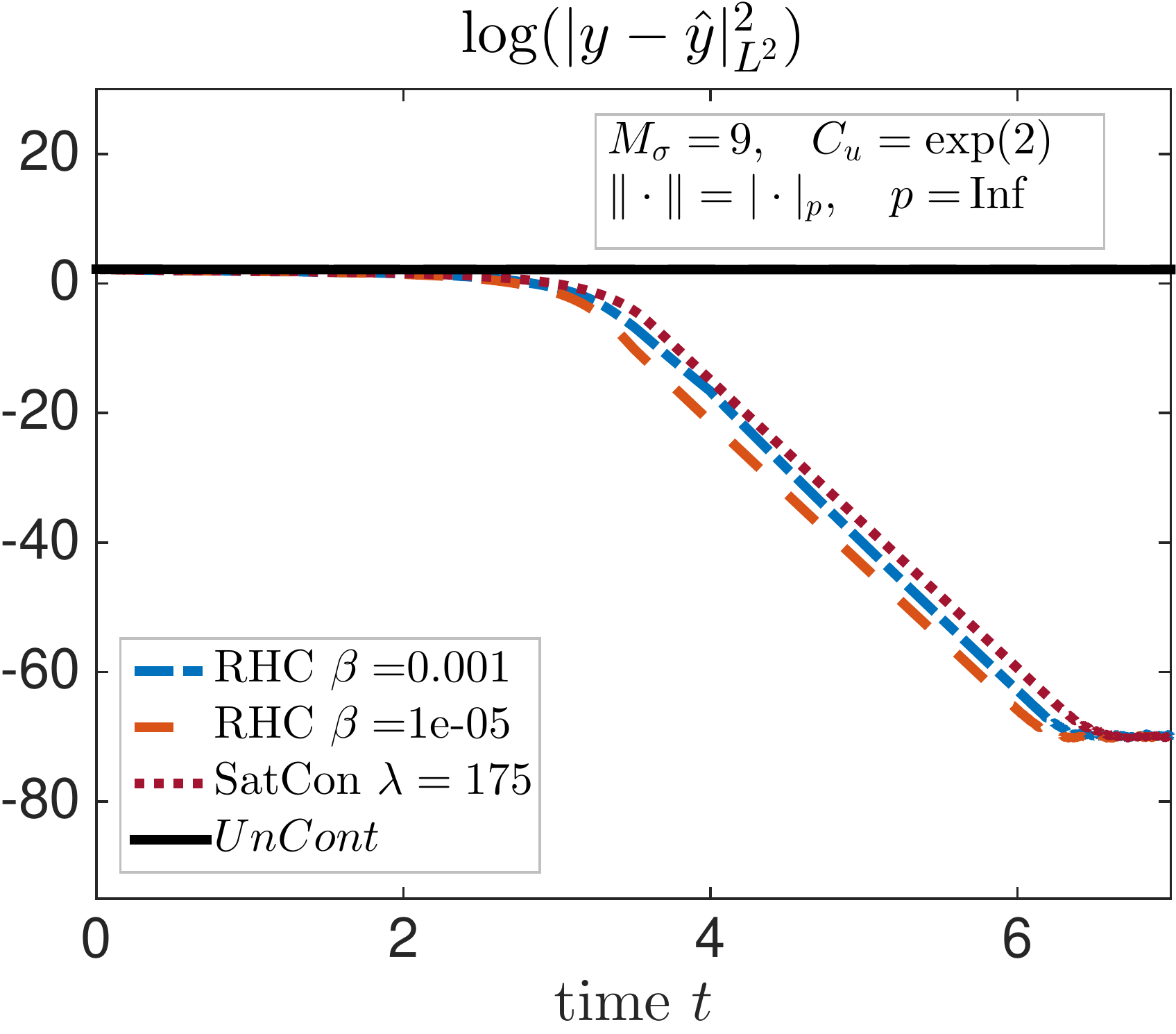}}
\quad
\subfigure
{\includegraphics[width=0.45\textwidth]{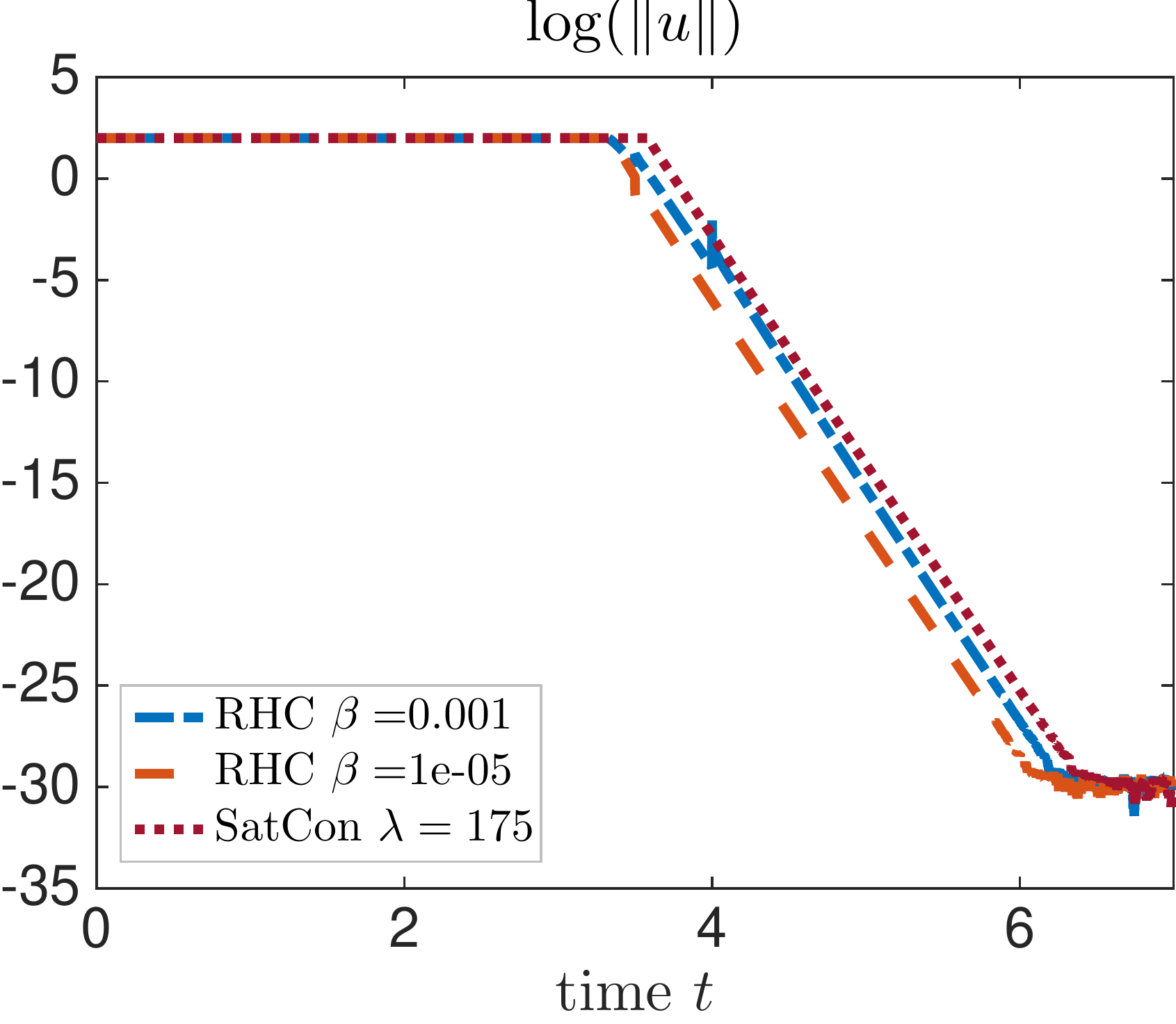}}
\caption{Norms of error and control. $(C_u,  T_{\infty}) =( e^{2},  7)$. (Ex.~\theexample)}
\label{Fig2}
\end{figure}

\begin{figure}[ht]
\centering
\subfigure
{\includegraphics[width=0.45\textwidth]{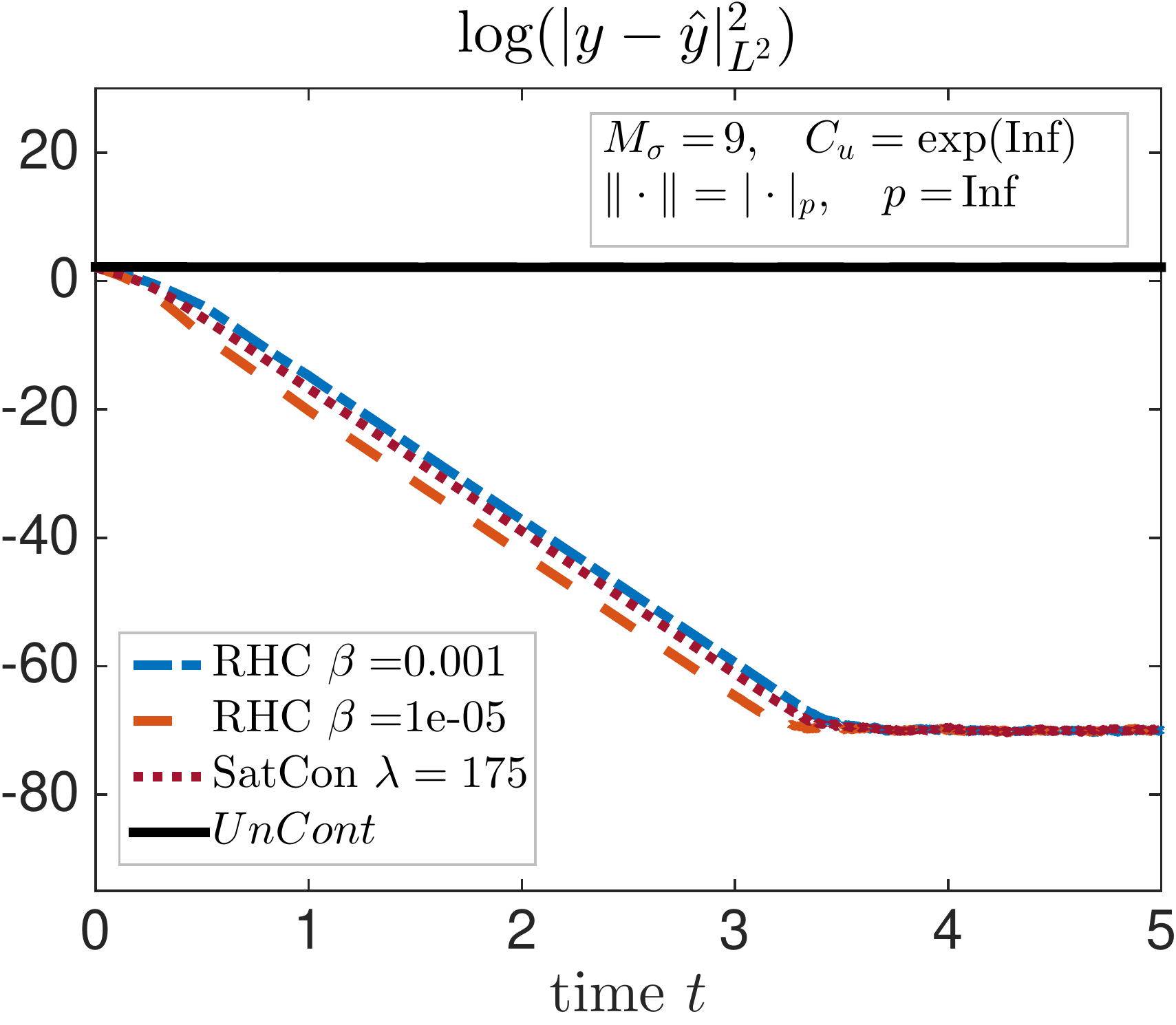}}
\quad
\subfigure
{\includegraphics[width=0.45\textwidth]{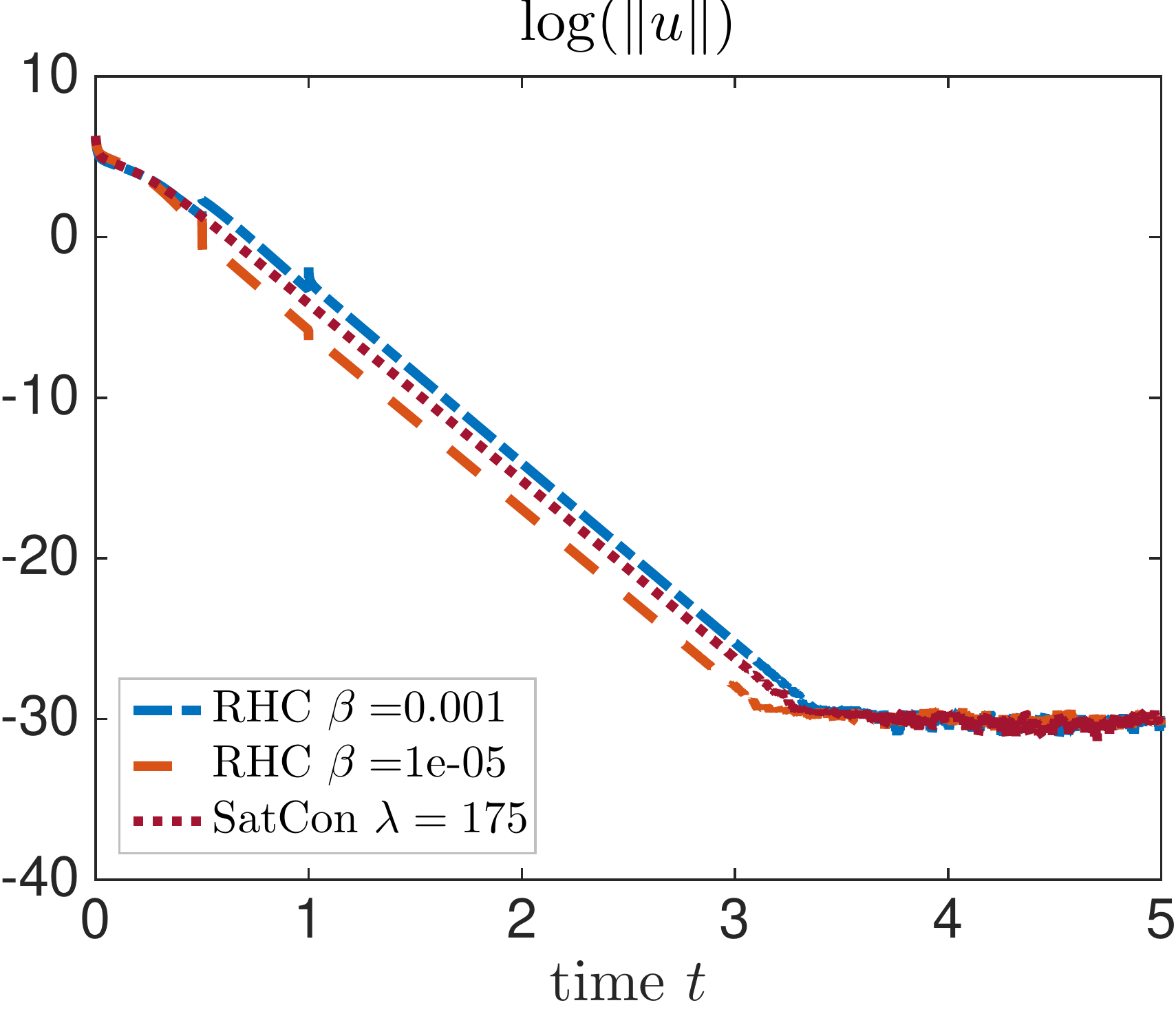}}
\caption{Norms of error and control. $(C_u,  T_{\infty}) =( e^{\infty},  5)$. (Ex.~\theexample)}
\label{FigInf}
\end{figure}

\begin{example}\label{Exa:rrhper}
Now we take as initial conditions
\[
\widehat y_0(x)= 10-20x_1x_2;
         \qquad y_0(x)= -10x_1+x_2.
\]
Their norms as well as the norm of the initial error $z_0=y_0-\widehat y_0$ are large when compared to those in Examples~\ref{Exa:eq23h0} and~\ref{Exa:eq31hper}. The external forcing is taken as~\eqref{h=per}.
In Fig.~\ref{Fig:haty_ic3}, we can see that the norm of the targeted trajectory~$\widehat y$ decreases fast as long as it is large  (cf. Sect.~\ref{sS:normdeclargeerror}). Asymptotically it exhibits again a periodic-like behavior near~$2$.
\begin{figure}[ht]
\centering
{\includegraphics[width=0.45\textwidth]{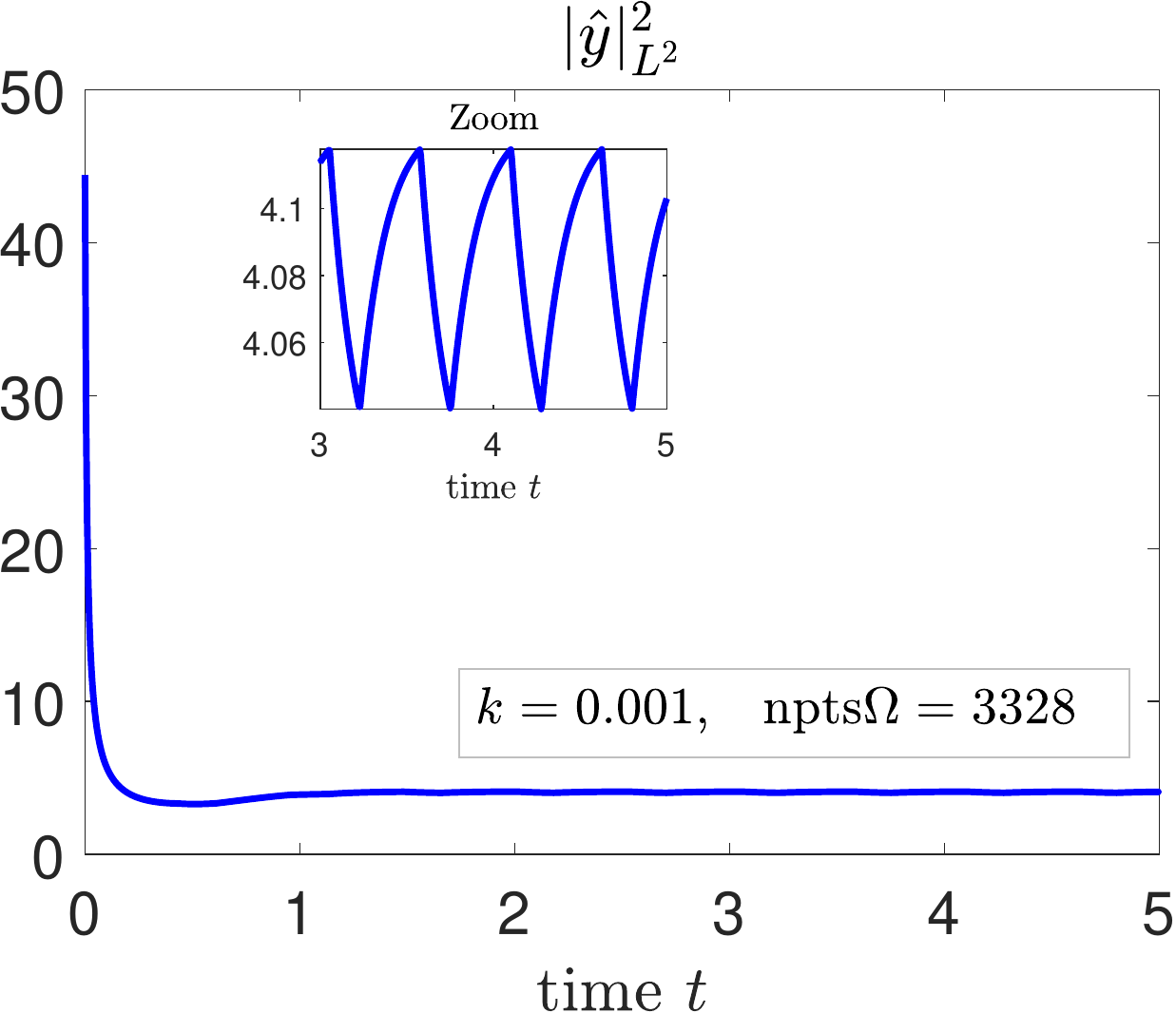}}
\caption{Norm of targeted trajectory. (Ex.~\theexample)}
\label{Fig:haty_ic3}
\end{figure}

Comparing Figs.~\ref{Fig:zCu_ic3CuLaInf} and~\ref{Fig:zCu_ic3CuSmInf}, we see again that we achieve the stability of the controlled error dynamics if, and only if, the magnitude~$C_u$ of the control constraint is large enough.
\begin{figure}[ht]
\centering
\subfigure
{\includegraphics[width=0.45\textwidth]{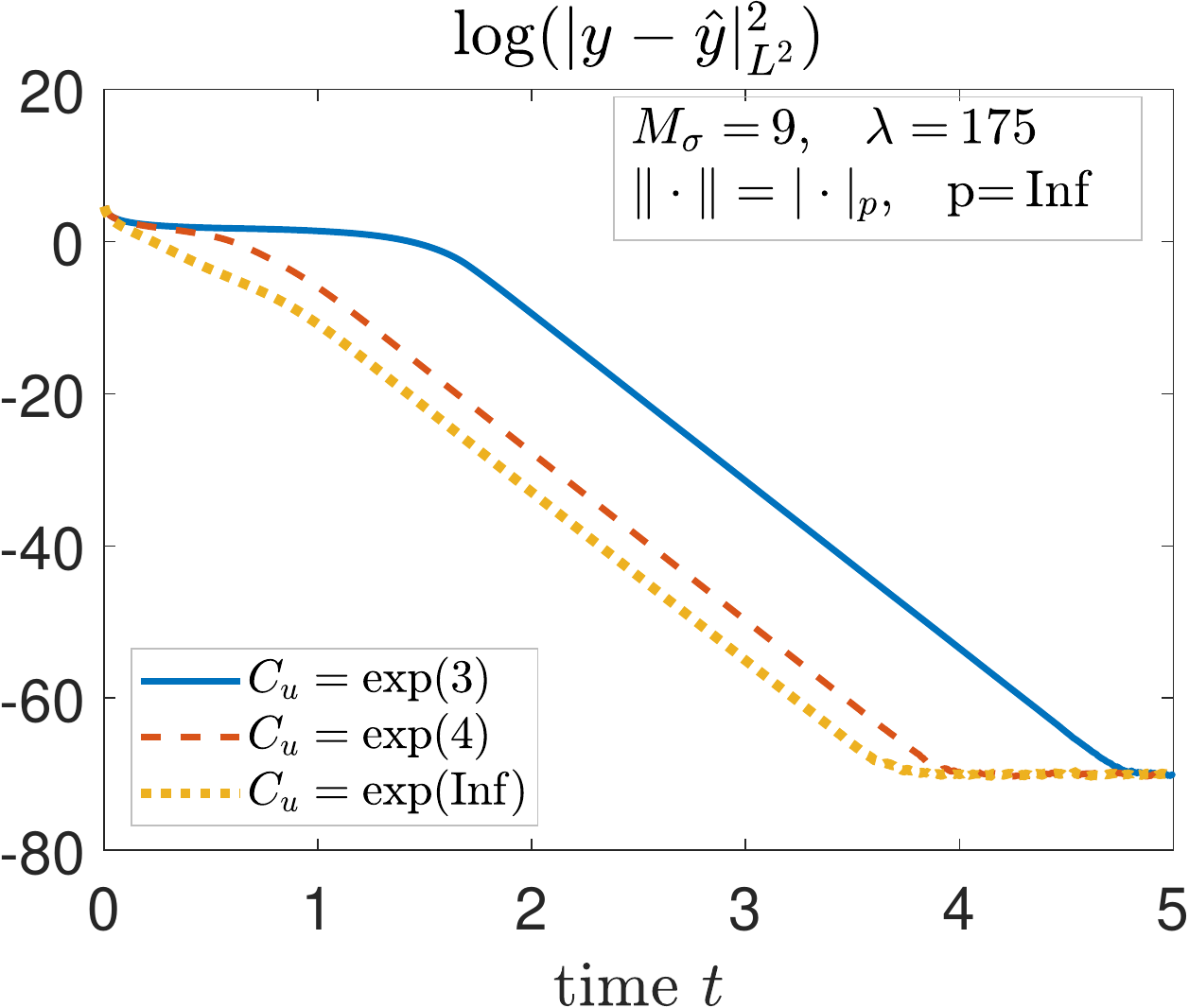}}
\quad
\subfigure
{\includegraphics[width=0.45\textwidth]{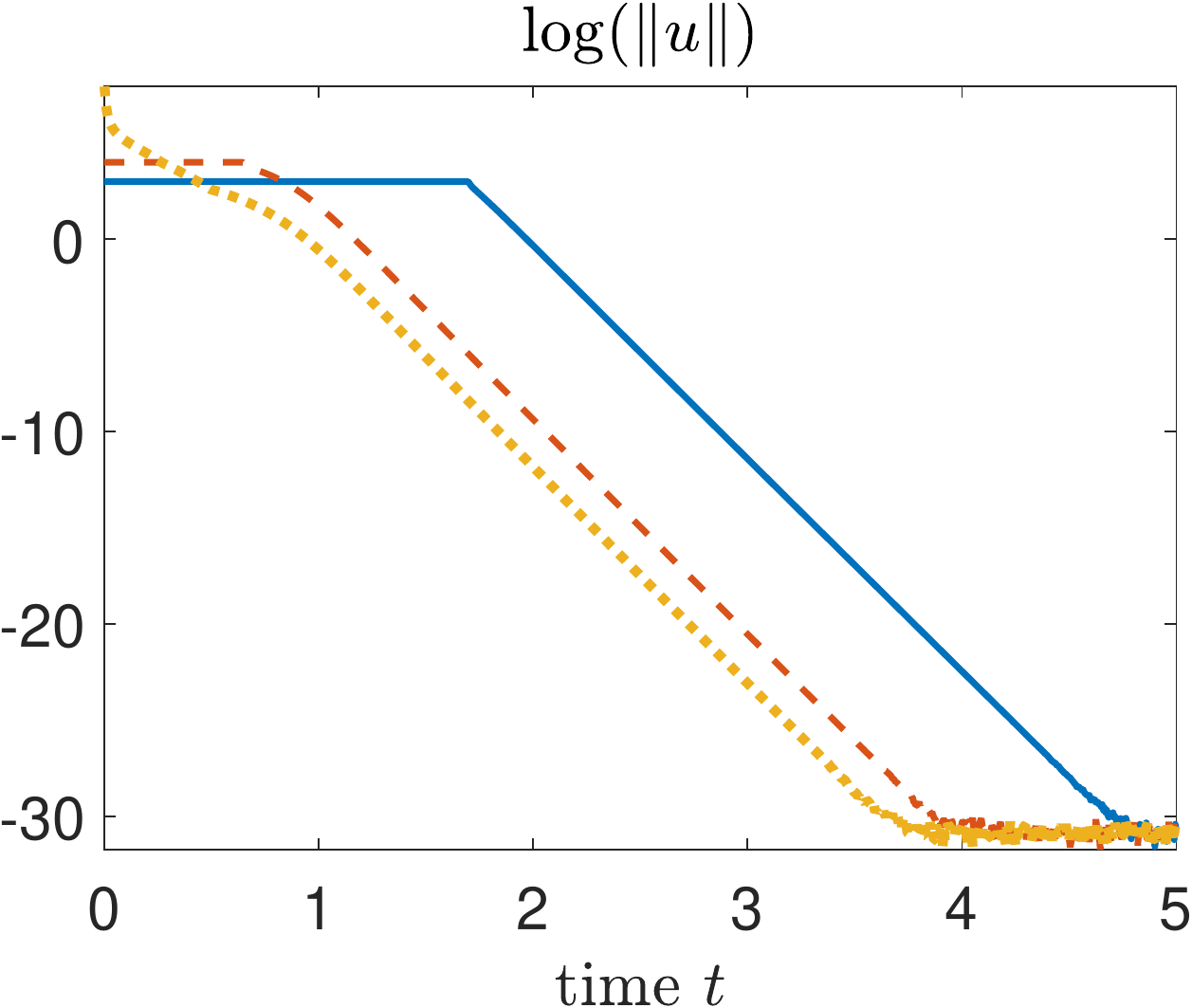}}
\caption{Norms of error and control. Large control constraint. (Ex.~\theexample)}
\label{Fig:zCu_ic3CuLaInf}
\end{figure}

\begin{figure}[ht]
\centering
\subfigure
{\includegraphics[width=0.45\textwidth]{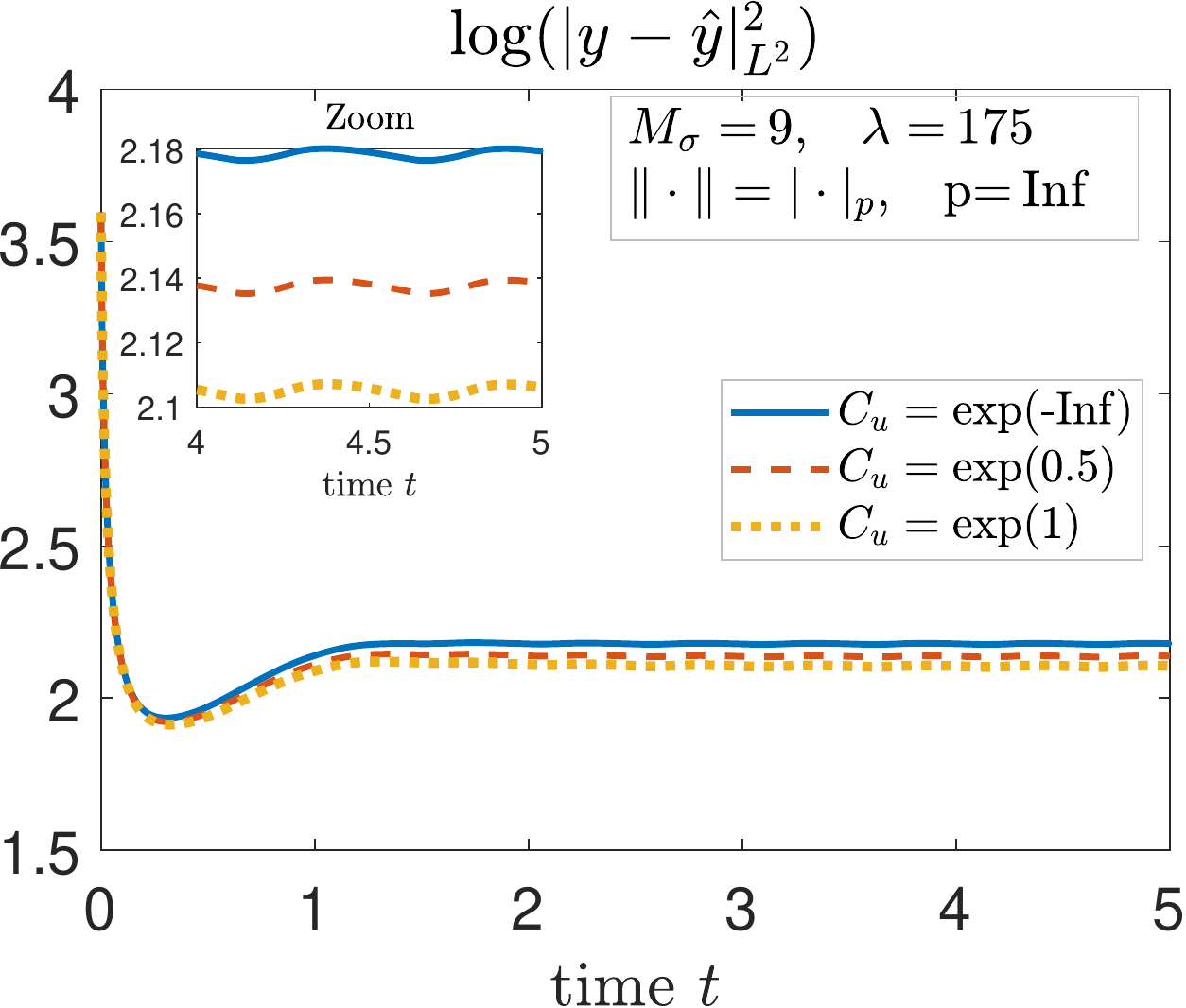}}
\quad
\subfigure
{\includegraphics[width=0.45\textwidth]{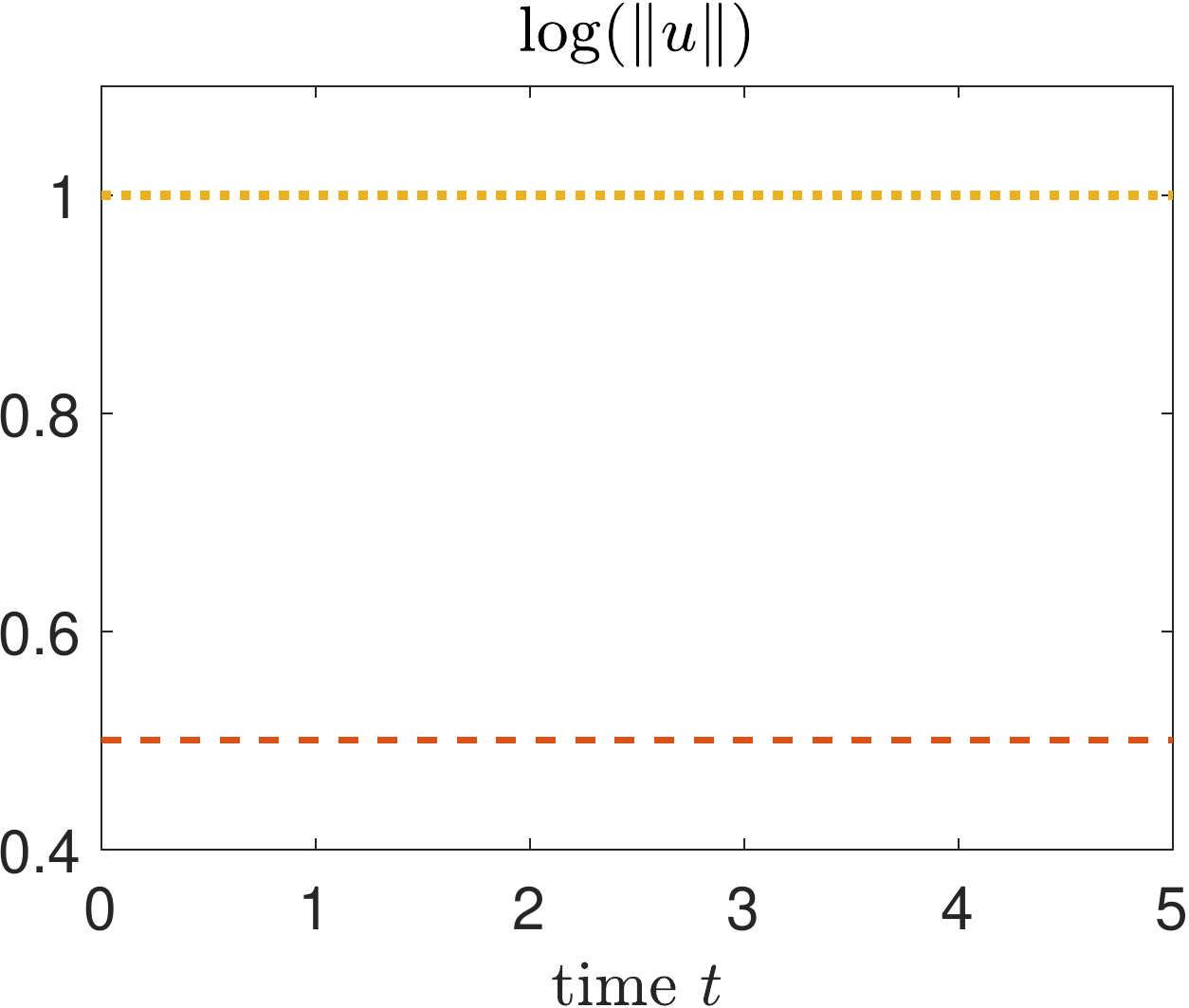}}
\caption{Norms of error and control. Small control constraint. (Ex.~\theexample)}
\label{Fig:zCu_ic3CuSmInf}
\end{figure}

\end{example}

\begin{example}\label{Exa:unc} We take the initial states and external force again as in Example~\ref{Exa:rrhper}.
But, instead of investigating the role played by the control constraint~$C_u$, we focus on  parameters~$M$ and~$\lambda$ determining the feedback law.%
 In Fig.~\ref{Fig:zCu_ic3CuLaInf} we see that by increasing the control constraint~$C_u$ we approach the behavior of the unconstrained limit case~$C_u=+\infty$.
To verify that an arbitrary exponential decrease rate~$\mu$ can be achieved with a sufficiently large~$C_u=C_u(\mu)$, it is enough to show that this rate can be achieved with the unconstrained feedback. Indeed, in Fig.~\ref{Fig:zMlamL} we confirm that by increasing~$M$ and~$\lambda$ we reach larger exponential decrease rates.
This is consistent with our theoretical result which says that we can achieve an arbitrary large exponential decrease rate~$\mu$. We recall the idea of the proof: for small time the exponential decrease rate  is guaranteed for large initial errors, with norm larger than a suitable constant~$D=D(\mu)$; for such~$D$ we choose~$M$ and $\lambda$ large enough to achieve such exponential decrease rate with the unconstrained control;
finally we choose~$C_u=C_u(D(\mu))$ large enough so that the constraint is inactive for an error norm smaller than~$D$.

\begin{figure}[ht]
\centering
\subfigure
{\includegraphics[width=0.45\textwidth]{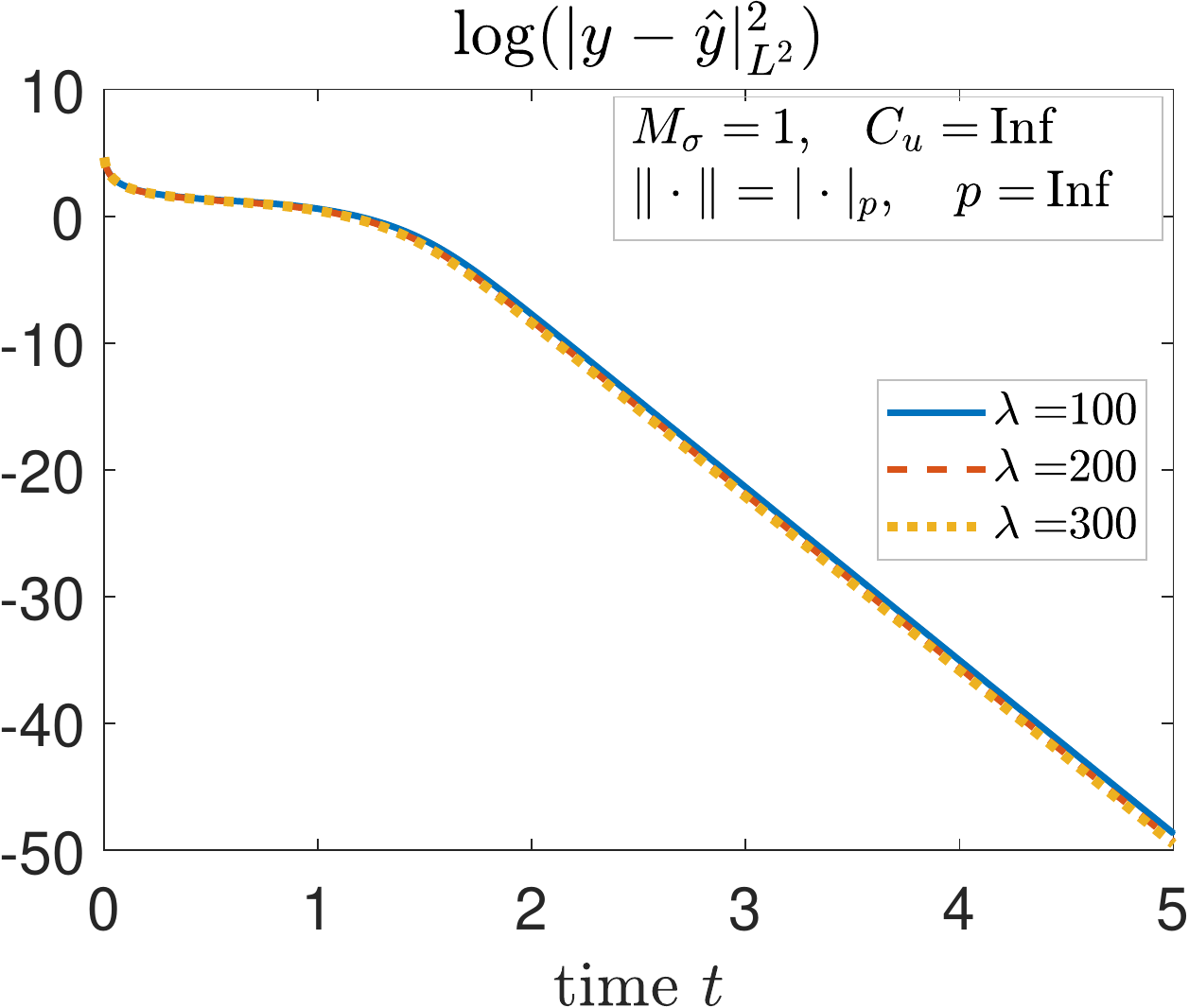}}
\quad
\subfigure
{\includegraphics[width=0.45\textwidth]{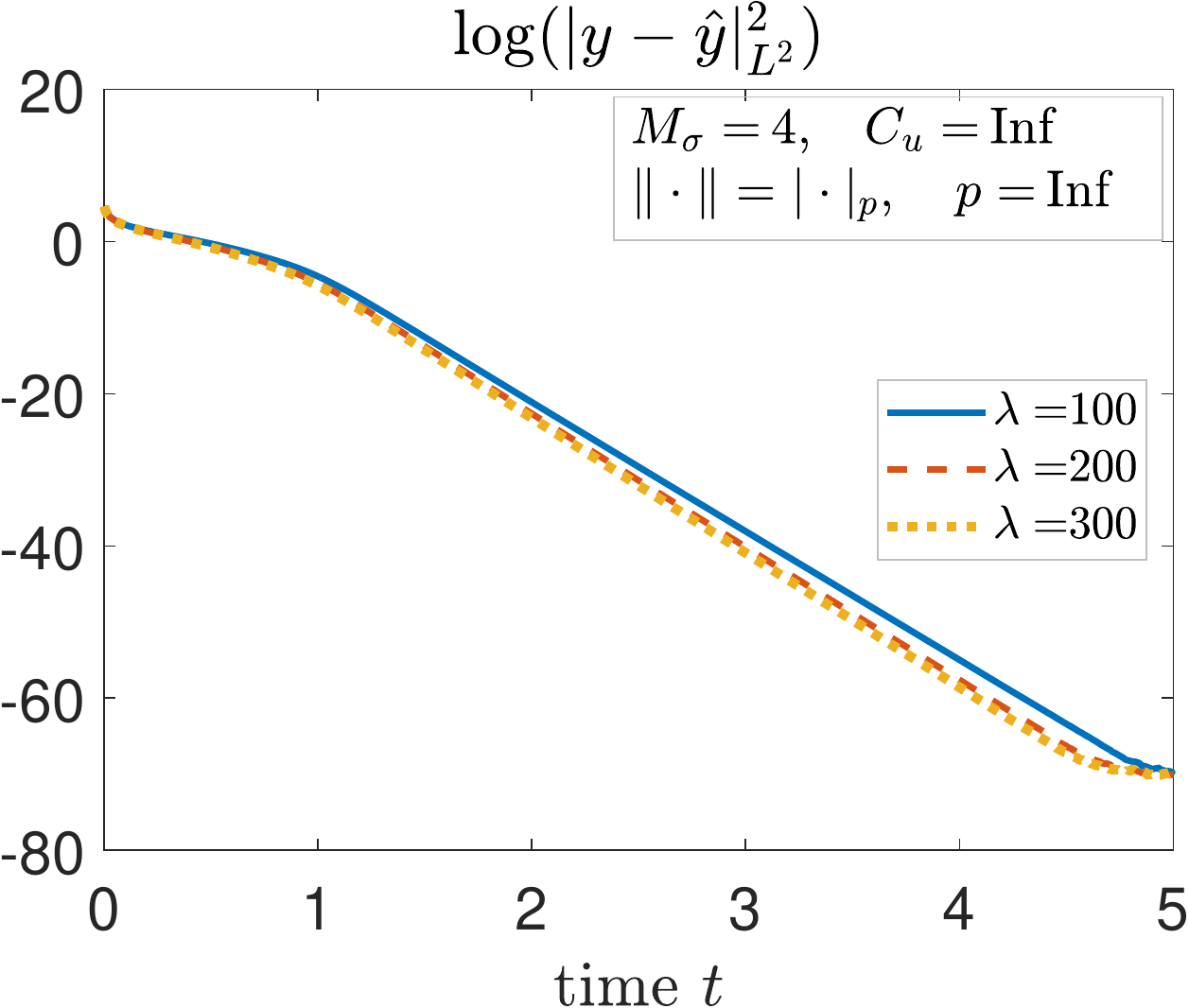}}
\\
\centering
\subfigure
{\includegraphics[width=0.45\textwidth]{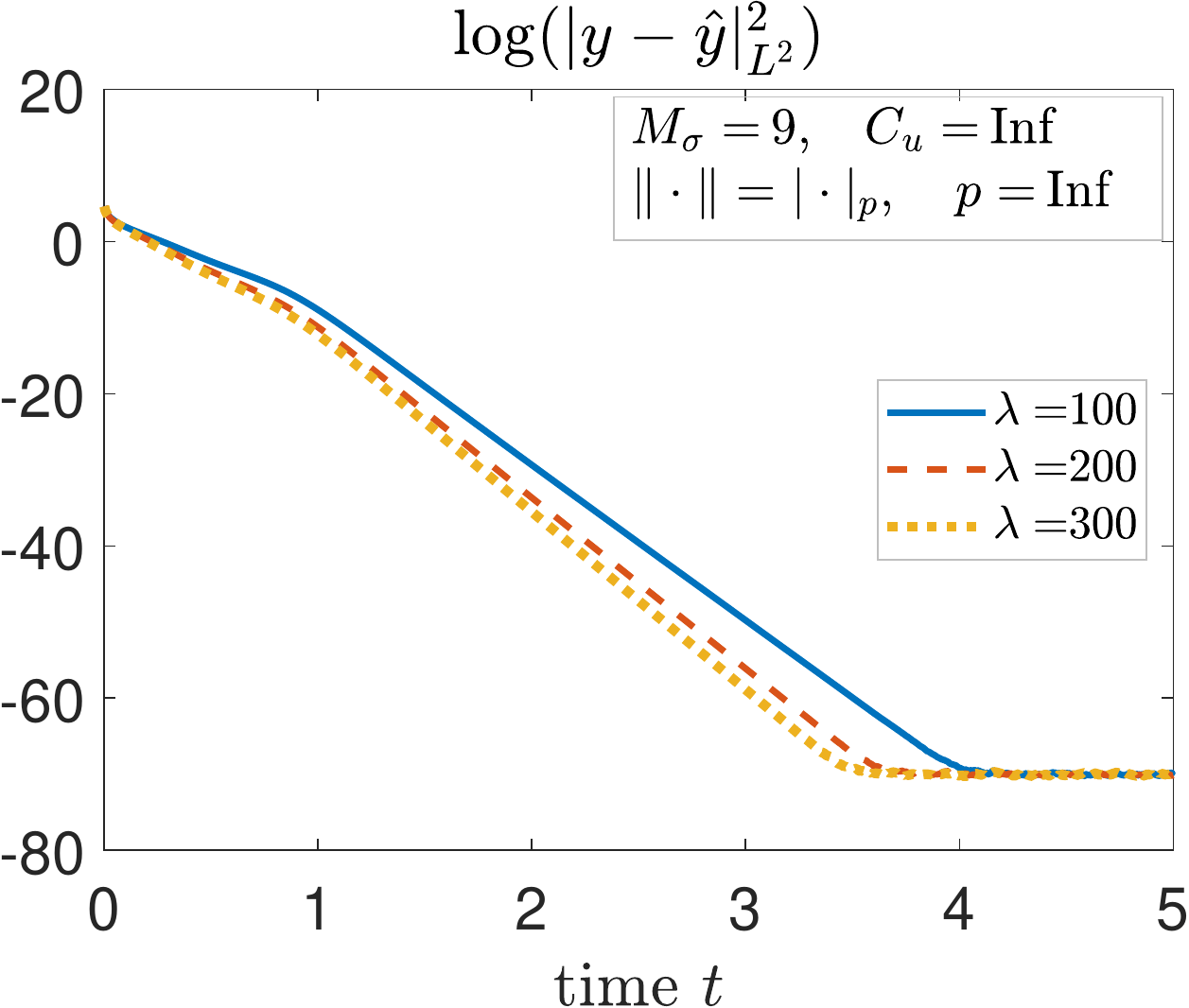}}
\quad
\subfigure
{\includegraphics[width=0.45\textwidth]{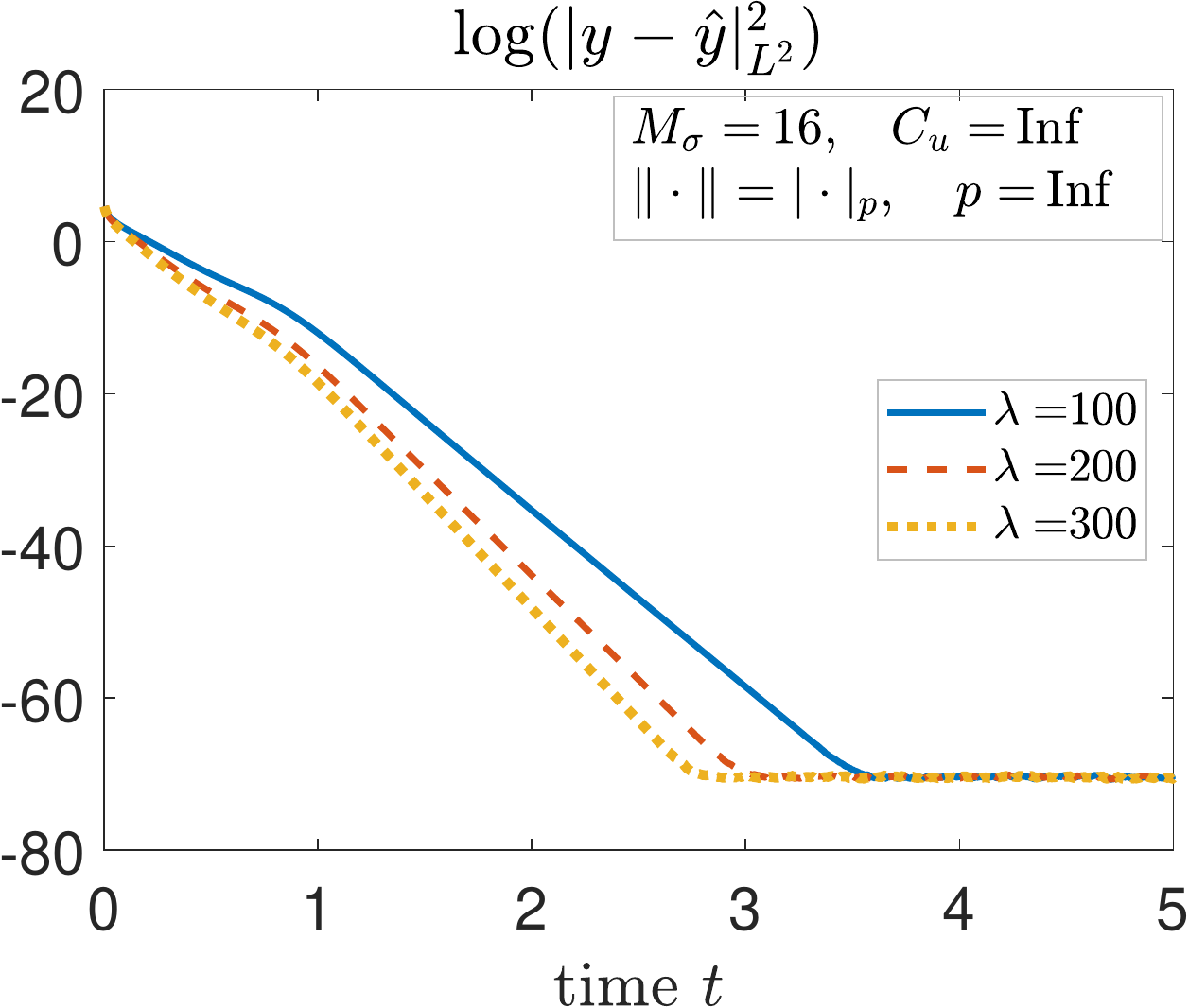}}
\caption{Norm of error for several  pairs~$(M_\sigma,\lambda)$. $C_u=+\infty$. (Ex.~\theexample)}
\label{Fig:zMlamL}
\end{figure}
\end{example}

\begin{remark}
In Fig.~\ref{Fig:zMlamL} we see that with a single actuator, $M_\sigma=1$,
we are able to achieve the exponential stability of the error dynamics, for a suitable rate~$\mu_1$. This does not follow from, nor contradicts, our theoretical results, from which we have that exponential stability with an apriori given rate~$\mu$ holds for large enough $M_\sigma$. This leads us to the following question: (when, if possible) can we guarantee/achieve exponential stability with a single actuator? This could be an interesting problem
for future research.
\end{remark}

Above, we have taken $\lambda\ge100$, which allow us to obtain large exponential decrease rates for the error norm. Fig.~\ref{Fig:zM16lam05} shows that, for ~$M_\sigma=16$ actuators, we can achieve exponential stability of the error dynamics with~$\lambda\ge5$. However, the figure also shows that $\lambda$ cannot be taken arbitrarily small, since the error dynamics is not exponentially stable with~$\lambda\le1$.
\begin{figure}[ht]
\centering
\subfigure
{\includegraphics[width=0.45\textwidth]{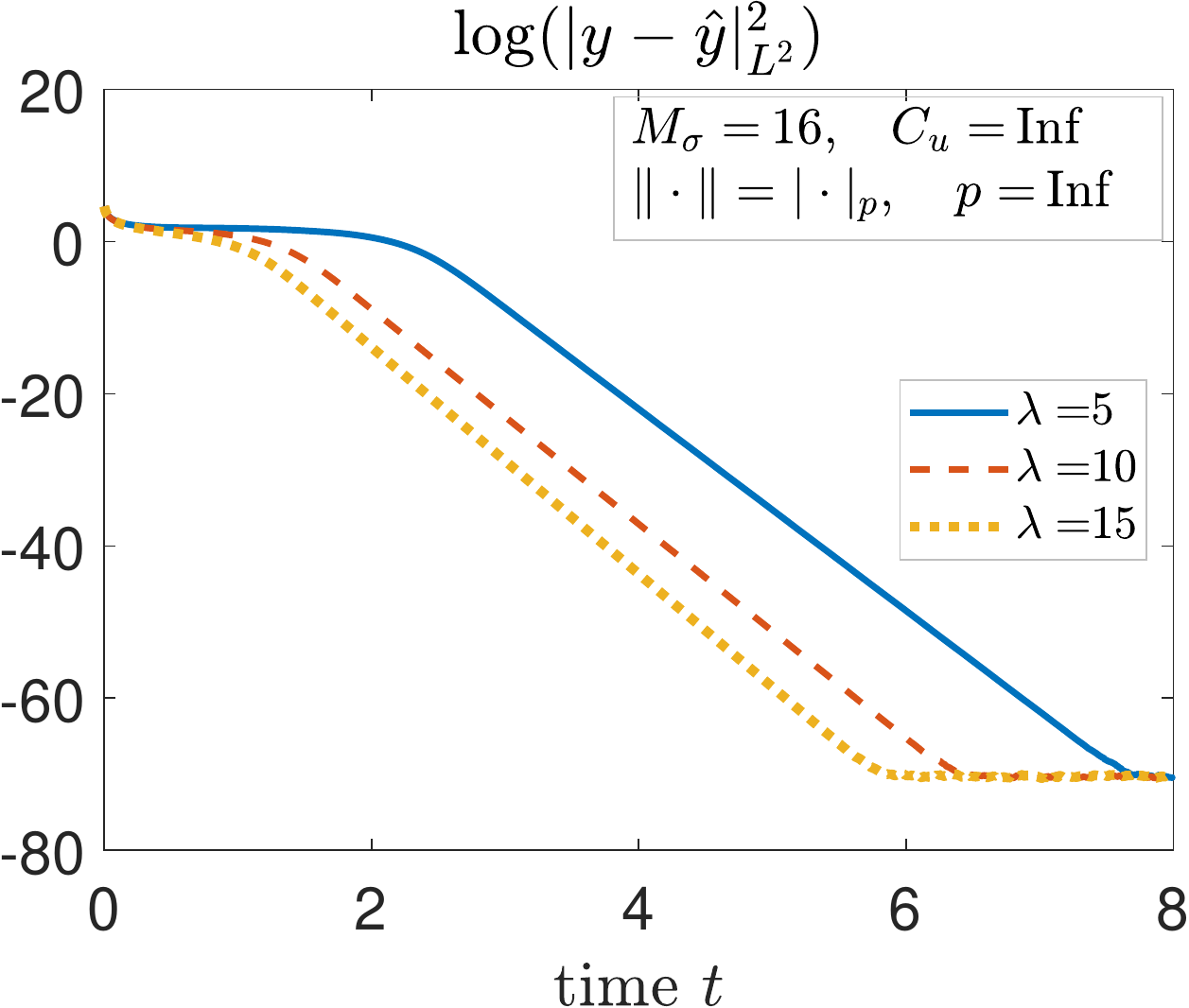}}
\quad
\subfigure
{\includegraphics[width=0.45\textwidth]{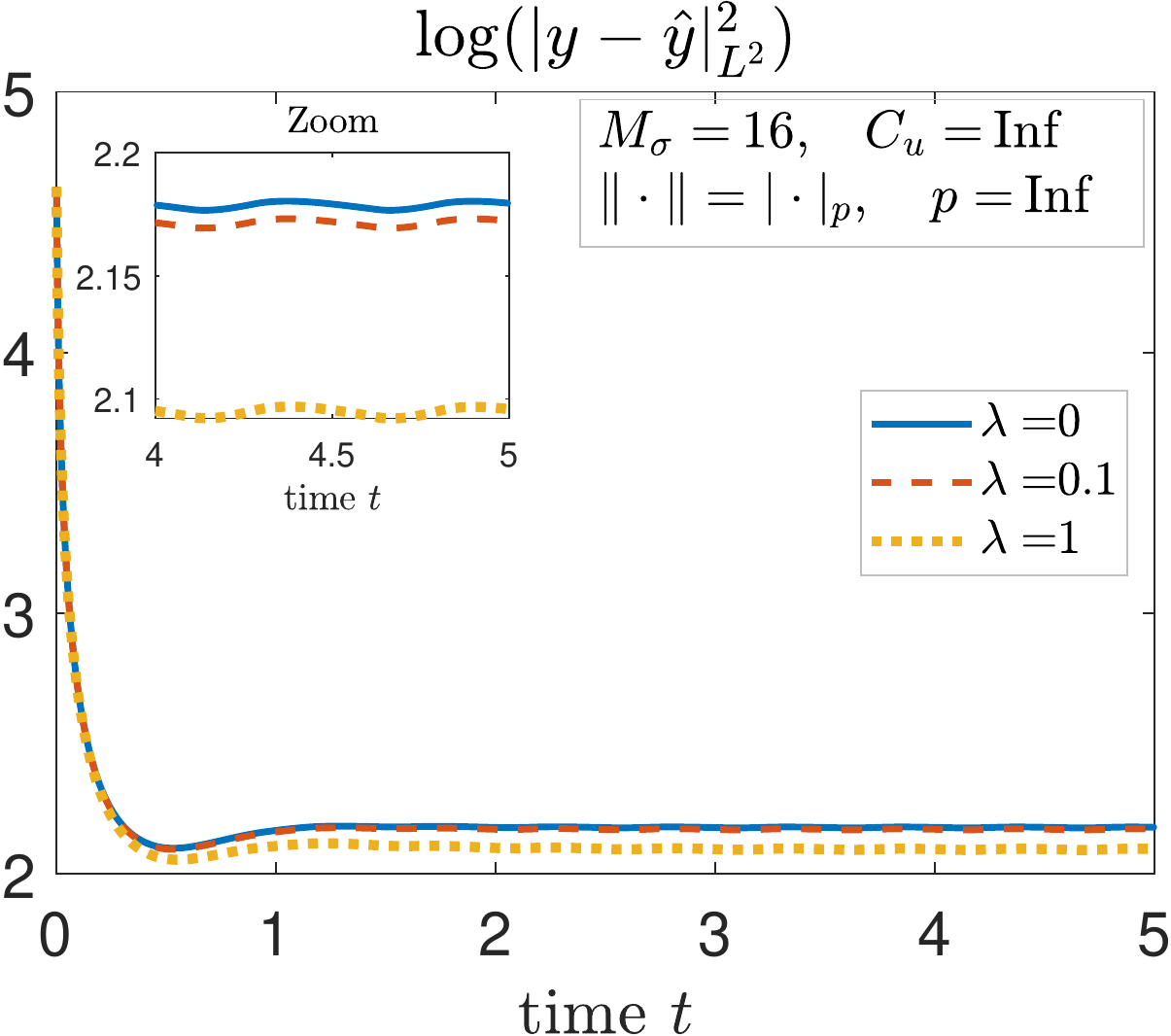}}
\caption{Norm of error for small~$\lambda$. $C_u=+\infty$. (Ex.~\theexample)}
\label{Fig:zM16lam05}
\end{figure}

\section{Concluding remarks}
We have shown the global stabilizability to trajectories for the Schl\"ogl model for chemical reactions, with a finite number of internal actuators and under control magnitude constraints. The number of actuators and the  magnitude of
the controls depend on the diffusion coefficient and on the
nonlinearity, but they are independent of the external forcing and
the targeted trajectory. The stabilizing controls can be taken as the
saturation of an explicit feedback operator. Its  performance  was compared to a receding horizon control approach.

An interesting topic  for  future work could the investigation of systems coupling the Schl\"ogl parabolic equation
with an ODE. These are systems of FitzHugh--Nagumo type
modeling phenomena in neurology and electrophysiology.

\appendix
\section*{Appendix}
\setcounter{section}{1}
\setcounter{theorem}{0} \setcounter{equation}{0}
\numberwithin{equation}{section}

\subsection{Proof of Theorem~\ref{T:ode1-intro-Cuinfty}}\label{Apx:proofT:ode1-intro-Cuinfty}
Let~$\mu>0$. Taking the feedback control $u_1=(r-2\mu)z$ and multiplying the dynamics in~\eqref{sys-y-u-ode1} by~$2z$, we find that~$\tfrac{\rmd}{\rmd t} z^2=-2\mu z^2$, from which we obtain~$z^2(t)=\rme^{-2\mu(t-s)}z^2(s)$.
\hfill\qed

\subsection{Proof of Theorem~\ref{T:ode1-intro}}\label{Apx:proofT:ode1-intro}
Let~$r<0$. Let us fix an arbitrary~$C_u\in\bbR_+$ and an arbitrary control input function~$u=u_1$  satisfying~$\norm{u(t)}{\bbR}\le C_u$, for all~$t\ge0$.
Multiplying the dynamics in~\eqref{sys-y-u-ode1} by~$2z$, we find that
\begin{align}
\tfrac{\rmd}{\rmd t} z^2&=-2rz^2+2u_1z\ge-2rz^2-2C_u\norm{z}{\bbR}\notag\\
&=2\norm{z}{\bbR}(-r\norm{z}{\bbR}-C_u),\notag
\end{align}
which implies that
\begin{align}%
\tfrac{\rmd}{\rmd t} z^2(t)>0\quad\mbox{if}\quad \norm{z(t)}{\bbR}>\tfrac{C_u}{-r}.\notag
\end{align}
Thus,  if~$\norm{z(0)}{\bbR}>\tfrac{C_u}{-r}$ then~$\norm{z(t)}{\bbR}>\norm{z(0)}{\bbR}$ for all $t>0$.
\hfill\qed

\bigskip\noindent
{\bf Acknowledgments:} K. K.  and S. R.  were supported by ERC advanced grant 668998 (OCLOC) under the EU’s H2020 research program. S. R. also acknowledges partial support from Austrian Science Fund (FWF): P 33432-NBL.

\bibliography{ParabSaturCont}
\bibliographystyle{plainurldoi}

\end{document}